\documentclass[a4paper, 11pt]{amsart}
\usepackage{amsmath,amsthm}
\usepackage{amssymb,esint}
\usepackage{amscd}
\usepackage{enumerate}

\usepackage[]{hyperref}
\hypersetup{
    colorlinks=true,       
    linkcolor=black,          
    citecolor=blue,       
    filecolor=magenta,     
    urlcolor=cyan          
}

\newtheorem{theorem}{Theorem}[section]

\newtheorem{corollary}[theorem]{Corollary}
\newtheorem{definition}[theorem]{Definition}

\newtheorem{lemma}[theorem]{Lemma}

\newtheorem{proposition}[theorem]{Proposition}
\newtheorem{remark}[theorem]{Remark}

\DeclareMathOperator{\diff}{d}
\DeclareMathOperator{\pv}{pv}
\DeclareMathOperator{\sign}{sign}

\renewcommand\({\left(}
\renewcommand\){\right)}
\renewcommand\lVert{\left\Vert}
\renewcommand\rVert{\right\Vert}

\def\L#1{\left\langle {#1}\right\rangle}

\def\nc{\newcommand}
\def\be{\beta}

\def\lam{\lambda}

\def\ra{\rightarrow}
\def\la{\leftarrow}

\def\D{\la D\ra}

\nc\pa{\partial}

\nc\CC{\mathbb{C}}
\nc\RR{\mathbb{R}}
\nc\QQ{\mathbb{Q}}
\nc\ZZ{\mathbb{Z}}
\nc\NN{\mathbb{N}}

\def\a{\alpha}
\def\ba{\begin{align}}
\def\bad{\begin{aligned}}
\def\be{\begin{equation}}
\def\ea{\end{align}}
\def\ead{\end{aligned}}
\def\ee{\end{equation}}
\def\e{\eqref}

\def\dalpha{\diff \! \alpha}
\def\ddelta{\diff \! \delta}

\def\dt{\diff \! t}
\def\dtau{\diff \! \tau}
\def\dh{\diff \! h}

\def\ds{\diff \! s}

\def\dx{\diff \! x}
\def\dxi{\diff \! \xi}

\def\dy{\diff \! y}
\def\dz{\diff \! z}
\def\fract{\frac{\diff}{\dt}}

\def\defn{\mathrel{:=}}
\def\eps{\varepsilon}
\def\la{\left\vert}
\def\lA{\left\Vert}
\def\bla{\big\vert}
\def\blA{\big\Vert}

\def\le{\leq}
\def\les{\lesssim}
\def\mez{\frac{1}{2}}
\def\ra{\right\vert}
\def\rA{\right\Vert}

\def\bra{\big\vert}
\def\brA{\big\Vert}
\def\tdm{\frac{3}{2}}

\def\uq{\frac{1}{4}}

\def\xN{\mathbb{N}}
\def\xR{\mathbb{R}}

\begin{document}
\title{Endpoint Sobolev theory for the Muskat equation}
\author{Thomas Alazard}
\author{Quoc-Hung Nguyen}
\date{}

\setlength{\baselineskip}{5mm}

\begin{abstract}
This paper is devoted to the study of solutions with critical 
regularity for the two-dimensional Muskat equation. 
We prove that the Cauchy problem is well-posed on the endpoint Sobolev space of $L^2$ functions 
with three-half derivative in $L^2$. This result 
is optimal with respect to the scaling of the equation. 
One well-known difficulty is that one cannot define 
a flow map such that the lifespan is bounded from below on bounded subsets of this critical Sobolev space. 
To overcome this,  we estimate the solutions for a norm 
which depends on the initial data themselves, using 
the weighted fractional Laplacians introduced in our previous works. 
Our proof is the first in which a null-type structure is identified for 
the Muskat equation, allowing 
to compensate for the degeneracy of the parabolic behavior for large slopes.

\end{abstract}

\maketitle

\section{Introduction}
This is the main in a sequence of three papers on 
the Cauchy problem for the Muskat equation for solutions with critical regularity. 
The goal of the present paper is to obtain a definitive result, that is to solve the Cauchy problem 
in the endpoint Sobolev space 
$$
H^\tdm(\xR)=\left\{ f\in L^2(\xR)\,;\, \int_{\xR}\big(1+\la\xi\ra)^{3}\bla\hat{f}(\xi)\bra^2\dxi<+\infty\right\}\cdot
$$
The study of the Cauchy problem for the Muskat equation begun 
two decades ago, but with many roots in earlier works. 
Inspired by the analysis of 
free boundary flows, different approaches succeeded to establish 
local well-posedness results for smooth enough initial data (see~\cite{Yi2003,SCH2004,Ambrose-2004,Ambrose-2007,CG-CMP,CCG-Annals}). 
In the last several years, this problem was extensively studied. There are now many 
different proofs that the  
Cauchy problem is well-posed, locally in time, 
on the sub-critical Sobolev spaces $H^{s}(\xR)$ with 
$s>3/2$ (see~\cite{Cheng-Belinchon-Shkoller-AdvMath,CGSV-AIHP2017,DLL,Matioc2,Alazard-Lazar,Nguyen-Pausader}). 

Another important component of the background is 
that the Muskat equation has a compact formulation in terms of singular integrals 
(see Caflisch, Howison and Siegel~\cite{SCH2004}, C\'ordoba and Gancedo~\cite{CG-CMP}). 
Besides its esthetic aspect, it allows 
to apply tools at interface of harmonic analysis and nonlinear partial differential equations. 
One can think of the circle of methods centering around the study of the Hilbert transform and Riesz potentials, 
or Besov and Triebel-Lizorkin spaces, which can be said to be at the root of the present analysis. 
In  this direction, we are strongly influenced by many 
earlier works about global existence results in critical spaces under some smallness assumptions 
(see in particular~\cite{CCGRPS-JEMS2013,CCGRPS-AJM2016,Cameron,Cordoba-Lazar-H3/2,Gancedo-Lazar-H2}). 

Building upon the above mentioned papers, we develop here 
an approach which allows to resolve several difficult questions. In particular, 
the analysis presented here 
if the first in which a \textbf{null-type structure} is exhibited for the Muskat equation. 
This provides a substantial clarification of the structure of the nonlinearity, which 
is the key to compensate for 
the degeneracy of the parabolic behavior for large slopes. 

The results proven in this article 
are of two types: global existence results valid under a smallness assumption, or
local well-posedness results for arbitrary initial data. For our subject, the last are the most important.
To explain this, we first of all recall that 
Castro, C\'{o}rdoba, Fefferman, Gancedo and L\'{o}pez-Fern\'{a}ndez proved that the 
Cauchy problem is not well posed globally in time for some large data~(\cite{CCFG-ARMA-2013, CCFG-ARMA-2016, CCFGLF-Annals-2012}).
A profound consequence of this fact is that one cannot
define a flow map such that the lifespan is bounded from below on bounded subsets of $H^\tdm(\xR)$  
(otherwise, we would obtain a global existence result for any initial data in $H^\tdm(\xR)$ 
by a scaling argument; using that if $f=f(t,x)$ solves the Muskat equation, so does the function 
$\lambda^{-1}f(\lambda t,\lambda x)$).
This is one of the reasons why it is
rather surprising that the Cauchy problem is well-posed 
for \textbf{large data} on the endpoint Sobolev space $H^\tdm(\xR)$. 
The same situation appears for various dispersive equations, however here 
the problem is considerably more difficult since the Muskat equation 
is fully-nonlinear. 

The discussion in the previous paragraph shows that one cannot rely on  
classical Sobolev energy estimates; since the time of existence must depend 
on the profile of the initial data and not only on its 
Sobolev norm. The problem therefore arises of estimating the solutions of the Muskat equation
with a norm whose definition involves the initial data themselves.
In fact, it is already a challenging question to define these norms.
For that, we have introduced in~\cite{AN1} some \textbf{weighted fractional Laplacians} 
which are tailored to the fractional characteristic of the Muskat equation, following previous works in~\cite{BN18a,BN18b,BN18c,BN18d,Ng}. 
We will review the main properties of these operators in~\S\ref{S:wRp}.
We will push this idea further by introducing some time-dependent frequency decompositions
which also depend on the solution.

The study of the well-posedness of the Cauchy problem 
for the Muskat equation on the endpoint Sobolev space $H^\tdm(\xR)$ 
faces two other serious obstacles. 
Firstly, an obvious difficulty is that $H^{\tdm}(\xR)$ is not embedded in the space of Lipschitz functions. 
As a consequence, the spatial derivative is not controlled in $L^\infty$-norm. 
This raises interesting new difficulties since $L^\infty$-norms 
are ubiquitous in classical nonlinear estimates. 
Secondly, the Muskat equation is a \textbf{degenerate parabolic equation} for solutions which 
are not controlled in Lipschitz norm (this statement will be given a rigorous meaning in \S\ref{S:pMe}, 
see Remark~\ref{R:dpb}). As mentioned above, the only way to overcome this difficulty is 
to identify a null-type structure in the nonlinearity. 
In addition, one has to perform a paradifferential-type analysis 
which is compatible with both the weighted fractional Laplacians 
and this null-type structure. This is the main problem whose solution is the subject of this article.

We will come back to these three points in \S\ref{S:1.2} after having 
introduced some notations and stated our main result. 
In~\S\ref{S:1.2}, we will also compare the main results of this article with our previous ones.

\subsection{Main result}
The Muskat equation dictates the dynamics  
of the interface $\Sigma$ separating two fluids having 
different densities $\rho_1$ and~$\rho_2$, obeying Darcy's law and evolving under the restoring force of gravity. 
Hereafter we assume that $\rho_2>\rho_1$ (this corresponds 
to the stable physical case where the heaviest fluid is located underneath the free surface) and 
we further assume that $\rho_2-\rho_1=2$ (without loss of generality).

On the supposition that the fluids are two-dimensional 
and that the interface is a graph, $\Sigma$ is of the form
$$
\Sigma(t)=\left\{ (x,y)\in\xR\times\xR\,;\, y=f(t,x)\right\} \qquad (t\ge 0),
$$
where the free surface elevation $f$ is an unknown. 
In the C\'ordoba-Gancedo formulation~\cite{CG-CMP}, 
the evolution equation for $f$ simplifies to
\be\label{a1}
\partial_tf=\frac{1}{\pi}
\pv\int_\xR\frac{\partial_x\Delta_\alpha f}{1+\left(\Delta_\alpha f\right)^2}\dalpha,
\ee
where the integral is understood as a principal value 
integral\footnote{Since it is not immediately obvious that the right-hand 
side is well defined, 
we refer to the discussion in paragraph~\S\ref{S:3.4} 
where we recall that the Muskat equation can be written under the form 
$\partial_t f+\D f=\mathcal{T}(f)f$ where $\mathcal{T}(f)f$ is given 
by a well-defined integral.}, 
and $\Delta_\alpha f$ is a slope, defined by the quotient
\begin{align}\label{eq2.2}
\Delta_\alpha f(t,x)=\frac{f(t,x)-f(t,x-\alpha)}{\alpha}\cdot
\end{align}

Recall that the following notations for Sobolev and homogeneous Sobolev norms:
\begin{equation}\label{defi:Sobolev}
\begin{aligned}
\lVert u \rVert_{\dot{H}^{\sigma}}^{2} &\defn (2\pi)^{-1} 
\int_{\xR} \la\xi\ra^{2\sigma} |\hat{u}(\xi)|^2\dxi,\\
\lVert u \rVert_{H^\sigma}^{2} &\defn (2\pi)^{-1} 
\int_{\xR} (1+|\xi|^2)^{\sigma} |\hat{u}(\xi)|^2\dxi,
\end{aligned}
\end{equation}
where~$\widehat{u}$ is the Fourier transform of~$u$.

\begin{theorem}\label{Theorem}
$i)$ For any initial data $f_0$ in $H^{\tdm}(\xR)$, there exists a time $T>0$ such that 
the Cauchy problem for the Muskat equation~\e{a1} has a unique solution $f$ in the space
\begin{equation*}
X^\tdm(T)=\left\{ f\in C^0\big([0,T];  H^{\tdm}(\xR)\big)\,;\, 
\int_{0}^{T}\!\!\int_\xR \frac{(\partial_{xx}f)^2}{1+(\partial_xf)^2} \dx\dt<\infty\right\}.
\end{equation*}
$ii)$ Moreover, there exists a positive constant $\varepsilon_0$ such that, 
for any initial data $f_0$ in $H^{\frac32}(\xR)$ satisfying
\begin{equation*}
\lA f_0\rA_{\dot H^{\tdm}}\leq \varepsilon_0,
\end{equation*}
the Cauchy problem for~\e{a1} has a unique global solution 
$f$ in $X^\tdm(+\infty)$. 
\end{theorem}

In the next paragraphs, 
we shall compare this result with previous ones for other critical problems. We shall also 
describe the general strategy of proof of Theorem~\ref{Theorem}, the difficulties one has to cope with, 
and the ideas used to overcome them. 

\subsection{Comparison with previous works}\label{S:1.2'}
The study of the well-posedness of the Cauchy problem 
for various partial differential equations 
has attracted a lot of attention in the last decades. 

For the Schr\"odinger equation, which 
is the prototypical example of a semi-linear dispersive 
equation, the study of Cauchy problem 
in the  energy critical case goes back to the works of Cazenave and Weissler~\cite{Cazanave-Weissler-1989,Cazanave-Weissler-1990} 
and culminates with the global existence results of 
Bourgain~\cite{Bourgain-1999}, Grillakis~\cite{Grillakis-2000} and 
Colliander, Keel, Staffilani, Takaoka and 
Tao~\cite{CKSTT-AoM}. By contrast with sub-critical problems, 
the time of existence given by the local theory in~\cite{Cazanave-Weissler-1989,Cazanave-Weissler-1990} 
depends on the profile of the data and not only on its energy. 
Consequently, the conservation of the energy is not sufficient to obtain global existence results. 
Detailed historical accounts of the 
subject can be found in~\cite{Tao-Book}. 
We also refer to the recent paper by Merle, Raph{\"a}el, 
Rodnianski and Szeftel~\cite{MRRS-2019NLS} which establishes 
an unexpected blow-up result for supercritical defocusing nonlinear Schr\"odinger equations. 
If instead of a semi-linear equation, one considers a quasi-linear problem, then 
the scaling is not necessarily the only relevant criteria. 
One key result in this direction is the resolution of the  bounded $L^2$ curvature conjecture by Klainerman, Rodnianski and Szeftel~\cite{KRS-invent}.

As for hyperbolic or dispersive equations, 
the study of critical problems for parabolic equations has a very rich litterature. 
Consider the equation
\begin{align*}
\partial_t\theta+u\cdot\nabla \theta+(-\Delta)^{\frac{\alpha}{2}}\theta=0\quad \text{with}
\quad u=\nabla^\perp(-\Delta)^{-\mez}\theta,
\end{align*}
which is a dissipative version of the classical surface quasi-geostrophic equation introduced by Constantin-Majda-Tabak~\cite{CMT-1994}. 
In the critical case $\alpha=1$, the global in time well-posedness 
has been proved by Kiselev-Nazarov-Volberg \cite{KNV-2007}, Caffarelli-Vasseur~\cite{Caffarelli-Vasseur-AoM} and 
Constantin-Vicol~\cite{CV-2012} (see also~\cite{KN-2009,Silvestre-2012,VaVi,NgYa}). 
For the Muskat equation, 
the nonlinear term is considerably more complicated. However, 
Cameron proved in~\cite{Cameron} that one can apply 
the method introduced by Kiselev-Nazarov-Volberg to prove 
the existence of global in time solutions 
whenever the product of the maximal and minimal slopes is less than $1$. 

For the Muskat equation, the study of the critical Sobolev problem was initiated by 
C\'ordoba and Lazar in~\cite{Cordoba-Lazar-H3/2}. They proved 
a global existence result 
for initial data $f_0$ in $H^{\frac52}(\xR)$ such that 
$(1+\lA \partial_x f_0\rA_{L^\infty}^4)\lA f_0\rA_{\dot{H}^{\tdm}}\ll 1$ (see also \cite{Gancedo-Lazar-H2,AN2}). 
Their main technical result is a key {\em a priori} estimate 
in the critical space $\dot{H}^\tdm(\xR)$, which reads:
\be\label{CL}
\fract \lA f\rA_{\dot H^{\frac{3}{2}}}^2
+ \int_\xR \frac{( \partial_{xx}f)^2}{1+(\partial_x f)^2} \dx\lesssim  \(\lA f\rA_{\dot H^{\frac{3}{2}}}
+\lA f\rA_{\dot H^{\frac{3}{2}}}^2\)
\lA f\rA_{\dot{H}^2}^2.
\ee
This estimate is easy to state, however its proof is very delicate. C\'ordoba and Lazar 
introduce in~\cite{Cordoba-Lazar-H3/2} 
a tricky reformulation of the Muskat equation in terms of oscillatory integrals, 
together with the systematic use of Besov spaces. 

For the sake of comparison, it might be useful for the reader to explain how 
to get a global existence result out of the inequality~\e{CL}. 
Here the idea is quite simple: one wants to absorb the right-hand side 
by the left-hand side. This requires 
to control from below the denominator $1+(\partial_xf)^2$. 
For that purpose, one uses an integration by parts argument that 
goes back to the work of C\'ordoba and Gancedo~\cite{CG-CMP2} 
(see also~\cite{Cameron,Gancedo-Lazar-H2}). 
This gives an estimate of the form (see~\cite{AN2})
$$
\fract \lA \partial_x f\rA_{L^\infty}\les \lA f\rA_{H^2}^2.
$$
By combining the latter inequality with \e{CL}, one obtains that 
there exists $c_0$ small enough, so that
\be\label{CL3}
\fract \lA f\rA_{\dot H^{\frac{3}{2}}}^2
+ c_0\int_\xR \frac{( \partial_{xx}f)^2}{1+(\partial_x f)^2} \dx\le 0,
\ee
which gives the wanted control of the $L^\infty_t(H^{3/2}_x)$-norm.

Eventually, we discuss the main result in our previous paper~\cite{AN2}, where 
we proved that the Cauchy problem is well-posed 
in the space $W^{1,\infty}(\xR)\cap H^\tdm(\xR)$ for arbitrary large data (together with 
a global in time regularity result under a smallness assumption in $\dot{H}^{\tdm}(\xR)$). 
The key difference between the former large data result and 
other known results for other critical parabolic problems 
is that we know that, for the Muskat equation, 
there exist some solutions with large data which do not exist globally in time 
(as proved by Castro, C\'{o}rdoba, Fefferman, 
Gancedo and L\'opez-Fern\'andez~(\cite{CCFG-ARMA-2013,CCFG-ARMA-2016,CCFGLF-Annals-2012}).  
Consequently, one cannot solve the 
Cauchy problem for a time $T$ which depends only on 
the norm of $f_0$ in $\dot W^{1,\infty}(\xR)\cap \dot H^{3/2}(\xR)$. 
Otherwise, this would yield a global existence result for any initial data 
by a scaling argument. 

As already said, the study of critical problems is a very vast subject of which we have 
barely scratched the surface in this paragraph. There are many related results. 
Let us mention papers by Vazquez~\cite{Vazquez-DCDS}, for 
the fractional porous media equation, as well as 
Granero-Belinch{\'o}n and Scrobogna~\cite{Granero-Scrobogna} and  Scrobogna~\cite{Scrobogna-HS} 
for quadratic approximation of the Hele-Shaw equation (with or without surface tension). 
We also refer the reader to 
Abedin and Schwab~\cite{Abedin-Schwab-2020}, who obtained recently 
the existence of a modulus of continuity for the Muskat equation 
by means of Krylov-Safonov estimates. 
Let us also mention that the existence and (possible) non-uniqueness 
of some weak-solutions has also been thoroughly studied 
(see~\cite{Brenier2009,cordoba2011lack,szekelyhidi2012relaxation,castro2016mixing,forster2018piecewise,noisette2020mixing}).

\subsection{Strategy of the proof}\label{S:1.2}
A key difficulty is that 
the coercive quantity which appears in the left-hand side of the C\'ordoba-Lazar's critical estimate~\e{CL}, that is
\be\label{co}
\int_0^T\int_\xR \frac{( \partial_{xx}f)^2}{1+(\partial_x f)^2} \dx\dt,
\ee
is insufficient to control the $\lA \cdot\rA_{L^2_t\dot{H}_x^2}$-norm of $f$. 
In sharp contrast with the previous works in~\cite{Cameron,Cordoba-Lazar-H3/2,Gancedo-Lazar-H2,AN2}, 
in this paper we make no extra assumption on the initial data which would allow us to control 
from above the denominator~$1+(\partial_xf)^2$. 

By comparison with the strategy recalled in the previous paragraph, 
this leads to a serious difficulty since the parabolic feature degenerates in this case. 
Thus a major ingredient in the proof of Theorem~\ref{Theorem}, 
and another object of the paper, is to prove that it is sufficient to control 
the quantity~\e{co}. To do so, we will uncover some key cancellations 
showing that one can divide the most singular terms by the possibly large factor 
$1+(\partial_x f)^2$. This is the 
null-type structure for the Muskat equation which will be systematically 
studied in the following.

In the rest of this introduction, we will explain how one can overcome 
the two main difficulties mentioned above; namely, the fact that the time 
of existence must depend on the initial data and the 
degeneracy in the parabolic behavior.

\subsubsection*{A weak null-type property}
Our first idea will consist in proving a (weak) null-type property. 
By this we mean the proof of an inequality of the form~\e{CL}
but with the 
key difference that the right-hand side does not involve the $\lA \cdot\rA_{\dot{H}^2}$-norm of~$f$, 
but the 
weaker coercive quantity given by~\e{co}. 

More precisely, we will prove that
\be\label{i10}
\fract \lA f\rA_{\dot H^{\frac{3}{2}}}^2
+ \int_\xR \frac{( \partial_{xx}f)^2}{1+(\partial_x f)^2} \dx\lesssim
\left(1+\Vert f\Vert_{\dot H^{\frac{3}{2}}}^{7}\right)
\Vert f\Vert_{\dot H^{\frac{3}{2}}} \int_\xR \frac{( \partial_{xx}f)^2}{1+(\partial_x f)^2} \dx.
\ee
This estimate will imply at once statement $ii)$ in Theorem~\ref{Theorem}, about 
the global existence result under a smallness assumption. Indeed, this part is 
perturbative in character, and one can absorb the right-hand side of \e{i10} by the left-hand side, 
provided that $\Vert f\Vert_{\dot H^{\frac{3}{2}}}$ is small enough. 

However, the estimate \e{i10} 
is insufficient to prove the main result, which is statement~$i)$ in Theorem~\ref{Theorem}.  
Indeed, as we have explained before, for the latter purpose, 
one cannot rely entirely on an estimate of the critical $\lA \cdot\rA_{\dot{H}^{\tdm}}$-norm. 
This will force us to prove a strong null-type property, whose proof is considerably more involved.

\subsubsection*{Weighted fractional Laplacians} Before going any further, we need 
to define the norm that we are going to estimate. To do so, we pause to 
introduce a weighted version of the classical fractional Laplacian.

Given a real number $s$ and a function $\phi\colon[0,+\infty)\to[1,+\infty)$ 
satisfying a doubling condition (this means that 
there is $c_0>0$ such that $\phi(2r)\leq c_0\phi(r)$ for any $r\geq 0$), 
we introduce the weighted fractional Laplacian $|D|^{s,\phi}$ 
as 
the Fourier multiplier with symbol $|\xi|^s\phi(|\xi|)$. More precisely,
\begin{equation*}
\mathcal{F}( |D|^{s,\phi}f)(\xi)=|\xi|^s\phi(|\xi|) \mathcal{F}(f)(\xi).
\end{equation*}
Moreover, we define the space
\be\label{defi:space}
\mathcal{H}^{s,\phi}(\xR)=\{ f\in L^2(\xR)\, :\, \D^{s,\phi}f\in L^2(\xR)\}.
\ee
A weighted fractional Laplacian $\D^{s,\phi}$, as the name indicates, attempts to be 
a slight modulation of the usual fractional Laplacian $\D^s=(-\partial_{xx})^{s/2}$.  
We will consider special 
weight functions $\phi$ depending 
on some extra functions $\kappa\colon [0,\infty)\to [1,\infty)$, of the form
\begin{equation}\label{ni10}
\phi(\lam)=\int_{0}^{\infty}\frac{1-\cos(h)}{h^2} \kappa\left(\frac{\lam}{h}\right) \dh, \quad \text{for }\lambda\ge 0.
\end{equation}
Let us clarify the assumptions on $\kappa$ which will be needed. 
We say that a function $\kappa\colon[0,+\infty)\to[1,+\infty)$ is 
an admissible weight if it satisfies the following conditions:
\begin{enumerate}[$({{\rm H}}1)$]
\item\label{iH1} $\kappa$ is increasing;
\item\label{iH2} 
there is a positive constant $c_0$ such that $\kappa(2r)\leq c_0\kappa(r)$ for any $r\geq 0$; 
\item\label{iH3} the function $r\mapsto \kappa(r)/\log(4+r)$ is decreasing on  $[0,\infty)$.
\end{enumerate}
The operator $\D^{s,\phi}$ have been introduced 
in~\cite{BN18a,BN18b,BN18d,Ng} in the special case $\phi(r)=(\log(2+r))^a$ 
for some $a\in [0,1]$. The general case has been introduced in~\cite{AN1}. 

In our first paper~\cite{AN1}, we studied the Muskat equation 
in the spaces $\mathcal{H}^{s,\log(\cdot+|\xi|)^a}(\xR)$ for some $a\ge \frac{1}{3}$, 
which permitted to solve the Cauchy problem for non-Lipschitz initial data, following earlier works 
by Deng, Lei and Lin~\cite{DLL}, Cameron~\cite{Cameron}, C\'ordoba and Lazar~\cite{Cordoba-Lazar-H3/2}, 
Gancedo and Lazar~\cite{Gancedo-Lazar-H2} which allowed arbitrary large slopes. 
The main difficulty in~\cite{AN1} 
is that the assumption of finite slopes is difficult to dispense of. 
Indeed, classical nonlinear estimates 
require to control the $L^\infty$-norm of 
some factors, which amounts for the Muskat problem 
to control the $L^\infty$-norm of the slopes~$\Delta_\alpha f$, equivalent to control 
the Lipschitz norm of~$f$. Loosely speaking, the assumption $a\ge \frac13$ might be 
understood as the minimal assumption 
in terms of weighted fractional Laplacian to overcome the degeneracy in the parabolic 
feature by an interpolation argument. 
In this paper, we will perform a much more difficult analysis 
which will allow us to remove completely the assumption $a\ge\frac13$.

\subsubsection*{Weighted Sobolev estimates} 
As already explained, for large data, the time of existence must depend 
on the initial data themselves and not only on its critical 
Sobolev norm. To overcome this problem, 
our strategy is to estimate the solution for a stronger norm whose definition involves the initial 
data themselves. To do so, we use the following elementary result (which follows from 
Lemma~$3.8$ in~\cite{AN1}).
\begin{lemma}\label{L:enhanced}
For all $f_0$ in $H^\tdm(\xR)$, there exists an admissible weight $\kappa$ satisfying 
$\lim_{r\to+\infty}\kappa(r)=\infty$ and such that 
$f_0$ belongs to $\mathcal{H}^{\tdm,\phi}(\xR)$ where $\phi$ is given by~\e{ni10}.
\end{lemma}

Then, we are in position to state an improvement of statement $i)$ 
in Theorem~\ref{Theorem} 
which asserts that, 
whenever one controls a bigger norm than 
the critical one, 
the time of existence depends only on the norm. 
Recall that $\mathcal{H}^{\tdm,\phi}(\xR)$ is defined by~\e{defi:space}.

\begin{theorem}\label{Theorem2}
Consider a real number $M_0>0$ and a 
function $\phi$ given by \e{ni10} for some admissible weight 
satisfying assumptions~$(\rm{H}\ref{iH1})$--$(\rm{H}\ref{iH3})$ 
with $\lim_{r\to+\infty}\kappa(r)=\infty$. 
Then there exists a time $T_0>0$ depending on $M_0,\kappa$ such that, 
for any initial data $f_0$ in $\mathcal{H}^{\tdm,\phi}(\xR)$ 
satisfying $\lA f_0\rA_{\mathcal{H}^{\tdm,\phi}}\le M_0$, 
the Cauchy problem for~\e{a1} has a unique solution in the space
\begin{equation*}
\mathcal{X}^{\tdm,\phi}(T_0)\defn \left\{ f\in C^0\big([0,T_0];\mathcal{H}^{\tdm,\phi}(\xR)\big)\,;\,
\int_{0}^{T_0}\!\!\int_\xR \frac{\bla \D^{2,\phi}f\bra^2}{1+(\partial_xf)^2} \dx\dt<\infty\right\}.
\end{equation*}
\end{theorem}
Notice that statement $i)$ in Theorem~\ref{Theorem} follows at once 
from Lemma~\ref{L:enhanced} 
and the above result.

\subsubsection*{A strong null-type property}
To prove Theorem~\ref{Theorem2}, we shall prove an estimate in weight 
Sobolev space $\mathcal{H}^{\tdm,\phi}(\xR)$ which extends~\e{i10} in several direction. 
Namely, 
given an initial data $f_0$ and an admissible weight $\phi$ such that $f_0$ 
belongs to $\mathcal{H}^{\tdm,\phi}(\xR)$, 
we shall prove that 
the following norms
\begin{equation*}
A(t)\defn \lA f(t)\rA_{L^2}^2+\blA \D^{\tdm,\phi}f(t)\brA_{L^2}^2,\quad 
B(t)\defn\int_\xR \frac{\bla \D^{2,\phi}f(t,x)\bra^2}{1+(\partial_xf(t,x))^2} \dx
\end{equation*}
satisfy
\begin{equation}\label{e:A-B}
\fract A
+ B\leq C_0 \left(1+A^{4}\right)\left(\log(4+B)^{\frac{1}{2}}B^{\frac{3}{4}}+\frac{B}{\kappa(B^{\mez})}+ \frac{B}{\kappa(B^{\frac{1}{50}})}\right)+\mathcal{G}\left(A\right),
\end{equation}
for some positive constant $C_0$ and some non-decreasing function $\mathcal{G}\colon\xR_+\to\xR_+$. 

The proof of~\e{e:A-B} relies on several new ideas. 
Firstly we shall systematically decompose the integrals in $\alpha$ into two parts: 
short scales $|\alpha|\le \lambda$ and large scales $|\alpha|\ge \lambda$. 
This is of course a classical idea, 
but here we are using it in a novel way tailored to our subject. 
Indeed, we will allow the frequency threshold to be a time dependent function 
(in the end, we select $\lambda=\lambda(t)=B(t)^\mez$). 
Also, the reason to introduce this decomposition is that, for any admissible weight $\kappa$ and any 
$\beta$ in $(0,1]$, there is a positive constant $C_\beta$ such that, for all $\lambda>0$ and all $\xi\in \xR$, 
$$
\min\{\lam|\xi|,1\}^\beta\le C_\beta \frac{\kappa(|\xi|)}{\kappa(1/\lambda)}\cdot
$$
Secondly, we shall uncover in the nonlinearity a special structure which allows 
to divide the most singular terms by the possibly large factor 
$1+(\partial_x f)^2$. To give this vague sentence a rigorous meaning, let us 
introduce the operator $\mathcal{T}(f)$ defined by
$$
\mathcal{T}(f)g = -\frac{1}{\pi}\int_\xR\left(\partial_x\Delta_\alpha g\right)
\frac{\left(\Delta_\alpha f\right)^2}{1+\left(\Delta_\alpha f\right)^2}\dalpha.
$$
This operator describes the nonlinearity in the Muskat equation since the latter can be written under the form (see Section~\ref{S:pMe})
\begin{align*}
\partial_tf+\D f = \mathcal{T}(f)f.
\end{align*}
We shall prove the following key estimate (see Proposition~\ref{P:5.7}),
\begin{align*}
&\la\left\langle \big[\D^{1,\phi},\mathcal{T}(f)\big](g),h\right\rangle\ra\\
&\quad \lesssim
\left( 1+\blA\D^{\frac{3}{2},\phi} f\brA_{L^2}^{5}\right)  \left(\brA\D^{\frac{3}{2},\phi}g\brA_{L^2}
\lA\frac{f_{xx}}{\L{f_x}}\rA_{L^2}
+\brA\D^{\frac{3}{2},\phi}f\brA_{L^2}\lA\frac{g_{xx}}{\L{f_x}}\rA_{L^2}\right)
\lA\frac{h}{\L{f_x}}\rA_{L^2}.
\end{align*}
This will allow us to commute derivatives with the Muskat equation. 
In addition, we shall prove two similar results required to analyze 
the convective term and a remainder term 
(see Proposition~\ref{X1} and Proposition~\ref{P:5.6}).

By so doing, we will get that, for any time dependent function $\lambda=\lambda(t)$,
\be\label{iinfty}
\begin{aligned}
&\fract \blA \D^{\tdm,\phi}f\brA_{L^2}^2
+ \lA\frac{|D|^{2,\phi}f}{\L{f_x}}\rA_{L^2}^2 \\
&\qquad\lesssim
\kappa\Big(\frac{1}{\lam}\Big)^{-1}\left(1+\blA\D^{\tdm,\phi}f\brA_{L^2}^{7}\right)
\blA\D^{\tdm,\phi}f\brA_{L^2} \lA\frac{|D|^{2,\phi}f}{\L{f_x}}\rA_{L^2}^2
\\
&\qquad\quad+
\left( 1+\blA\D^{\frac{3}{2},\phi} f\brA_{L^2}^{5}\right)\brA\D^{\frac{3}{2},\phi}f\brA_{L^2} \lA\frac{f_{xx}}{\L{f_x}}\rA_{L^2}\lA\frac{|D|^{2,\phi}f}{\L{f_x}}\rA_{L^2} \\
&\qquad\quad+\frac{1}{\sqrt{\lam}}\left(1+\blA\D^{\frac{3}{2},\phi}f\brA_{L^2}\right) \blA\D^{\frac{3}{2},\phi}f\brA_{L^2}^2\blA\D^{2,\phi}f\brA_{L^2},
\end{aligned}
\ee
where $f_x=\partial_x f$, $f_{xx}=\partial_{xx}f$ and $\L{f_x}=\sqrt{1+f_x^2}$. 
Notice that \e{iinfty} implies~\e{i10} by 
choosing $\kappa=1$ (then $\phi$ is a constant function) 
and letting $\lambda$ goes to $+\infty$.

Once \e{iinfty} will be established, the derivation of our main estimate~\e{e:A-B} 
will require to prove commutator estimates and interpolation inequalities 
which are also compatible with our singular weights. 

One can mention in this introduction the following 
estimate (see Lemma~\ref{L:Hilbert2}) 
about the Hilbert transform,
\be\label{i20a-1}
\int \frac{(\mathcal{H}g)^2}{1+h^2}\dx\lesssim\(1+\Vert h\Vert_{\dot H^{\frac{1}{2}}}^2\)
\int  \frac{g^2}{1+h^2}\dx.
\end{equation}
In the same vein, we shall prove that (see Proposition \ref{P:commD})
\be\label{i21}
\lVert \D^{\sigma,\phi}\left(\frac{g}{\sqrt{1+h^2}}\right)- \frac{1}{\sqrt{1+h^2}}
\D^{\sigma,\phi}g\rVert_{L^2}
\lesssim\blA \D^{\frac{1}{2},\phi}h\brA_{L^2}   
\lVert \frac{g}{\sqrt{1+h^2}}\rVert_{\dot H^{\sigma}}.
\ee

\subsubsection*{Uniqueness and propagation of regularity} 
Once {\em a priori} estimates will be granted, the end of the proof will rely on 
two additional ingredients. Firstly, a contraction estimate for 
the difference of two solutions. Namely, we shall prove (see Proposition~\ref{P:6.final}) that if $f_1$ and $f_2$ solve the Muskat equation, then
\begin{equation*}
\fract\Vert (f_1-f_2)(t) \Vert_{\dot{H}^\mez}\leq \mathcal{G}(M)
\log\left(2+\sum_{k=1}^2
\lA\frac{\partial_{xx}f_k}{\langle\partial_xf_k\rangle}(t)\rA_{L^2}\right)  \Vert (f_1-f_2)(t)\Vert_{\dot H^{\mez}},
\end{equation*}
where
$$
M\defn \sup_{t\in [0,T]}\big\{
\lA f_1(t)\rA_{\mathcal{H}^{\tdm,\phi}}+\lA f_2(t)\rA_{\mathcal{H}^{\tdm,\phi}}\}.
$$ 
Since $M$ involves a stronger norm than the one of $C^0([0,T];H^\tdm(\xR))$, 
to close the argument and obtain an unconditional uniqueness result 
in the space $X^\tdm(\xR)$ introduced 
in the statement of Theorem~\ref{Theorem}, 
we shall prove an enhanced regularity result. 
Namely, we shall prove (see Proposition~\ref{P:6.1b}) 
that for any solution in the space $X^\tdm(T)$, if the initial data 
belongs to $\mathcal{H}^{\tdm,\phi}(\xR)$ for some $\phi$, then the solution 
$f$ belongs to $C^0([0,T];\mathcal{H}^{\tdm,\phi}(\xR))$, after a possible shrinking of 
the time interval. 

\subsection{Plan of the paper}
The remainder of the paper is organized as follows. Section~\ref{S:2} reviews many inequalities 
which will be used throughout the article: Hardy, Hardy-Littlewood-Sobolev, Minkowski, embedding of Sobolev spaces in 
Triebel-Lizorkin spaces. We also prove in Section~\ref{S:2.2} the estimate~\e{i20a-1}. 
In Section~\ref{S:wRp} we recall various results about weighted fractional Laplacians 
that we borrow from our previous paper~\cite{AN1}, along with the 
key commutator estimate~\e{i21}. In Section~\S\ref{S:pMe}, we recall from~\cite{Alazard-Lazar} a paradifferential 
analysis of the Muskat equation, and apply it to obtain an energy estimate. 
The main task is then to estimate the various commutators and remainder terms, to obtain~\e{iinfty}. 
This is done in \S\ref{S:5}, which is the core of this paper. 
Once \e{iinfty} is proved, we conclude the proof in Section~\S\ref{S:6}
by using interpolation inequalities and a fair amount of bookkeeping.

\subsection{Notations}
Triebel-Lizorkin spaces are recalled in~\S\ref{S:2.1}. We refer to \S\ref{S:2.2} for the definitions of 
the Hilbert transform and Riesz potentials. More definitions and notations about 
the weighted fractional Laplacians are given in~\S\ref{S:wRp}. 

In addition:
\begin{enumerate}
\item We denote partial derivatives using subscripts: $f_x=\partial_xf$ 
and $f_{xx}=\partial_{xx}f$. 

\item If $A, B$ are nonnegative quantities, we use $A \lesssim B$ to denote the estimate 
$A\le CB$ for some constant $C$ depending only on fixed quantities, and $A\sim B$ to 
denote the estimate $A\les B\les A$.

\item Given two operators $A$ and $B$, we denote by $[A,B]$ the commutator 
$A\circ B-B\circ A$.
\item Given a normed space~$X$
and a function~$\varphi=\varphi(t,x)$ defined on~$[0,T]\times \xR$ with values in~$X$, 
we denote by~$\varphi(t)$ the function~$x\mapsto
\varphi(t,x)$. In the same vein, we use $\lA \varphi\rA_X$ as 
a compact notation for the time dependent function 
$t\mapsto \lA \varphi(t)\rA_X$. 

\item All functions are assumed to be real-valued. However, in many computations, 
we will use the complex-modulus notation: for instance, 
we will write $\la \Delta_\alpha f\ra^2$ or $\la \alpha\ra^2$ instead of $(\Delta_\alpha f)^2$ 
or $\alpha^2$, since 
we think that this might help to read certain formulas.

\end{enumerate}

\section{Preliminaries}\label{S:2}

We begin by discussing some classical inequalities which are basic to the analysis developed below. 
We also prove in~\S\ref{S:2.2} a new estimate about the Hilbert transform.

\subsection{Minkowski's inequality}
Suppose that $(S_1,\mu_1)$ and $(S_2,\mu_2)$ 
are two $\sigma$-finite measure spaces and 
$F\colon S_1\times S_2\to\xR$ is measurable. 
Then
\be\label{Min}
\bigg(\underset{S_2}{\int} 
\bigg\vert\underset{S_1}{\int} F(x,y)\mu_1(\dx)\bigg\vert^p 
\mu_2(\dy)\bigg)^{\frac{1}{p}}\le 
\bigg(\underset{S_1}{\int} 
\bigg(\underset{S_2}{\int} \la F(x,y)\ra^p \mu_2(\dy)\bigg)
^{\frac{1}{p}}
\mu_1(\dx).
\ee

\subsection{Hardy's inequality}
Consider two real numbers $r\geq 1$ and $s<r-1$. 
For any non-negative function $f$, there holds (see~\cite[page 20]{Zygmund})
\begin{equation}\label{Hardy0}
\int_0^\infty \left(\frac{1}{x}\int_0^xf(\tau)\dtau\right)^r x^s\dx 
\le \(\frac{r}{r-s-1}\)^r\int_0^\infty f(x)^rx^s\dx.
\end{equation}
In particular, one has,
\begin{equation}\label{Hardy}
\int_{\mathbb{R}} \left(\fint_0^x f(\tau)\dtau\right)^r |x|^s\dx 
\le \(\frac{r}{r-s-1}\)^r\int_{\mathbb{R}}  f(x)^r|x|^s\dx.
\end{equation}

\subsection{Riesz potentials}\label{S:2.2.a}
The Riesz potentials are the operators
$$
\mathbf{I}_s f(x)=\int\frac{f(y)}{|x-y|^{1-s}}\dy,\quad 0<s<1.
$$
The Hardy-Littlewood-Sobolev lemma states that these operators map $L^p\to L^q$, for 
\be\label{Riesz}
1<p<\frac{1}{s}
\quad\text{such that}\quad 
s+\frac1q=\frac1p\cdot
\ee
The previous result will be used 
throughout the paper. In particular, 
we will make extensive use of the fact 
that, for $0<s<1$, the Fourier multiplier $\D^{-s}$ is a multiple of $\mathbf{I}_{s}$:
\be\label{b1}
\D^{-s}=c_{s}\mathbf{I}_{s},
\ee
for some constant $c_{s}$.

\subsection{Hilbert transform}\label{S:2.2}
The Hilbert transform $\mathcal{H}$ is defined by:
\be\label{n3}
\mathcal{H}u(x)=-\frac{1}{\pi}\mathrm{pv}\int_\xR \frac{u(y)}{x-y}\dy.
\ee
This operator is ubiquitous in the study of the Muskat equation. 

For later purposes, we record in this paragraph two commutator estimates for the
Hilbert transform. 

Firstly, we will need to estimate the commutator of $\mathcal{H}$ 
with a function in $\dot{H}^1(\xR)$. 
Since $\dot{H}^1(\xR)$ is embedded into $C^{\mez}(\xR)$ and since 
the Hilbert transform is an operator of order $0$ (bounded from $H^\mu(\xR)$ 
to itself for any $\mu\in\xR$), we expect 
the commutator to be an operator of order $-1/2$. 
The next result (which is Lemma~$3.4$ in~\cite{AN1}) 
states that this is indeed the case when 
the commutator is acting on an element of $\dot{H}^{-\mez}(\xR)$.
\begin{lemma}[from~\cite{AN1}]\label{L:Hilbert1}
For all $g_1$ in $\dot{H}^1(\xR)$ and all $g_2$ in $\dot{H}^\mez(\xR)$, there holds
\begin{equation*}
\lA \left[\mathcal{H}, g_1\right ](\partial_x g_2)\rA_{L^2}
\lesssim\lA g_1\rA_{\dot H^{1}}\lA g_2\rA_{\dot H^{\mez}}.
\end{equation*}
\end{lemma}
We will also need a lemma which generalizes 
the boundedness of the Hilbert transform on $L^2(\xR)$.  
\begin{lemma}\label{L:Hilbert2}
Consider a function $\gamma\colon\xR\to [0,1]$ such that, 
for any $z_1,z_2$ in $\xR$,  
\be\label{assu:F2}
\la \gamma(z_1)-\gamma(z_2)\ra\le 
C\gamma(z_2)\la z_1-z_2\ra,
\ee
for some absolute constant $C$. 
Then, for all function $g$ in $L^2(\xR)$ 
and all function $h$ in $\dot{H}^\mez(\xR)$, there holds
\begin{equation}\label{Z3}
\int \la\gamma(h)\mathcal{H}g\ra^2\dx \lesssim_{C}\(1+\Vert h\Vert_{\dot H^{\frac{1}{2}}}^2\)
\int \la \gamma(h)g\ra^2\dx.
\end{equation}
\end{lemma}
\begin{remark}\label{R:3.4}
$i)$ We will apply this result with $\gamma(z)=\L{z}^{-1}=(1+z^2)^{-1/2}$. 

$ii)$ A classical result states that
$$
\lA \big[b,\mathcal{H}\big]u\rA_{L^p}\le C \lA b\rA_{{\rm BMO}}\lA u\rA_{L^p},\quad 1<p<\infty.
$$
This implies that the left-hand side of~\e{Z3} is bounded by 
$$
\lA \gamma(h)g\rA_{L^2}+\lA \gamma(h)\rA_{{\rm BMO}}\lA g\rA_{L^2}.
$$
Compared to the latter inequality, the key point is that the 
right-hand side of \e{Z3} does not involve the $L^2$-norm 
of $g$, but the one of $\gamma(h)g$. This is crucial here since we do not assume that 
$\gamma(h)$ is not bounded from below. 
\end{remark}
\begin{proof}
Set $u=\gamma(h)g$ and write
\begin{align*}
\int \la \gamma(h)\mathcal{H}g\ra^2\dx&\sim\int \bigg(\int \gamma(h)(x)\frac{g(y)}{x-y}\dy\bigg)^2\dx\\
&\lesssim \int |\mathcal{H}u|^2\dx+\int \left(\int \la u(y)\ra \la h(x)-h(y)\ra\frac{\dy}{|x-y|}\right) ^2 \dx,
\end{align*}
where we have used \eqref{assu:F2} in the last inequality. 

The first term in the right-hand side above is bounded by $\lA u\rA_{L^2}^2$ while the H\"older's inequality 
implies that the second term is estimated by
$$
\bigg(\int\bigg(\int \frac{\la u(y)\ra^\tdm}{\la x-y\ra^\mez}\dy\bigg)^{4}\dx\bigg)^{\frac{1}{3}}\times 
\bigg(\iint \bigg(\frac{|h(x)-h(y)|}{|x-y|^{\frac{1}{3}}}\bigg)^{3} \frac{\dy\dx}{|x-y|}\bigg)^{\frac{2}{3}}.
$$
Using the boundedness of the Riesz potentials on the Lebesgue spaces (see~\e{Riesz}), 
we get that 
the first factor is estimated by
$$
\bigg(\int\bigg(\int \frac{\la u(y)\ra^\tdm}{\la x-y\ra^\mez}\dy\bigg)^{4}\dx\bigg)^{\frac{1}{3}}\les
\blA |u|^{\tdm}\brA_{L^{\frac{4}{3}}}^{\frac{4}{3}}=\lA u\rA_{L^2}^2.
$$
On the other hand, the Sobolev 
embedding $\dot{H}^\mez(\xR)\hookrightarrow \dot{W}^{\frac13,3}(\xR)$ (see~\e{Sobolev} and \e{Gagliardo} below) implies 
that
$$
\bigg(\underset{\xR\times\xR}{\iint} \bigg(\frac{|h(x)-h(y)|}{|x-y|^{\frac{1}{3}}}\bigg)^{3}
\frac{\dy\dx}{|x-y|}\bigg)^{\frac{2}{3}}
\les \Vert h\Vert_{\dot H^{\frac{1}{2}}}^2.
$$
This proves~\e{Z3}.
\end{proof}

\subsection{Triebel-Lizorkin norms}\label{S:2.1}

For easy reading, we review some elementary results about 
Triebel-Lizorkin spaces following Triebel (see~\cite[$\S 2.3.5$]{Triebel-TFS} 
and~\cite[section~$3$]{Triebel1988}). 

Given a function $f\colon\xR\to\xR$, an integer $m\in\xN\setminus\{0\}$ 
and a real number $h\in\xR$,
 we define the finite differences 
$\delta_h^mf$ as follows:
$$
\delta_hf(x)=f(x)-f(x-h),\qquad \delta_h^{m+1}f=\delta_h(\delta_h^mf).
$$
\begin{definition}
Given 
an integer $m\in\xN\setminus\{0\}$, 
a real number $s\in [m-1,m)$ and $(p,q)\in [1,\infty)^2$, 
the homogeneous Triebel-Lizorkin space $\dot F^{s}_{p,q}(\xR)$ 
consists of those tempered distributions $f$
whose Fourier transform is integrable on a neighborhood of 
the origin and such that 
\be\label{n5}
\lA f\rA_{\dot F^s_{p,q}}
=\left(\int_{\mathbb{R}}\left(\int_{\mathbb{R}}
\la\delta_h^mf(x)\ra^q\frac{\dh}{|h|^{1+qs}}\right)^{\frac{p}{q}}\dx\right)^{\frac{1}{p}}<+\infty.
\ee
\end{definition}

We will extensively use the embedding of homogeneous Sobolev spaces 
in  Triebel-Lizorkin spaces.  Firstly, recall that 
$\lA \cdot\rA_{\dot H^{s}}$ and $\lA \cdot\rA_{\dot{F}^s_{2,2}}$ 
are equivalent. 
Moreover, for $s\in (0,1)$,
\be\label{SE0}
\lA u\rA_{\dot H^{s}}^2=\frac{1}{4\pi c(s)}\lA u\rA_{\dot{F}^s_{2,2}}^2
\quad\text{with}\quad c(s)=\int_\xR \frac{1-\cos(h)}{\la h\ra^{1+2s}}\dh.
\ee 
In addition, for any $2<p< \infty$ and any $q\leq\infty$, if $t-\frac{1}{2}=s-\frac{1}{p}$ then 
\begin{equation}\label{FB}
\lA f\rA_{\dot F^{s}_{p,q}(\mathbb{R})}\lesssim \lA f\rA_{\dot H^{t}(\mathbb{R})}.
\end{equation}
This includes the Sobolev embedding
\be\label{Sobolev}
\dot{H}^t(\xR)\hookrightarrow L^{\frac{2}{1-2t}}(\xR) \quad\text{for}\quad 
0\le t<\mez\cdot
\ee

Recall also that, for $s\in (0,1)$ and $1\le p<+\infty$, 
fractional Sobolev spaces~$\dot{W}^{s,p}(\xR)$ are 
associated with the Gagliardo semi-norms
\be\label{Gagliardo}
\lA u\rA_{\dot{W}^{s,p}}^p\defn 
\underset{\xR\times\xR}\iint \frac{\la u(x)-u(y)\ra^p}{\la x-y\ra^{sp}}\frac{\dx\dy}{\la x-y\ra}\cdot
\ee

\subsection{Several inequalities}
We gather in this paragraph several inequalities which will be used repeatedly. 
We also want to mention them in this preliminary section because their proofs 
are good examples of the kind of techniques that will be used later on. 

We begin with the three following elementary estimates:
\begin{align}
&\iint_{\xR^2}  \la \Delta_\alpha f\ra^2\dalpha\dx\sim \lA f\rA_{\dot{H}^{\mez}}^2,\label{an1:n15}\\[1ex]
&\iint_{\xR^2} \la \Delta_\alpha f-f_x\ra^2\frac{\dalpha}{|\alpha|^2}\dx\sim \lA f\rA_{\dot{H}^\tdm}^2,\label{an1:x1}\\[1ex]
&\int_\xR \left( \int_\xR \la\Delta_\alpha f\ra^2\dalpha\right)^2\dx\lesssim\lA f\rA_{\dot H^{\frac{3}{4}}}^4.\label{an1:n120}
\end{align}
The estimates~\e{an1:n15} and \e{an1:n120} follow immediately from the 
Sobolev embeddings in Tribel-Lizorkin spaces (see~\e{SE0} and \e{FB}), 
while \e{an1:x1} is proved by using the Fourier transform. 

Next, we prove more technical estimates for the following quantities:
$$
\Delta_\alpha f-f_x \quad\text{or}\quad \Delta_\alpha f-\Delta_{-\alpha}f.
$$
We are going to prove estimates valid for 
frequencies large enough 
(or equivalently, estimates valid when we integrate 
only on sets where the physical variable $|\alpha|$ is smaller than some threshold $\lambda$).

\begin{definition}\label{D:ES}
Given a real number $\alpha$, we define the operators $E_\alpha$ and $S_\alpha$ by
\begin{align*}
E_\alpha(f)(x)&=\Delta_\alpha f(x)-f_x(x)=\frac{f(x)-f(x-\alpha)}{\alpha}-f_x(x),\\
S_\alpha(f)(x)&=\Delta_\alpha f(x)-\Delta_{-\alpha}f(x)=\frac{2f(x)-f(x-\alpha)-f(x+\alpha)}{\alpha}\cdot
\end{align*}
\end{definition}
\begin{lemma}\label{L:2.7}
$i)$ Consider five real numbers $(\gamma,a_1,a_2,a_3,\beta)$ satisfying
$$
a_1\in \xR,\quad \gamma\ge 2,\quad 1<a_2<1+\gamma,\quad a_3=a_1+\frac{a_2-1}{\gamma},\quad \beta=\frac{2+2\gamma-2a_2}{\gamma}\cdot
$$
Then, for any $\lambda>0$ and any $f\in \mathcal{S}(\xR)$, the following inequalities hold:
\begin{align}
&\int_{|\alpha|\le\lam}\lA E_\alpha(f)\rA_{\dot{H}^{a_1}}^\gamma\frac{\dalpha}{|\alpha|^{a_2}}
\lesssim
\left(\int \min\{\lam|\xi|,1\}^{\beta} |\xi|^{2+2a_3} |\hat f(\xi)|^2\dxi
\right)^{\frac{\gamma}{2}},\label{e1}\\
&\int_{|\alpha|\le\lam}\lA S_\alpha(f)\rA_{\dot{H}^{a_1}}^\gamma\frac{\dalpha}{|\alpha|^{a_2}}
\lesssim\left(\int \min\{\lam|\xi|,1\}^{\beta} |\xi|^{2+2a_3} |\hat f(\xi)|^2\dxi
\right)^{\frac{\gamma}{2}},\label{e2}
\end{align}
and
\be\label{e3}
\int_{|\alpha|\le\lam}\lA \delta_\alpha(f)\rA_{\dot{H}^{a_1}}^\gamma\frac{\dalpha}{|\alpha|^{a_2}}
\lesssim\left(\int \min\{\lam|\xi|,1\}^{\beta} |\xi|^{2a_3} |\hat f(\xi)|^2\dxi
\right)^{\frac{\gamma}{2}}.
\ee
\end{lemma}

\begin{proof}
This proof is based on the Fourier transform. The key point is that 
\be\label{e4}
\begin{aligned}
\Vert E_\alpha(f)\Vert_{\dot H^{a_1}}^2&=\int |\xi|^{2a_1}\frac{\left|1-e^{-i\alpha\xi}-i\xi\alpha\right|^2 }{\alpha^2}|\hat f(\xi)|^2\dxi\\
&\lesssim \int |\xi|^{2+2a_1}\min\{1,\la \alpha\ra \la \xi\ra\} ^2|\hat f(\xi)|^2\dxi,\\ 
\Vert S_\alpha(f)\Vert_{\dot H^{a_1}}^2&\lesssim \int |\xi|^{2+2a_1}\min\{1,\la \alpha\ra \la \xi\ra\} ^2|\hat f(\xi)|^2\dxi,\\
\Vert \delta_\a(f)\Vert_{\dot H^{a_1}}^2&\lesssim \int |\xi|^{2a_1}
\min\{1,\la \alpha\ra \la \xi\ra\}^2|\hat f(\xi)|^2\dxi.
\end{aligned}
\ee

$i)$ Let us prove~\e{e1} and \e{e2} . 
Directly from~\e{e4} we have
\begin{align*}
Q^\gamma&\defn\int_{|\alpha|\le\lam}\lA E_\alpha(f)\rA_{\dot{H}^{a_1}}^\gamma\frac{\dalpha}{|\alpha|^{a_2}}
+\int_{|\alpha|\le\lam}\lA S_\alpha(f)\rA_{\dot{H}^{a_1}}^\gamma\frac{\dalpha}{|\alpha|^{a_2}}
\\
&\lesssim \underset{|\alpha|\le\lam}{\int}\left(\int_{\xR}|\xi|^{2+2a_1}\min\{1,\la \alpha\ra \la \xi\ra\} ^2 |\hat f(\xi)|^2\dxi\right)^{\frac{\gamma}{2}}\frac{\dalpha}{|\alpha|^{a_2}}.
\end{align*}
We raise both sides 
of the previous inequality to the power $\gamma^{-1}$, to obtain
\begin{equation*}
Q\lesssim \Bigg(\underset{|\alpha|\le\lam}{\int}\left(\int_{\xR}
\min\{1,\la \alpha\ra \la \xi\ra\} ^2 |\xi|^{2+2a_1}|\hat f(\xi)|^2\dxi\right)^{\frac{\gamma}{2}}
\frac{\dalpha}{|\alpha|^{a_2}}\Bigg)^{\frac2\gamma \cdot \mez}.
\end{equation*}
Now we are in position to apply Minkowski's inequality. Since $\gamma/2\ge 1$, 
we obtain that
\begin{equation}\label{e5}
Q
\lesssim \Bigg(\int_{\xR}\Bigg( \underset{|\alpha|\le\lam}{\int}
\min\{1,\la \alpha\ra \la \xi\ra\} ^\gamma
\frac{\dalpha}{|\alpha|^{a_2}}\Bigg)^{\frac{2}{\gamma}}|\xi|^{2+2a_1}|\hat f(\xi)|^2\dxi\Bigg)^\frac{\gamma}{2}.
\end{equation}
Now, by an elementary change of variable, we have
\begin{align*}
\underset{|\alpha|\le\lam}{\int}\min\{1,\la \alpha\ra \la \xi\ra\} ^\gamma
\frac{\dalpha}{|\alpha|^{a_2}}&=\la \xi\ra^{a_2-1}\underset{|\alpha|\le\lam\la\xi\ra}{\int}\min\{1,|\alpha|\} ^\gamma
\frac{\dalpha}{|\alpha|^{a_2}}\\
&\lesssim \la \xi\ra^{a_2-1} \min\left\{1, \lam|\xi|\right\}^{1+\gamma-a_2}.
\end{align*}
It follows that
\begin{align}\label{V1}
\Bigg(\underset{|\alpha|\le\lam}{\int}\min\{1,\la \alpha\ra \la \xi\ra\} ^\gamma
\frac{\dalpha}{|\alpha|^{a_2}}\Bigg)^{\frac{2}{\gamma}}&\les \la \xi\ra^{\frac{2 (a_2-1)}{\gamma}}
\times \min\left\{1, \lam|\xi|\right\}^{\frac{2+2\gamma-2a_2}{\gamma}}.
\end{align}
Then the wanted estimates~\e{e1} and \e{e2} 
follow from the previous bound and \e{e5}.

$ii)$ It remains to prove~\e{e3}. 
Starting from~\e{e4} and repeating the arguments used to infer~\e{e5}, we get
\begin{align*}
\int_{|\alpha|\le\lam}\lA \delta_\alpha(f)\rA_{\dot{H}^{a_1}}^\gamma\frac{\dalpha}{|\alpha|^{a_2}}
&\lesssim
\Bigg(\int_{\xR}\Bigg( \underset{|\alpha|\le\lam}{\int}\min\{1,\la \alpha\ra \la \xi\ra\}^\gamma
\frac{\dalpha}{|\alpha|^{a_2}}\Bigg)^{\frac{2}{\gamma}}|\xi|^{2a_1}
|\hat f(\xi)|^2\dxi\Bigg)^\frac{\gamma}{2}\\&\overset{\eqref{V1}}
\lesssim \left(\int \min\{\lam|\xi|,1\}^{\frac{2+2\gamma-2a_2}{\gamma}} |\xi|^{2a_3} |\hat f(\xi)|^2\dxi
\right)^{\frac{\gamma}{2}},
\end{align*}
which  implies~\e{e3}. 
\end{proof}

\section{Weighted fractional Laplacians}\label{S:wRp}

This section is devoted to the study of weighted fractional Laplacians, 
which is our pivotal tool to study the Muskat equation in critical spaces. 
The first paragraph is a review consisting of various notations and results about 
these operators, following our previous paper~\cite{AN1}. 
After these preliminaries, in $\S\ref{S:4.2}$ we pass to the proof of a key commutator estimate, 
which generalizes in some sense Lemma~\ref{L:Hilbert2} about 
the Hilbert transform.

\subsection{Definition and characterization}

Following our previous works, we will work with the following weighted fractional Laplacians. 

\begin{definition}\label{defi:D}
Consider $s$ in $[0,+\infty)$ and 
a function $\phi\colon [0,\infty)\to [1,\infty)$ satisfying a doubling condition
(such that $\phi(2r)\leq c_0\phi(r)$ for any $r\geq 0$ and some constant $c_0>0$). Then $|D|^{s,\phi}$ denotes 
the Fourier multiplier with symbol $|\xi|^s\phi(|\xi|)$. More precisely,
\begin{equation*}
\mathcal{F}( |D|^{s,\phi}f)(\xi)=|\xi|^s\phi(|\xi|) \mathcal{F}(f)(\xi).
\end{equation*}
Moreover, we define the space
$$
\mathcal{H}^{s,\phi}(\xR)=\{ f\in L^2(\xR)\, :|D|^{s,\phi}f\in L^2(\xR)\},
$$
equipped with the norm
$$
\lA f\rA_{\mathcal{H}^{s,\phi}}\defn \lA f\rA_{L^2}
+\left(\int_\xR \la \xi\ra^{2s}(\phi(\la\xi\ra))^2\bla\hat{f}(\xi)\bra^2\dxi\right)^\mez.
$$
\end{definition}
\begin{remark}
The Fourier multiplier $|D|^{s,\phi}$ with 
$\phi(r)=\log(2+r)^a$
was introduced and studied in~\cite{BN18a,BN18b,BN18d} for $s\in [0,1)$ (also see \cite{Ng}). 
\end{remark}
We shall consider some special functions $\phi$ which depend 
on some extra function $\kappa\colon [0,\infty)\to [1,\infty)$, of the form
\begin{equation}\label{n10}
\phi(\lam)=\int_{0}^{\infty}\frac{1-\cos(h)}{h^2} \kappa\left(\frac{\lam}{h}\right) \dh, \quad \text{for }\lambda\ge 0.
\end{equation}
Let us clarify the assumptions on $\kappa$ which will be needed in the sequel.

\begin{definition}\label{defi:admissible}
We say that a function $\kappa\colon[0,\infty) \to [1,\infty)$ 
is an admissible weight if it satisfies the following three conditions: 
\begin{enumerate}[$({{\rm H}}1)$]
\item\label{H1} $\kappa$ is increasing;
\item\label{H2} there is a positive constant $c_0$ such that $\kappa(2r)\leq c_0\kappa(r)$ for any $r\geq 0$; 
\item\label{H3} the function $r\mapsto \kappa(r)/\log(4+r)$ is decreasing on  $[r_0,\infty)$, for some $r_0>0$ large enough.
\end{enumerate}
\end{definition}

In \cite[Prop.~$2.7$]{AN1}, it is proved that $\phi$ and $\kappa$ are equivalent. 
More precisely, we have the following

\begin{proposition}[from~\cite{AN1}]\label{Z9}
Assume that $\phi$ is as defined in~\e{n10} for some admissible weight 
$\kappa$. 
Then there exist two constants $c,C>0$ such that, for all $\lambda\ge 0$,
\be\label{n60b}
c\kappa(\lam)\le \phi(\lam)\le C \kappa(\lam).
\ee
In particular, $\phi$ is also an admissible weight.
\end{proposition}

For later applications, we make some preliminary remarks about admissible weights.
\begin{lemma}
$i)$ For any $\sigma>0$, there exists a positive constant $C$ such that, for any real numbers 
$0< r\le \mu$,
\be\label{est:kappa}
r^\sigma\kappa\left(\frac{1}{r}\right) \le C_\sigma\,\mu^\sigma
\kappa\left(\frac{1}{\mu}\right).
\ee
$ii)$ For any $\beta$ in $(0,1]$, there is a positive constant $C_\beta$ such that, 
for all $\lambda>0$ and all $\xi\in \xR$, 
\begin{equation}\label{est:kappa3}
\min\{\lam|\xi|,1\}^\beta\le C_\beta \frac{\kappa(|\xi|)}{\kappa(1/\lambda)}\cdot
\end{equation}
\end{lemma}
\begin{proof}
To prove statement $i)$ we use a decomposition into 
three factors:
$$
r^\sigma\kappa\(\frac{1}{r}\)=
\frac{\kappa\(\frac{1}{r}\)}{\log\(4+\frac{1}{r}\)}\times \left[ r^{\sigma}\log\(e^{\frac1\sigma}+\frac{1}{r}\)\right]
\times \frac{\log\(4+\frac{1}{r}\)}{\log\(e^{\frac1\sigma}+\frac{1}{r}\)}\cdot
$$ 
The first factor is an increasing function of $r$ by 
assumption on $\kappa$ (see Definition~\ref{defi:admissible}); 
the second factor is also increasing, 
as can be verified by computing its derivative, and the 
third factor is a bounded function on $(0,+\infty)$. 

Now, notice that the inequality in statement $ii)$ is obvious if $|\lambda \xi|\ge 1$ 
since $\kappa$ is increasing. 
When $|\lambda \xi|\le 1$, it follows directly from $i)$.
\end{proof}

The next proposition clarifies the links between $\D^{s,\phi}$ and the 
function $\kappa$ (we refer to Lemma~$2.6$ and Proposition~$2.7$ in~\cite{AN1} for the proof).

\begin{proposition}[from~\cite{AN1}]\label{P:3.5}
Assume that $\phi$ is as defined in~\e{n10} for some admissible weight $\kappa$. 

$i)$ For all 
$g\in \mathcal{S}(\xR)$, there holds
\begin{equation}\label{d1phi}
\D^{1,\phi}g(x)=\frac{1}{4}\int_{\xR}\frac{2g(x)-g(x+\alpha)-g(x-\alpha)}{\alpha^2}
\kappa\left(\frac{1}{|\alpha|}\right) \dalpha.
\end{equation}

$ii)$ Given $g\in \mathcal{S}(\xR)$, define the semi-norm
$$
\lA g\rA_{s,\kappa}:=\left(\iint_{\xR^2} \la 2g(x)-g(x+\alpha)-g(x-\alpha)\ra^2
\left(\frac{1}{|\alpha|^{s}}\kappa\left(\frac{1}{|\alpha|}\right)\right)^2\frac{\dx \dalpha}{|\alpha|}\right)^\mez\cdot
$$
Then, for all $1<s<2$, there exist two constants $c,C>0$ such that, for 
all $g\in \mathcal{S}(\xR)$,
\begin{equation*}
c\int_{\xR}  \la \D^{s,\phi}g(x)\ra^2\dx\le \lA g\rA_{s,\kappa}^2
\le C\int_{\xR}\la \D^{s,\phi}g(x)\ra^2\dx.
\end{equation*}
\end{proposition}

\subsection{A commutator estimate}\label{S:4.2}
Here is the main result of this section.
\begin{proposition}\label{P:commD}
Let 
$\phi$ be 
as defined in~\e{n10} for some admissible weight~$\kappa$. 
Consider a function $\gamma\colon\xR\to[0,1]$ such that, 
for any $x,y$ in $\xR$,  
\be\label{assu:F}
\la \gamma(x)-\gamma(y)\ra\le 
C\gamma(y)\la x-y\ra,
\ee
for some absolute constant $C$. 
Then, for any $\sigma\in (0,1/2)$, and $h,g\in \mathcal{S}(\mathbb{R})$, there holds
\be\label{Z1}
\lVert \D^{\sigma,\phi}\left(
\gamma(h)g\right)- \gamma(h)
\D^{\sigma,\phi}g\rVert_{L^2}
\lesssim_{C} \blA \D^{\frac{1}{2},\phi}h\brA_{L^2}   
\lVert \gamma(h)g\rVert_{\dot H^{\sigma}}.
\ee
\end{proposition}
\begin{remark}
As mentioned in Remark~$\ref{R:3.4}$, the key point is that 
the right-hand side does not involve the $L^2$-norm 
of $g$ but the one of $\gamma(h)g$. This is crucial here since we do not assume that 
$\gamma(h)$ is not bounded from below. 
\end{remark}
\begin{proof}
The first step of the proof consists in proving the following inequality:
\begin{multline}
\label{b2}
\left| \D^{\sigma,\phi}\left(\gamma(h) g\right)(x)- (\gamma(h)\D^{\sigma,\phi}g )(x)\right|
\\ \les \int \la	 \delta_z h(x)\ra \kappa\(\frac{1}{|z|}\) 
(\gamma(h)g)(x-z)\frac{\dz}{|z|^{1+\sigma}}\cdot
\end{multline}
To do so, we begin by applying the representation formula~\e{d1phi} (resp.\ \e{b1}) 
for the operator $\D^{1,\phi}$ (resp.\ $\D^{-a}$ with $0<a<1$). 
Since 
$$
\D^{\sigma,\phi}=\D^{1,\phi}\D^{-(1-\sigma)},
$$
it follows that
\begin{align*}
&\D^{\sigma,\phi} \eta(x)\\
&\quad=\frac{1}{4}\int_{\xR}
\frac{2|D|^{-(1-\sigma)}\eta(x)-|D|^{-(1-\sigma)}\eta(x+t)-|D|^{-(1-\sigma)}\eta(x-t)}{t^2}
\kappa\left(\frac{1}{|t|}\right) \dt\\
&\quad=\frac{c_{1-\sigma}}{4}  \int_{\xR}\int_{\xR}\left(\frac{2}{|x-z|^\sigma}
-\frac{1}{|x+t-z|^\sigma}-\frac{1}{|x-t-z|^\sigma}\right)
\frac{1}{t^2}\kappa\left(\frac{1}{|t|}\right)\eta(z)\dz\dt.
\end{align*}
Set $u=\gamma(h)g$. It follows from the assumption~\e{assu:F} 
and 
the triangle inequality that
$$
\la g(z)\big(\gamma(h(z))-\gamma(h(x))\big)\ra\les \la u(z)\ra \la h(z)-h(x)\ra.
$$
Consequently, we get
\begin{align*}
\left| \D^{\sigma,\phi}\left(\gamma(h) g\right)(x)- (\gamma(h)\D^{\sigma,\phi}g )(x)\right| 
\les\int_{\xR}K(x-z)\la u(z)\ra\la h(z)-h(x)\ra\dz,
\end{align*}
where
$$
K(y)=\int_{\xR}\left|\frac{2}{|y|^\sigma}-\frac{1}{|y+t|^\sigma}-\frac{1}{|y-t|^\sigma}\right|
\frac{1}{t^2}\kappa\left(\frac{1}{|t|}\right)\dt.
$$
Now, to obtain the claim~\e{b2} (with $z$ replaced by $x-z$) it is sufficient to prove that
\be\label{b3}
K(y)\les \kappa\left(\frac{1}{|y|}\right) \frac{1}{|y|^{1+\sigma}}\cdot
\ee
To do so, we begin by applying \eqref{est:kappa3} to get that 
\begin{equation}
\kappa\left(\frac{1}{|t|}\right)\lesssim \kappa\left(\frac{1}{|y|}\right) \left(1+|y|^{\frac{\sigma}{4}}|t|^{-\frac{\sigma}{4}}\right).
\end{equation}
Thus, we find that, for any $y\in\mathbb{R}\setminus \{0\}$, 
\begin{equation}
\label{Z01}
K(y)\lesssim \kappa\left(\frac{1}{|y|}\right)
\int\left|\frac{2}{|y|^\sigma}-\frac{1}{|y+t|^\sigma}-\frac{1}{|y-t|^\sigma}\right|\left(1+|y|^{\frac{\sigma}{4}}|t|^{-\frac{\sigma}{4}}\right)
\frac{1}{|t|^{2}}\dt.
\end{equation}
Now, the wanted estimate~\e{b3} follows from the fact that, for any $\beta$ in $[0,1)$, 
$$
\int\left|\frac{2}{|y|^\sigma}-\frac{1}{|y+t|^\sigma}-\frac{1}{|y-t|^\sigma}\right|
\frac{\dt}{|t|^{2+\beta}}=\frac{c(\sigma,\beta)}{|y|^{1+\sigma+\beta}},
$$
as can be verified by an immediate homogeneity argument. 
This completes the proof of the claim~\e{b2}. 

Now, using~\e{b1} and Cauchy-Schwarz inequality, we have
\begin{align*}
I&\defn\int \left| \D^{\sigma,\phi}\left(\gamma(h) g\right)- (\gamma(h)\D^{\sigma,\phi}g )\right|^2\dx
\\&\lesssim \int \left(\int \la	 \delta_zh(x)\ra \kappa\(\frac{1}{|z|}\) 
u(x-z)\frac{\dz}{|z|^{1+\sigma}}\right)^2 \dx\\
&\lesssim \int(|D|^{-\frac{1}{2}+2\sigma}(u^2)(x)\left(
\int \la \delta_zh(x)\ra^2\kappa\(\frac{1}{|z|}\)^2 \frac{\dz}{|z|^{\frac{3}{2}}}\right)\dx\\
&\lesssim
  \blA \D^{-\frac{1}{2}+2\sigma}(u^2)\brA_{L^2}
\int \Vert \delta_zh(\cdot)\Vert_{L^4}^2\kappa\(\frac{1}{|z|}\)^2 \frac{\dz}{|z|^{\frac{3}{2}}}\cdot
\end{align*}
Then, by using the boundedness of Riesz potentials on Lebesgue spaces (see~\e{Riesz}) 
and the Sobolev embedding $\dot{H}^{\frac14}(\xR)\hookrightarrow L^4(\xR)$ 
(see~\e{Sobolev}), 
we conclude that
\begin{align*}
I&\les   \Vert u\Vert_{L^{\frac{2}{1-2\sigma}}}^2\int \Vert \delta_zh(\cdot)\Vert_{\dot H^{\frac{1}{4}}}^2
\kappa\(\frac{1}{|z|}\)^2 \frac{\dz}{|z|^{\frac{3}{2}}}
\\&\lesssim 
\Vert u\Vert_{\dot H^{\sigma}}^2 \int |\xi|^{\frac{1}{2}}\left(\int \min\{|\xi||z|,1\}^2\kappa\(\frac{1}{|z|}\)^2
\frac{\dz}{|z|^{\frac{3}{2}}}\right) |\hat h(\xi)|^2 \dxi.
\end{align*}
The estimate~\e{est:kappa3} already used above to find  \eqref{b3} also implies that
$$\int \min\{|\xi||z|,1\}^2\kappa\(\frac{1}{|z|}\)^2
\frac{\dz}{|z|^{\frac{3}{2}}}\lesssim |\xi|^{\frac{1}{2}} \kappa\(|\xi|\)^2.
$$
This gives
$$
I\les
\Vert u\Vert_{\dot H^{\sigma}}^2\int |\xi|\kappa(|\xi|)^2|\hat h(\xi)|^2 \dxi 
\les \Vert u\Vert_{\dot H^{\sigma}}^2\blA\D^{\frac{1}{2},\phi}h\brA_{L^2}^2,
$$ 
where we have used the equivalence $\phi\sim\kappa$ (see~\e{n60b}). 
This implies \eqref{Z1}, which completes the proof.
\end{proof}

\section{Preliminaries about the Muskat equation}\label{S:pMe}
In this section, we recall some simple properties 
about the nonlinearity in the Muskat equation. 
This will serve to rigorously justify the computations in the sequel. 
Firstly, this will clarify the meaning of the integrals with respect to the variable $\alpha$. 
And secondly, we will exploit these results to introduce 
approximate Cauchy problems which are well-posed globally in time. 

\subsection{About the nonlinearity}\label{S:3.4}

Write
\be\label{ae20}
\frac{\partial_x\Delta_\alpha f}{1+\left(\Delta_\alpha f\right)^2}=\partial_x\Delta_\alpha f
-\partial_x\Delta_\alpha f\frac{\left(\Delta_\alpha f\right)^2}{1+\left(\Delta_\alpha f\right)^2}\cdot
\ee
Notice that the Hilbert transform (see~\e{n3}) satisfies
$$
\mathcal{H}u(x)=-\frac{1}{\pi}\mathrm{pv}\int_\xR\Delta_\alpha u\dalpha.
$$
In addition, the fractional laplacian $\D=(-\partial_{xx})^{\mez}$ satisfies $\D=\partial_x\mathcal{H}$. 
Consequently, it follows from~\e{ae20} 
that
\begin{align}\label{Tf1}
\partial_tf+\D f = \mathcal{T}(f)f,
\end{align}
where $\mathcal{T}(f)$ is the operator defined by
\be\label{def:T(f)f}
\mathcal{T}(f)g = -\frac{1}{\pi}\int_\xR\left(\partial_x\Delta_\alpha g\right)
\frac{\left(\Delta_\alpha f\right)^2}{1+\left(\Delta_\alpha f\right)^2}\dalpha.
\ee
Then we have the following elementary result, which is Proposition~$2.3$ in~\cite{Alazard-Lazar}.
\begin{proposition}[from~\cite{Alazard-Lazar}]\label{P:continuity}
\begin{enumerate}[i)]
\item\label{Prop:low1} For all $f$ in $\dot{H}^{1}(\xR)$ and $g$ in $\dot{H}^{\tdm}(\xR)$, 
the function
$$
\alpha\mapsto \Delta_\alpha g_x \ \frac{(\Delta_\alpha f)^2}{1+(\Delta_\alpha f)^2}
$$
belongs to $L^1_\alpha(\xR;L^2_x(\xR))$. Therefore, 
$\mathcal{T}(f)g$ belongs to $L^2(\xR)$. 
Moreover, there is a positive constant $C$ such that
\be\label{nTL2}
\lA \mathcal{T}(f)g\rA_{L^2}\le C \lA f\rA_{\dot{H}^1}\lA g\rA_{\dot{H}^\tdm}.
\ee

\item \label{Prop:low3} For all $\delta\in [0,1/2)$, there is a positive 
constant $C$ such that, 
for all functions $f_1,f_2$ in $\dot{H}^{1-\delta}(\xR)\cap \dot{H}^{\tdm+\delta}(\xR)$, 
\be\label{pl:2}
\lA (\mathcal{T}(f_1)-\mathcal{T}(f_2))f_2\rA_{L^2}
\le C \lA f_1-f_2\rA_{\dot{H}^{1-\delta}}\lA f_2\rA_{\dot{H}^{\tdm+\delta}}.
\ee
In particular,
$$
\lA \mathcal{T}(f)f \rA_{L^2}
\le C \lA f\rA_{\dot{H}^{\frac34}}\lA f_2\rA_{\dot{H}^{\frac74}}.
$$
\item \label{Prop:low2} The map $f\mapsto \mathcal{T}(f)f$ is locally Lipschitz from 
$\dot{H}^1(\xR)\cap \dot{H}^{\tdm}(\xR)$ to $L^2(\xR)$.
\end{enumerate}
\end{proposition}

\subsection{Approximate Cauchy problems}
We now take up the second question raised at the beginning of this section, 
that of defining approximate Muskat equations. The goal is to introduce equations 
for which the Cauchy problem is easily studied, 
and whose solutions are expected to converge to solutions of the original Muskat equation. 
There are many possible ways to do so. 
The next proposition introduces a parabolic regularization of the Muskat equation, depending on some 
viscosity parameter $\eps>0$, which is convenient for our purposes. 
The following result states that these equations are well-posed for any fixed~$\eps>0$. 
Our main task in the next sections will be to obtain estimates which are uniform in~$\eps$. 
\begin{proposition}\label{P:3.6}
Consider an initial data $f_0$ in $H^\tdm(\xR)$. For any $\eps>0$, 
the Cauchy problem
\begin{align}\label{main:acp}
\partial_tf+\varepsilon\D^{8}f+\D f = \mathcal{T}(f)f\quad;\quad 
f\arrowvert_{t=0}=f_0,
\end{align}
has a unique solution $f$ in $C^0([0,+\infty);H^\tdm(\xR))\cap L^2([0,+\infty);\dot{H}^2(\xR))$. 
In addition $f$ is $C^\infty$ on $(0,+\infty)\times\xR$.
\end{proposition}
\begin{proof}
Given an initial data $f_0$ in $H^\tdm(\xR)$, the existence of a unique solution $f$  locally in time 
follows at once from Proposition~\ref{P:continuity}, by applying the classical methods for 
nonlinear parabolic equations. 
Then, to prove that the solution $f$ exists globally in time, it will suffice to estimate its 
$H^{\tdm}(\xR)$ norm. 

We recall first that the square of the $L^2$-norm is a Lyapunov functional. Indeed, 
by \cite[Section 2]{CCGRPS-JEMS2013}, we have
$$
\int_\xR\left(\D f-\mathcal{T}(f)f\right)f(x)\dx
=\iint_{\xR^2}\log\left(\sqrt{1+\frac{(f(t,x)-f(t,x-\alpha))^2}{\alpha^2}}\right)\dx\dalpha\ge 0.
$$
This implies that
\begin{equation}\label{nac:1}
\mez\fract \lA f\rA_{L^2}^2+\varepsilon \lA f\rA_{\dot H^{4}}^2\leq 0. 
\end{equation}
Now we multiply the equation \e{main:acp} by $\D^3$ and 
use Plancherel's identity to obtain the following Sobolev energy estimate:
$$
\mez\fract \lA f\rA_{\dot H^{\tdm}}^2+\varepsilon \lA f\rA_{\dot H^{\frac{11}{2}}}^2
+ \lA f\rA_{\dot H^{2}}^2\leq   \lA f\rA_{\dot H^{3}} \lA \mathcal{T}(f)f\rA_{L^2}.
$$
Now, apply~\e{nTL2} to obtain 
$$
\mez\fract \lA f\rA_{\dot H^{\tdm}}^2+\varepsilon \lA f\rA_{\dot H^{\frac{11}{2}}}^2
+ \lA f\rA_{\dot H^{2}}^2\leq   C\lA f\rA_{\dot H^{3}}  \lA f\rA_{\dot{H}^{1}}\lA f\rA_{\dot{H}^{\frac32}}.
$$
One may estimate the right-hand side above by means of the classical 
interpolation inequality in Sobolev spaces. This gives
$$
\mez\fract \lA f\rA_{\dot H^{\tdm}}^2+\varepsilon \lA f\rA_{\dot H^{\frac{11}{2}}}^2
+ \lA f\rA_{\dot H^{2}}^2\les 
\lA f\rA_{L^2}^2\lA f\rA_{\dot H^{\frac{11}{2}}}.
$$
Hence, using Young's inequality,
\be\label{nac:2}
\fract \lA f\rA_{\dot H^{\tdm}}^2\les \varepsilon^{-1}
\lA f\rA_{L^2}^4.
\ee
By combining \e{nac:1} and \e{nac:2}, we get the wanted 
Sobolev estimates, which implies that, for any fixed $\eps>0$, the solution exists globally in time. 
\end{proof}

\subsection{A reformulation of the Muskat equation}

In this paragraph, we recall a key decomposition of the Muskat equation, as introduced in~\cite{Alazard-Lazar}. 

\begin{proposition}[from~\cite{Alazard-Lazar}]\label{P:4.1}
$i)$ The Muskat equation~\e{a1} can be written under the form
\be\label{aW}
\partial_t f+W(f)\partial_x f+\frac{1}{1+f_x^2}\D f=R(f)f,
\ee
where the coefficient $W(f)$ the operator $R(f)$ are defined by
\be\label{defi:V}
\begin{aligned}
W(f)&\defn\frac{1}{\pi}\int_\xR\frac{\mathcal{O}\left(\alpha,.\right)}{\alpha}\dalpha,\\
R(f)g&\defn-\frac{1}{\pi}\int_\xR\left(\partial_x\Delta_\alpha g\right)
\left(\mathcal{E}\left(\alpha,\cdot\right)-\frac{(\partial_xf)^2}{1+(\partial_xf)^2}\right)\dalpha\\
&\quad+\frac{1}{\pi}\int_\xR\frac{\partial_x g(\cdot-\alpha)}{\alpha}\mathcal{O}\left(\alpha,\cdot\right) \dalpha,
\end{aligned}
\ee
with 
\begin{align}
&\mathcal{O}\left(\alpha,\cdot\right)
= \frac{1}{2}\frac{\left(\Delta_\alpha f\right)^2}{1+\left(\Delta_\alpha f\right)^2}
-\frac{1}{2}\frac{\left(\Delta_{-\alpha} f\right)^2}{1+\left(\Delta_{-\alpha} f\right)^2},\label{Oss}\\&
\mathcal{E}\left(\alpha,\cdot\right) = \frac{1}{2}\frac{\left(\Delta_\alpha f\right)^2}{1+\left(\Delta_\alpha f\right)^2}
+\frac{1}{2}\frac{\left(\Delta_{-\alpha} f\right)^2}{1+\left(\Delta_{-\alpha} f\right)^2}\cdot\label{Ess}
\end{align}
$ii)$ Similarly, the approximate Muskat problem~\e{main:acp} is equivalent to
\begin{align}\label{main:acp-bis}
\partial_tf+\varepsilon\D^{8}f+W(f)\partial_x f+\frac{1}{1+f_x^2}\D f=R(f)f.
\end{align}
\end{proposition}
\begin{proof}
Recall from~\S\ref{S:3.4} that the Muskat equation \e{a1} reads
$$
\partial_tf+\D f = \mathcal{T}(f)f,
$$
where $\mathcal{T}(f)$ is the operator defined by
$$
\mathcal{T}(f)g = -\frac{1}{\pi}\int_\xR\left(\partial_x\Delta_\alpha g\right)
\frac{\left(\Delta_\alpha f\right)^2}{1+\left(\Delta_\alpha f\right)^2}\dalpha.
$$
Then splits the coefficient 
$$
F_\alpha\defn\frac{\left(\Delta_\alpha f\right)^2}{1+\left(\Delta_\alpha f\right)^2}
$$
into its odd and even components $\mathcal{O}\left(\alpha,\cdot\right)$ 
and $\mathcal{E}\left(\alpha,\cdot\right)$, as defined in the statement of the proposition. 
Since $\Delta_\alpha f(x)$ converges to $f_x(x)$ when $\alpha$ 
goes to $0$, we further 
decompose $\mathcal{E}\left(\alpha,\cdot\right)$ as
$$
\mathcal{E}\left(\alpha,\cdot\right) =\frac{(\partial_xf)^2}{1+(\partial_xf)^2}
+\left(\mathcal{E}\left(\alpha,\cdot\right) -\frac{(\partial_xf)^2}{1+(\partial_xf)^2}\right).
$$
It follows that
\begin{equation}\label{Tf2}
\mathcal{T}(f)g=\frac{(\partial_xf)^2}{1+(\partial_xf)^2}\D g-W(f)\partial_x g+R(f)g,
\end{equation}
which implies the wanted result~\e{aW}. The same argument proves 
also the second statement.
\end{proof}

The main application that we will work out of the previous formulation 
is an $L^2$-energy estimate for the 
weighted fractional derivatives $\D^{\tdm,\phi}f$.
\begin{corollary}\label{C:5.2}
Consider a function $\phi$ defined by~\e{n10}, for some admissible weight~$\kappa$. 
For any $\eps\in [0,1]$ and for any smooth solution $f$ of the approximate 
Muskat equation~\e{main:acp}, 
there holds
\begin{align*}
\frac{1}{2}\fract \blA D^{\tdm,\phi}f\brA_{L^2}^2
+ \int_\xR \frac{\bla \D^{2,\phi}f\bra^2}{1+(\partial_x f)^2} \dx+\eps\lA f\rA_{\dot{H}^4}^2&=
(I)+(II)+(III)\quad \text{with}\\[1ex]
(I)&\defn-\frac{1}{2} \Big\langle \left[\mathcal{H}, W(f)\right ] \D^{2,\phi}f,\D^{2,\phi}f\Big\rangle,\\[1ex]
(II)&\defn\big\langle R(f)(\D^{1,\phi}f), \D^{2,\phi}f\big\rangle,\\[1ex]
(III)&\defn\Big\langle\left[\D^{1,\phi},\mathcal{T}(f)\right]f, \D^{2,\phi}f\Big\rangle,
\end{align*}
where $[A,B]$ denotes the commutator $A\circ B-B\circ A$ of two operators.
\end{corollary}
\begin{remark}\label{R:dpb}
$i)$ Hereafter, we say that $f=f(t,x)$ is a smooth function if it satisfies the conclusion 
of Proposition~\ref{P:3.6}, that is:
$$
f\in C^0([0,+\infty);H^\tdm(\xR))\cap L^2([0,+\infty);\dot{H}^2(\xR))\cap 
C^\infty\big((0,+\infty)\times\xR\big).
$$
$ii)$ As explained in \S\ref{S:1.2}, observe that one does not control $\int \bla \D^{2,\phi}f\bra^2 \dx$. 
Instead, we merely control
$$
\int_\xR \frac{\bla \D^{2,\phi}f\bra^2}{1+(\partial_x f)^2} \dx.
$$
The key point is that 
$1+(\partial_xf)^2$ is not bounded from above for solutions $f$ which are $L^\infty$ in times 
with values in $H^\tdm(\xR)$. 
This is why the parabolic feature of the Muskat equation degenerates in the critical regime.
\end{remark}
\begin{proof}
By commuting $\D^{1,\phi}$ with the equation~\e{Tf1} and then using the 
decomposition~\e{Tf2} with $g=\D^{1,\phi}f$, we find that
\be\label{aW2}
\Big(\partial_t +W(f)\partial_x +\frac{1}{1+f_x^2}\D+\eps\D^8
-R(f)\Big)\D^{1,\phi}f=\big[ \D^{1,\phi},\mathcal{T}(f)\big]f.
\ee
Then multiply the equation~\e{aW2} by $\D^{2,\phi}f$ and use 
Plancherel's identity to write
\begin{equation*}
\big\langle\partial_t \D^{1,\phi}f,\D^{2,\phi} f\big\rangle
=\frac{1}{2}\fract \blA \D^{\tdm,\phi}f\brA_{L^2}^2.
\end{equation*}
If follows that
\begin{multline*}
\frac{1}{2}\fract \blA D^{\tdm,\phi}f\brA_{L^2}^2+\eps\lA f\rA_{\dot{H}^4}^2
+ \int_\xR \frac{\bla \D^{2,\phi}f\bra^2}{1+(\partial_x f)^2} \dx\\
=-\langle W(f)\partial_x \D^{1,\phi}f, \D^{2,\phi}f\rangle+
(II)+(III),
\end{multline*}
where $(II)$ and $(III)$ are as defined in the statement. 

To conclude, we notice that, since $\partial_x=-\mathcal{H}\D$ and 
since $\mathcal{H}^*=-\mathcal{H}$, one has 
\be\label{w1}
\begin{aligned}
\langle W(f)\partial_x \D^{1,\phi}f, \D^{2,\phi}f\rangle&= \frac{1}{2} \Big\langle \left[\mathcal{H}, W(f)\right ] \D^{2,\phi}f,\D^{2,\phi}f\Big\rangle=-(I).
\end{aligned}
\ee
This completes the proof.
\end{proof}

\section{Paralinearization of the Muskat equation}\label{S:5}
With these preliminaries established, we start the analysis of the 
nonlinearity in the Muskat equation. 
More precisely, we want to estimate the three terms which appear in Corollary~\ref{C:5.2}, 
in terms 
of the quantities
$$
\sup_{t\in [0,T]}\blA \D^{\tdm,\phi}f(t)\brA_{L^2}^2,\qquad
\int_0^{T}\!\!\int_\xR \frac{\bla \D^{2,\phi}f(t,x)\bra^2}{1+(\partial_xf(t,x))^2} \dx\dt.
$$
To achieve this goal, as the title of this section attempts to indicate, 
we shall develop several tools needed to commute the 
weighted fractional Laplacians with the Muskat equation. 
As we have explained in the introduction, the main difficulty 
is to exhibit a null-type structure. 

Let us introduce three notations which will be used 
continually in this section. 

\begin{enumerate}[(i)]
\item We use the bracket notation:
\begin{equation*}
\L{a}=\sqrt{1+a^2}.
\end{equation*}
To compensate for the degeneracy of the parabolic behavior 
for large slopes mentioned in Remark~\ref{R:dpb}, 
we want to prove estimates which involves the $L^2(\dx/\L{f_x})$-norm (instead of the 
$L^2(\dx)$-norm).
\item 
A key ingredient of our analysis is to decompose the frequency space. 
To do so, 
we will use the following notation. Given a parameter $\lambda>0$, we set
\be\label{defi:Hlam}
\lA f\rA_{\dot{H}^{\tdm}_\lambda}
\defn\left(\int_{\xR} \min\{\lam|\xi|,1\}^{\frac{1}{100}} |\xi|^{3} |\hat f(\xi)|^2\dxi \right)^{\frac{1}{2}}.
\ee
\item 
To truncate the frequency domain, it will be convenient to consider a smooth 
bump function $\chi\colon\xR\to [0,1]$ satisfying $\chi(x)=\chi(-x)$ and 
such that $\chi(x)=1$ for $x\in [-1,1]$ and $\chi(x)=0$ if $\la x\ra\ge 2$.  
\end{enumerate}

\subsection{The commutator with the Hilbert transform}
Our first result will allow us to estimate the term $(I)$ which appears in Corollary~\ref{C:5.2}.
\begin{proposition}\label{X1}
For any  $\lambda>0$ and any 
smooth functions $f,g,h$, there holds
\be\label{h11}
\begin{aligned}
\left|\langle [\mathcal{H},W(f)](g_x),h \rangle\right|&\lesssim\frac{1}{\sqrt{\lam}}
\Vert f\Vert_{\dot H^{\frac{3}{2}}} \Vert g\Vert_{\dot H^{\frac{1}{2}}} \Vert h\Vert_{L^2}\\
&\quad+ \left(1+\Vert f\Vert_{\dot H^{\frac{3}{2}}}^{7}\right)\lA f\rA_{\dot{H}^{\tdm}_\lambda}
\lA\frac{g_x}{\L{f_x}}\rA_{L^2}\lA\frac{h}{\L{f_x}}\rA_{L^2}.
\end{aligned}
\ee
\end{proposition}
\begin{remark}
Assume that $\phi$ is as defined in~\e{n10} for some admissible weight 
$\kappa$. Then it follows from \e{est:kappa3} and the previous inequality that
\begin{align*}
\left|\langle [\mathcal{H},W(f)](g_x),h \rangle\right|&\lesssim \frac{1}{\sqrt{\lam}}
\Vert f\Vert_{\dot H^{\frac{3}{2}}} \Vert g\Vert_{\dot  H^{\frac{1}{2}}} \Vert h\Vert_{L^2}\\
&+ \kappa\Big(\frac{1}{\lam}\Big)^{-1}\left(1+\Vert f\Vert_{\dot H^{\frac{3}{2}}}^{7}\right)
\blA \D^{\tdm,\phi}f\brA_{L^2}\lA\frac{g_x}{\L{f_x}}\rA_{L^2}\lA\frac{h}{\L{f_x}}\rA_{L^2}.
\end{align*}
\end{remark}
\begin{proof}
One can write the function $\mathcal{O}(\alpha,x)$ (see~\e{Oss}) 
as follows:
\begin{equation}\label{X2}
\begin{aligned}
&\mathcal{O}\left(\alpha,\cdot\right) = 
A_\alpha(f)\left(\Delta_\alpha f-\Delta_{-\alpha} f\right) \quad\text{where}\\
&A_\alpha(f)= \frac{1}{2}
\frac{\Delta_\alpha f+\Delta_{-\alpha} f}{(1+\left(\Delta_\alpha f\right)^2)(1+\left(\Delta_{-\alpha} f\right)^2)}\cdot
\end{aligned}
\end{equation}
Then, by definition of $W(f)$ (see~\e{defi:V}), for any $\lambda>0$, one has 
$W(f)=W_{1,\lambda}(f)+W_{2,\lambda}(f)$ with
\begin{align*}
&W_{1,\lam}(f)= \frac{1}{\pi}\int_\xR\frac{\mathcal{O}\left(\alpha,.\right)}{\alpha}\chi\left(\frac{\alpha}{\lam}\right)\dalpha, \\
&W_{2,\lam}(f)=\frac{1}{\pi}\int_\xR\frac{\mathcal{O}\left(\alpha,.\right)}{\alpha}\(1-\chi\left(\frac{\alpha}{\lam}\right)\)\dalpha.
\end{align*}
Then, the Cauchy-Schwarz inequality implies that
\be\label{X3}
\begin{aligned}
\left|\langle [\mathcal{H},W(f)](g_x),h \rangle\right|
&\le \lVert \L{f_x}[\mathcal{H},W_{1,\lam}(f)](g_x)\rVert_{L^2} \lVert\frac{h}{\L{f_x}}\rVert_{L^2}\\
&\quad+\lVert[\mathcal{H},W_{2,\lam}(f)](g_x)\rVert_{L^2} \lVert h\rVert_{L^2}.
\end{aligned}
\ee
We claim that 
\begin{equation}\label{X4}
\lVert[\mathcal{H},W_{2,\lam}(f)](g_x)\rVert_{L^2}\lesssim\frac{1}{
\sqrt{\lam}}  \lVert f\rVert_{\dot H^{\frac{3}{2}}} \lVert g\rVert_{\dot H^{\frac{1}{2}}},
\end{equation}
and 
\begin{equation}\label{X4"}
\Vert \L{f_x}[\mathcal{H},W_{1,\lam}(f)](g_x)\Vert_{L^2}\lesssim  \left(1+\Vert f\Vert_{\dot H^{\frac{3}{2}}}^{7}\right)
\lA f\rA_{\dot{H}^{\tdm}_\lambda} \lA\frac{g_x}{ \L{f_x} }\rA_{L^2}.
\end{equation}
The wanted conclusion~\e{h11} 
will follow at once from~\e{X3}, \e{X4} 
and \e{X4"}.

\textit{Step 1: proof of the claim} \e{X4}. 
We use Lemma~\ref{L:Hilbert1} which gives an estimate for the 
commutator of the Hilbert transform and a function in $\dot{H}^1(\xR)$. 
It implies that
\begin{align}\label{X5}
\Vert[\mathcal{H},W_{2,\lam}(f)](g_x)\Vert_{L^2}\leq \Vert W_{2,\lam}(f)\Vert_{\dot H^1}\Vert g\Vert_{\dot H^{\frac{1}{2}}}.
\end{align}
By definition of $W_{2,\lam}(f)$, we have 
$$
\Vert W_{2,\lam}(f)\Vert_{\dot H^1}^2\lesssim 
\int \left(\int_\xR\frac{|\partial_x\mathcal{O}\left(\alpha,x\right)|}{|\alpha|}
\Big(1-\chi\Big(\frac{\alpha}{\lam}\Big)\Big)\dalpha\right)^2\dx.
$$
Now, directly from the definition of $\mathcal{O}$, we have
$$
\la \partial_x\mathcal{O} (x,\alpha)\ra\les \la \Delta_\alpha f_x\ra+\la \Delta_{-\alpha}f_x\ra. 
$$
Observe that, 
by symmetry\footnote{Let us 
mention that we will make extensive of this symmetry to handle 
the contributions of terms involving $\Delta_{-\alpha}$, without explicitly recalling this argument.} 
it is sufficient to consider the contribution of $\Delta_\alpha f_x$. Consequently, 
using the Cauchy-Schwarz inequality together with the bound~\e{an1:n15}, we get
\begin{align*}
\Vert W_{2,\lam}(f)\Vert_{\dot H^1}^2&\lesssim  
\int  \Bigg(\underset{|\alpha|\geq \lam}{\int}\frac{|f_x(x)-f_x(x-\alpha)|}{|\alpha|^2}\dalpha\Bigg)^2\dx
\\&\lesssim \frac{1}{\lambda}
\iint\frac{|f_x(x)-f_x(x-\alpha)|^2}{|\alpha|^2}\dalpha\dx
\\&\lesssim \frac{1}{\lambda}\Vert f\Vert_{\dot H^{\frac{3}{2}}}^2.
\end{align*}
So the desired result \eqref{X4} follows from \eqref{X5} and the previous inequality.

\textit{Step 2: proof of the claim} \e{X4"}. This is the most delicate task. 

Directly from the definition of the Hilbert transform, one has
\begin{align*}
\Vert \L{f_x}[\mathcal{H},W_{1,\lam}(f)](g_x)\Vert_{L^2}^2
=\frac{1}{\pi^2}\int \left(\int a(x,y) \frac{W_{1,\lam}(f)(x)-W_{1,\lam}(f)(y)}{x-y} \tilde{g}(y)\dy\right)^2\dx
\end{align*}
with
$$
\tilde{g}(y)=\frac{g_x(y)}{ \L{f_x(y)} } \quad\text{and}\quad a(x,y)=\L{f_x(x)}\L{f_x(y)}.
$$
By applying 
Holder's inequality 
together with the representation 
formula~\e{b1}, we find
\begin{align}\label{Z1-W}
&\Vert \L{f_x}[\mathcal{H},W_{1,\lam}(f)](g)\Vert_{L^2}^2\\\nonumber
&\les \int \left(\int a(x,y)^5 
\la W_{1,\lam}(f)(x)-W_{1,\lam}(f)(y)\ra^5
\frac{\dy}{|x-y|^{2}}\right)^{\frac{2}{5}} \left(|D|^{-\frac{1}{4}} (|\tilde{g}|^{\frac{5}{4}})(x)\right)^{\frac{8}{5}}\dx\\\nonumber
&\lesssim \left(\iint a(x,y)^5
|W_{1,\lam}(f)(x)-W_{1,\lam}(f)(y)|^5 \frac{\dy\dx}{|x-y|^{2}}\right)^{\frac{2}{5}}
\left( \int \left(|D|^{-\frac{1}{4}} (|\tilde{g}|^{\frac{5}{4}})(x)\right)^{\frac{8}{3}}\dx\right)^{\frac{3}{5}}\\\nonumber
&\lesssim   \left(\iint a(x,y)^5 |W_{1,\lam}(f)(x)-W_{1,\lam}(f)(y)|^5 \frac{\dy\dx}{|x-y|^{2}}\right)^{\frac{2}{5}}
\Vert\tilde{g}\Vert_{L^2}^{2},
\end{align}
where we have used~\e{Riesz} to write 
\begin{equation*}
\blA\D^{-\frac{1}{4}} h\brA_{L^{\frac{8}{3}}}\lesssim 	\Vert h\Vert_{L^{\frac{8}{5}}}.
\end{equation*}
So, it is enough to show that 
\begin{equation}\label{X12}
\iint a(x,y)^5 |W_{1,\lam}(f)(x)-W_{1,\lam}(f)(y)|^5 \frac{\dy\dx}{|x-y|^{2}}
\lesssim  \left(1+\Vert f\Vert_{\dot H^{\frac{3}{2}}}^{35}\right)
\lA f\rA_{\dot{H}^{\tdm}_\lambda}^5.
\end{equation}
One has
\begin{align*}
&a(x,y)|W_{1,\lam}(f)(x)-W_{1,\lam}(f)(y)|\\
&\qquad\leq \frac{a(x,y)}{\pi}\left|\int A_\alpha(f)(x)
\left(S_\alpha f(x)-S_\alpha f(y)\right) \chi\left(\frac{\alpha}{\lam}\right) \frac{\dalpha}{\alpha}\right|\\
&\qquad\quad+ \frac{a(x,y)}{\pi}\left|\int\left(A_\alpha(f)(x)-A_\alpha(f)(y)\right)
 S_\alpha f(y) \chi\left(\frac{\alpha}{\lam}\right) \frac{\dalpha}{\alpha}\right|.
\end{align*}
We further decompose $A_\alpha(f)$ according to 
$A_0(f)+\left(A_\alpha(f)-A_0(f)\right)$, with
$$
A_0=A_\alpha\arrowvert_{\alpha=0}=\frac{f_x}{(1+f_x^2)^2}\cdot
$$
Then
$$
a(x,y)|W_{1,\lam}(f)(x)-W_{1,\lam}(f)(y)|\leq\frac{1}{\pi}\left( I_1(x,y)+I_2(x,y)+I_3(x,y)\right),
$$
where 
\begin{align*}
&I_1(x,y)= a(x,y)A_{0}(f)(x)\left|\int \left(S_\alpha f(x)-S_\alpha f(y)\right)
\chi\left(\frac{\alpha}{\lam}\right)\frac{\dalpha}{\alpha}\right|,\\
&
I_2(x,y)= a(x,y)\int |A_{\alpha}(f)(x)-A_{0}(f)(x)| \left| S_\alpha f(x)-S_\alpha f(y)\right|
\chi\left(\frac{\alpha}{\lam}\right)\frac{\dalpha}{|\alpha|}
,\\&I_3(x,y)= a(x,y)\int\left|A_\alpha(f)(x)-A_\alpha(f)(y)\right|\left| S_\alpha f(y)\right|
\chi\left(\frac{\alpha}{\lam}\right)\frac{\dalpha}{|\alpha|}\cdot
\end{align*}

\textit{Step 2.1: estimate of the contribution of $I_1$}. 
We begin by estimating
$$
\iint_{\xR^2} I_1(x,y)^5\frac{\dx\dy}{|x-y|^2}\cdot
$$

\begin{lemma}
Set
$$
\theta_\lam(\alpha)=\sign(\alpha)\int_{|\alpha|}^\infty\chi\Big(\frac{s}{\lam}\Big)\frac{\ds}{s^2}
$$
and define
\begin{equation*}
\mathcal{H}_\lam h(x)=h\star \theta_\lambda(x)=\int_\xR h(x-\alpha) \theta_\lam(\alpha) \dalpha.
\end{equation*}
Then
$$
\int_\xR S_\alpha f(x)\chi\left(\frac{\alpha}{\lam}\right)\frac{\dalpha}{\alpha}=2\mathcal{H}_\lambda f_x(x).
$$
\end{lemma}
\begin{proof}
Recall that
$$
S_\alpha(f)(x)=\Delta_\alpha f(x)-\Delta_{-\alpha} f(x)=\frac{2f(x)-f(x-\alpha)-f(x+\alpha)}{\alpha}.
$$
By a parity argument, we have
$$
\int_\xR S_\alpha f(x)\chi\left(\frac{\alpha}{\lam}\right)\frac{\dalpha}{\alpha}
=2\int_0^{+\infty}\big(2f(x)-f(x-\alpha)-f(x+\alpha)\big)\chi\left(\frac{\alpha}{\lam}\right)\frac{\dalpha}{\alpha^2}\cdot
$$
Now, on $\xR_+$, by definition of $\theta_\lambda(\alpha)$, we have $\partial_\alpha\theta_\lambda=-\chi(\alpha/\lam)/\alpha^2$. 
Hence, integrating by parts, we obtain
\begin{align*}
\int_\xR S_\alpha f(x)\chi\left(\frac{\alpha}{\lam}\right)\frac{\dalpha}{\alpha}
&=2\int_0^\infty \partial_\alpha \big(2f(x)-f(x-\alpha)-f(x+\alpha)\big)\theta_\lambda(\alpha)\dalpha\\
&=2\int_0^\infty  \big(f_x(x-\alpha)-f_x(x+\alpha)\big)\theta_\lambda(\alpha)\dalpha.
\end{align*}
Since $\theta_\lambda(-\alpha)=-\theta_\lambda(\alpha)$, this gives the wanted formula.
\end{proof}

On the other hand, remembering that
$$
a(x,y)=\L{f_x(x)}\L{f_x(y)} \quad\text{and}\quad A_0(x)=\frac{f_x(x)}{\L{f_x(x)}^4},
$$
we get
\begin{align*}
a(x,y)A_0(x)\le \frac{\L{f_x(y)}}{1+f_x(x)^2}\les \la f_x(x)-f_x(y)\ra +1.
\end{align*}
Consequently, we deduce 
from the definition of $I_1$ and the previous lemma that
$$
I_1(x,y)\lesssim|f_x(x)-f_x(y)|\left|\mathcal{H}_\lam f_x(x)
-\mathcal{H}_\lam f_x(y)\right|+\left|\mathcal{H}_\lam f_x(x)-\mathcal{H}_\lam f_x(y)\right|.
$$
So, using H\"older's inequality and the characterization of the 
fractional Sobolev spaces in terms of 
Gagliardo semi-norms (see~\e{Gagliardo}), we have
$$
\iint I_1(x,y)^5\frac{\dx\dy}{|x-y|^2}
\lesssim \Vert\mathcal{H}_\lam f_x\Vert_{\dot W^{\frac{1}{10},10}}^5
\Vert f_x\Vert_{\dot W^{\frac{1}{10},10}}^5+ \Vert\mathcal{H}_\lam f_x\Vert_{\dot W^{\frac{1}{5},5}}^5.
$$
Then, using Sobolev's embeddings, we conclude that
\be\label{r867}
\iint I_1(x,y)^5\frac{\dx\dy}{|x-y|^2}\lesssim ( \Vert f\Vert_{\dot H^{\frac{3}{2}}}^{5} +1)
\Vert\mathcal{H}_\lam f_x\Vert_{\dot H^{\frac{1}{2}}}^{5}.
\ee

It remains to estimate the $H^{\frac{1}{2}}$-norm of $\mathcal{H}_\lam f_x$. 
To do so, we will use the following inequality.
\begin{lemma}
There exists a constant $C>0$ such that, for any $\lambda$,
\begin{equation}\label{C1}
|\hat \theta_\lam(\xi)|\lesssim \min\{1,|\xi|\lam\}.
\end{equation}
\end{lemma}
\begin{proof}
One has
\begin{align*}
\bla\hat \theta_\lam(\xi)\bra
&=|\int \theta_\lam(\alpha)\sin(\alpha\xi) \dalpha|\\
&\leq\left| \int \chi(\xi \alpha) \theta_\lam(\alpha) \sin(\alpha\xi) \dalpha\right|
+\left| \int (1- \chi(\xi \alpha) )\theta_\lam(\alpha)\sin(\alpha\xi) \dalpha\right|.
\end{align*} 
Since $|\sin(\alpha\xi)|\leq |\alpha \xi|$ and $\xi \sin(\alpha\xi)=-\partial_\alpha\cos(\alpha \xi)$, 
\begin{align*}
|\hat \theta_\lam(\xi)|&\leq  \int_{|\xi \alpha|\leq 2} \la\theta_\lam(\alpha)\ra 
\la \alpha\xi\ra\dalpha+\frac{1}{|\xi|}\left| \int \partial_\alpha
\left[(1- \chi(\xi \alpha) )\theta_\lam(\alpha)\right]\cos(\alpha\xi) \dalpha\right|.
\end{align*} 
Now, using 
\begin{equation*}
|\theta_\lam(\xi)|\leq \frac{\mathbf{1}_{|\alpha|\leq 2\lam}}{|\alpha|},~~	|\theta_\lam'(\xi)|
\leq \frac{\mathbf{1}_{|\alpha|\leq 2\lam}}{|\alpha|^2},
\end{equation*}
one obtains,
\begin{align*}
|\hat \theta_\lam(\xi)|&\lesssim  \int_{|\xi \alpha|\leq 2}
\mathbf{1}_{|\alpha|\leq 2\lam}|\xi|\dalpha+\frac{1}{|\xi|}
\int_{|\xi\alpha|\geq 1} \frac{\mathbf{1}_{|\alpha|\leq 2\lam}}{|\alpha|^2}\dalpha
+ \int_{1\leq |\xi \alpha|\leq 2} \frac{\mathbf{1}_{|\alpha|\leq 2\lam}}{|\alpha|} \dalpha\\
&\lesssim |\xi| \min\left\{\frac{1}{|\xi|},\lam\right\},
\end{align*} 
equivalent to~\eqref{C1}. 
\end{proof}

Since 
$\widehat{\mathcal{H}_\lam f_x}(\xi)=i\xi\hat{f}(\xi)\hat{\theta}_\lambda(\xi)$, 
the above lemma implies that
\begin{equation*}
\Vert\mathcal{H}_\lam f_x\Vert_{\dot H^{\frac{1}{2}}}\leq 
\left(\int |\xi|^{3}(\min\{1,|\xi|\lam\})^2
\bla\hat f(\xi)\bra^2\dxi\right)^{\frac{1}{2}}
\le \lA f\rA_{\dot{H}^{\tdm}_\lambda},
\end{equation*}
where recall that $\lA f\rA_{\dot{H}^{\tdm}_\lambda}$ is defined by \e{defi:Hlam}. 
Therefore, we obtain from~\e{r867} that
\begin{equation}\label{X6}
\iint I_1(x,y)^5\frac{\dx\dy}{|x-y|^2}
\lesssim \big(1+\Vert f\Vert_{\dot H^{\frac{3}{2}}}^{5}\big)
\lA f\rA_{\dot{H}^{\tdm}_\lambda}^{5}.
\end{equation}
This concludes the analysis of this term.

\textit{Step 2.2: estimate of the contribution of $I_2$}. 
We now estimate
$$
\iint I_2(x,y)^5\frac{\dx\dy}{|x-y|^2}\cdot
$$
Recall that
$$
I_2(x,y)= a(x,y)\int \la A_{\alpha}(f)(x)-A_{0}(f)(x)\ra \la S_\alpha f(x)-S_\alpha f(y)\ra
\chi\left(\frac{\alpha}{\lam}\right)\frac{\dalpha}{|\alpha|},
$$
with $a(x,y)=\L{f_x(x)}\L{f_x(y)}$ and
$$
A_\alpha(f)= \frac{1}{2}
\frac{\Delta_\alpha f+\Delta_{-\alpha} f}{(1+\left(\Delta_\alpha f\right)^2)(1+\left(\Delta_{-\alpha} f\right)^2)},
\qquad A_0=\frac{f_x}{(1+f_x^2)^2}\cdot
$$

We begin by estimating the factor $A_{\alpha}(f)(x)-A_{0}(f)(x)$ by means of the following
\begin{lemma}\label{L:5.5}
For any $(x,x_1,x_2)$, the quantity
$$
M=	\left|\frac{x_1+x_2}{2(1+x_1^2)(1+x_2^2)}-\frac{x}{(1+x^2)^2}\right|
$$
satisfies
\begin{equation}\label{X7}
M\lesssim\frac{1}{\L{x}^{2}}\Big(|x_1-x|^2+|x_2-x|^2+|x_1-x|+|x_2-x|\Big).
\end{equation}
\end{lemma}
\begin{proof}
Set $M_0=2(1+x_1^2)(1+x_2^2)(1+x^2)^2$. Then
\be\label{M0}
\begin{aligned}
M&=\la \frac{1}{M_0}\Big((x_1+x_2)(1+x^2)^2-2x(1+x_1^2)(1+x_2^2)\Big)\ra\\
&\le \frac{\la x_1+x_2-2x\ra}{M_0}+\frac{\la (x_1+x_2)x^2-x(x_1^2+x_2^2)\ra}{M_0}
+\frac{\la (x_1+x_2)x^4-2xx_1^2x_2^2\ra}{M_0}\cdot
\end{aligned}
\ee
The first two factors are clearly bounded by $(1+x^2)^{-1}\left(|x_1-x|+|x_2-x|\right)$. 
To bound the last one, set $h_1=x-x_1$, $h_2=x-x_2$ and write
$$
(x_1+x_2)x^4-2xx_1^2x_2^2=x_1x\Big(
x_1 h_2^2+2x_1x_2 h_2+h_1(x_2^2+h_2^2+2x_2h_2)\Big).
$$
This implies that the last term in the right-hand side of~\e{M0} is bounded by the right-hand side of~\eqref{X7}. 
This completes the proof.
\end{proof}

Recall that  $E_\alpha(f)=\Delta_{\alpha}f-f_x$. It follows from the previous lemma that
\begin{equation*}
\la A_{\alpha}(f)-A_{0}(f)\ra
\le \frac{C}{\L{f_x}^2}\Big( 
\la E_\alpha(f)\ra+\la E_\alpha(f)\ra^2+
\la E_{-\alpha}(f)\ra+\la E_{-\alpha}(f)\ra^2\Big).
\end{equation*}
As above, by using the change of variables $\alpha\to -\alpha$ and a parity argument, 
it is sufficient to handle the contributions of $E_{\alpha}f$. We have
\begin{align*}
I_2(x,y)&\lesssim \L{f_x(y)}\L{f_x(x)}^{-1}\int \left(\la E_{\alpha}f(x)\ra 
+ \la E_{\alpha}f(x)\ra^2\right)\left| S_\alpha f(x)-S_\alpha f(y)\right|
\chi\left(\frac{\alpha}{\lam}\right)\frac{\dalpha}{|\alpha|}\\
&\lesssim 
\left( 1+|f_x(x)-f_x(y)|\right)\underset{|\alpha|\leq 2\lam}{\int} \left(\la E_{\alpha}f(x)\ra 
+ \la E_{\alpha}f(x)\ra^2\right)\left|S_\alpha(x)-S_\alpha(y)\right|\frac{\dalpha}{|\alpha|},
\end{align*}
where we have used the fact that $
\L{f_x(y)}\L{f_x(x)}^{-1}\lesssim 1+|f_x(x)-f_x(y)|.$

By H\"older and Minkowski's inequalities, the term $\iint I_2(x,y)^5\frac{\dx\dy}{|x-y|^2}$ 
is bounded by
\begin{align*}
&\Bigg[\underset{|\alpha|\leq 2\lam}{\int}\left(\Vert E_\alpha(f)\Vert_{L^{10}}
+\Vert E_\alpha(f)\Vert_{L^{20}}^2\right)\Vert S_\alpha(f)\Vert_{\dot{F}^{\frac{1}{5}}_{10,5}}
\frac{\dalpha}{|\alpha|} \Bigg]^{5}\\
&+
\Vert f_x\Vert_{\dot W^{\frac{1}{10},10}}^5\Bigg[\underset{|\alpha|\leq 2\lam}{\int}
\left(\Vert E_\alpha(f)\Vert_{L^{20}}+\Vert E_\alpha(f)\Vert_{L^{40}}^2\right)
\Vert S_\alpha(f)\Vert_{\dot{F}^{\frac{1}{10}}_{20,10}}\frac{\dalpha}{|\alpha|} \Bigg]^{5}.
\end{align*}
Hence, the Sobolev embeddings~\e{FB},~\e{Sobolev} and~\e{Gagliardo} imply that 
\begin{align*}
&\Vert h\Vert_{L^{10}}\lesssim \Vert h\Vert_{\dot H^{\frac{2}{5}}},
\quad \Vert h\Vert_{L^{20}}\lesssim \Vert h\Vert_{\dot H^{\frac{9}{20}}},
\quad \Vert h\Vert_{L^{40}}\lesssim \Vert h\Vert_{\dot H^{\frac{19}{40}}},\\
&\Vert f_x\Vert_{\dot W^{\frac{1}{10},10}}
\lesssim\Vert f\Vert_{\dot H^{\frac{3}{2}}},\Vert h\Vert_{\dot{F}^{\frac{1}{10}}_{20,10}}
\leq \Vert h\Vert_{\dot H^{\frac{11}{20}}},\quad
\Vert h\Vert_{\dot{F}^{\frac{1}{5}}_{10,5}}\leq \Vert h\Vert_{\dot H^{\frac{3}{5}}}.
\end{align*}
Consequently, using again H\"older's inequality,
\begin{align*}
&\iint I_2(x,y)^5\frac{\dx\dy}{|x-y|^2}\\
&\lesssim 
\Bigg[\underset{|\alpha|\leq 2\lam}{\int}
\left(\Vert E_\alpha(f)\Vert_{\dot H^{\frac{2}{5}}}^2+\Vert E_\alpha(f)\Vert_{\dot H^{\frac{9}{20}}}^4\right)
\frac{\dalpha}{|\alpha|^{\frac{6}{5}}} \Bigg]^{\frac{5}{2}}
\Bigg[\underset{|\alpha|\leq 2\lam}{\int}\Vert S_\alpha(f)\Vert_{\dot H^{\frac{3}{5}}}^2
\frac{\dalpha}{|\alpha|^{\frac{4}{5}}} \Bigg]^{\frac{5}{2}}\\&+
\Vert f\Vert_{\dot H^{\frac{3}{2}}}^5\Bigg[\underset{|\alpha|\leq 2\lam}{\int}
\left(\Vert E_\alpha(f)\Vert_{\dot H^{\frac{9}{20}}}^2
+\Vert E_\alpha(f)\Vert_{\dot H^{\frac{19}{40}}}^4\right)
\frac{\dalpha}{|\alpha|^{\frac{11}{10}}}\Bigg]^{\frac{5}{2}}
\Bigg[\underset{|\alpha|\leq 2\lam}{\int}\Vert S_\alpha(f)\Vert_{\dot H^{\frac{11}{20}}}^2
\frac{\dalpha}{|\alpha|^{\frac{9}{10}}}\Bigg]^{\frac{5}{2}}.
\end{align*}

So, it follows from Lemma~\ref{L:2.7} that
\begin{equation}\label{X10}
\iint I_2(x,y)^5\frac{\dx\dy}{|x-y|^2}
\lesssim \left(1+\Vert f\Vert_{\dot H^{\frac{3}{2}}}^{10}\right) \lA f\rA_{\dot{H}^{\tdm}_\lambda}^{10}.
\end{equation}

\textit{Step 2.3: Estimate of the contribution of $I_3$}. 
We now pass to the estimate of the last term
$$
\iint_{\xR^2} I_3(x,y)^5\frac{\dx\dy}{|x-y|^2}.
$$
Here we use the following variant of the inequality given by Lemma~\ref{L:5.5}:
\begin{equation*}
\left|\frac{x_1+x_2}{(1+x_1^2)(1+x_2^2)}-
\frac{y_1+y_2}{(1+y_1^2)(1+y_2^2)}\right|
\lesssim
\frac{\sum_{j=1}^{2}|x_j-y_j|+|x_j-y_j|^3}{(1+y_1^2)(1+y_2^2)}.
\end{equation*}
as can be verified by using arguments parallel to those 
used to prove Lemma~\ref{L:5.5}. 

It follows that 
\begin{align*}
I_3(x,y)&\lesssim 
\underset{|\alpha|\leq 2\lam}{\int}\frac{a(x,y)}{\L{\Delta_{\alpha} f(y)}^2}N_\alpha(x,y)
\left|\Delta_\alpha f(y)-\Delta_{-\alpha} f(y)\right|\frac{\dalpha}{|\alpha|},
\end{align*}
where 
\begin{align}\label{X9}
N_\alpha(x,y)= \left|\Delta_\alpha f(x)-\Delta_\alpha f(y)\right|+\left|\Delta_\alpha f(x)-\Delta_\alpha f(y)\right|^3.
\end{align}
Thanks to 
$$\frac{a(x,y)}{\L{\Delta_{\alpha} f(y)}^2}\lesssim (1+\la f_x(x)-f_x(y)\ra )(1+(f_x(y)-\Delta_{\alpha} f(y))^2)$$
it follows that 
\begin{align*}
I_3(x,y)\lesssim
\underset{|\alpha|\leq 2\lam}{\int}N_\alpha(x,y)M_\alpha(y)\frac{\dalpha}{|\alpha|}+
 |f_x(x)-f_x(y)|
 \underset{|\alpha|\leq 2\lam}{\int}N_\alpha(x,y)M_\alpha(y)\frac{\dalpha}{|\alpha|},
\end{align*}
with 
\begin{equation}\label{X8}
M_\alpha(y)=\left(1+(f_x(y)-\Delta_{\alpha} f(y))^2\right) \left|\Delta_\alpha f(y)-\Delta_{-\alpha} f(y)\right|. 
\end{equation}
Once again, we argue using Holder's inequality and Minkowski's inequality, to deduce
\begin{align*}
&\iint I_3(x,y)^5\frac{\dx\dy}{|x-y|^2}\\
&\quad\lesssim 
\Bigg[\underset{|\alpha|\leq 2\lam}{\int} \left(\iint |M_\alpha(y)|^5
\left|N_\alpha(x,y)\right|^5\frac{\dx\dy}{|x-y|^{2}}\right)^{\frac{1}{5}} \frac{\dalpha}{|\alpha|} \Bigg]^{5}\\
&\quad\quad+
\Vert f_x\Vert_{\dot W^{\frac{1}{10},10}}^5\Bigg[\underset{|\alpha|\leq 2\lam}{\int}
\left(\iint |M_\alpha(y)|^{10}\left|N_\alpha(x,y)\right|^{10}\frac{\dx\dy}{|x-y|^{2}}\right)^{\frac{1}{10}}
\frac{\dalpha}{|\alpha|} \Bigg]^{5}\\
&\quad
\lesssim \Bigg[\underset{|\alpha|\leq 2\lam}{\int} \Vert M_\alpha\Vert_{L^{10}}
\left(\int\left(\int \left|N_\alpha(x,y)\right|^5\frac{\dx}{|x-y|^{2}}\right)^2\dy\right)^{\frac{1}{10}}
\frac{\dalpha}{|\alpha|} \Bigg]^{5}\\
&\quad\quad+
\Vert f\Vert_{\dot H^{\frac{3}{2}}}^5\Bigg[\underset{|\alpha|\leq 2\lam}{\int} 
\Vert M_\alpha\Vert_{L^{20}}
\left(\int\left(\int \left|N_\alpha(x,y)\right|^{10}\frac{\dx}{|x-y|^{2}}\right)^2 \dy\right)^{\frac{1}{20}}
\frac{\dalpha}{|\alpha|} \Bigg]^{5}.
\end{align*}
By \eqref{X8} and Sobolev inequality,  for any $p>2$, we have
\begin{align*}
\Vert M_\alpha\Vert_{L^{p}}&\lesssim \Vert\Delta_\alpha f-\Delta_{-\alpha} f\Vert_{L^{p}}
+\Vert\Delta_\alpha f-\Delta_{-\alpha} f\Vert_{L^{2p}}^2+\Vert\Delta_\alpha f-f_x\Vert_{L^{4p}}^4\\
&\lesssim  \Vert\Delta_\alpha f-\Delta_{-\alpha} f\Vert_{\dot H^{\frac{1}{2}(1-\frac{2}{p})}}
+\Vert\Delta_\alpha f-\Delta_{-\alpha} f\Vert_{\dot H^{\frac{1}{2}(1-\frac{1}{p})}}^2\\
&\quad+\Vert\Delta_\alpha f-f_x\Vert_{\dot H^{\frac{1}{2}(1-\frac{1}{2p})}}^4.
\end{align*}
On the other hand, by \eqref{X9}, for any $p>2$, we have
\begin{align*}
\left(\int\left(\int \left|N_\alpha(x,y)\right|^{p}\frac{\dx}{|x-y|^{2}}\right)^2 \dy\right)^{\frac{1}{2p}}
&\lesssim \Vert\Delta_{\alpha} f\Vert_{\dot{F}^{\frac{1}{p}}_{2p,p}}
+\Vert\Delta_{\alpha} f\Vert_{\dot{F}^{\frac{1}{3p}}_{6p,3p}}^3\\
&\lesssim  \Vert\Delta_{\alpha} f\Vert_{\dot H^{\frac{1}{2}\left(1+\frac{1}{p}\right)}}
+\Vert\Delta_{\alpha} f\Vert_{\dot H^{\frac{1}{2}\left(1+\frac{1}{3p}\right)}}^3.
\end{align*}
So, after some extra bookkeeping, by 
applying Lemma~\ref{L:2.7}, we get that
\begin{equation}\label{X11}
\iint I_3(x,y)^5\frac{\dx\dy}{|x-y|^2}
\lesssim \left(1+\Vert f\Vert_{\dot H^{\frac{3}{2}}}^{30}\right)
\lA f\rA_{\dot{H}^{\tdm}_\lambda}^{5}.
\end{equation}
Therefore, we have proved \eqref{X12} which 
completes the proof of the proposition.
\end{proof}

\subsection{Estimate of the remainder term}
We now consider the remainder term $R(f)(g)$ which appears in Proposition~\ref{P:4.1}. 
Our goal is to estimate the term $(II)$ in Corollary~\ref{C:5.2}.
\begin{proposition}\label{P:5.6}
For any $\lambda>0$ and any smooth functions $f$ and $g$, there holds
\begin{equation*}
|\big\langle R(f)(g), h\big\rangle|\lesssim\left(1+\Vert f\Vert_{\dot H^{\frac{3}{2}}}^2\right)
\lA f\rA_{\dot{H}^{\tdm}_\lambda}\lA\frac{g_x}{\L{f_x}}\rA_{L^2}\lA\frac{h}{\L{f_x}}\rA_{L^2}  +\frac{1}{\sqrt{\lam}} \Vert g\Vert_{\dot H^{\mez}} \Vert f\Vert_{\dot H^{\frac{3}{2}}}^2\Vert h\Vert_{L^2}.
\end{equation*}
\end{proposition}
\begin{remark}
In particular, it follows from~\e{est:kappa3} that
\begin{align*}
\la\big\langle R(f)(g), h\big\rangle\ra
&\lesssim\kappa\Big(\frac{1}{\lam}\Big)^{-1}\left(1+\Vert f\Vert_{\dot H^{\frac{3}{2}}}^2\right)\blA\D^{\frac{3}{2},\phi}f\brA_{L^2}  \lA\frac{g_x}{\L{f_x}}\rA_{L^2} \lA\frac{h}{\L{f_x}}\rA_{L^2} \\
&\quad+\frac{1}{\sqrt{\lam}} \Vert g\Vert_{\dot H^{\mez}} \Vert f\Vert_{\dot H^{\frac{3}{2}}}^2\Vert h\Vert_{L^2}.
\end{align*}
\end{remark}
\begin{proof}
 Recall that 
\begin{align*}
R(f)g&\defn-\frac{1}{\pi}\int_\xR\left(\partial_x\Delta_\alpha g\right)
\left(\mathcal{E}\left(\alpha,\cdot\right)-\frac{(\partial_xf)^2}{1+(\partial_xf)^2}\right)\dalpha
+\frac{1}{\pi}\int_\xR\frac{\partial_x g(\cdot-\alpha)}{\alpha}\mathcal{O}\left(\alpha,\cdot\right) \dalpha.
\end{align*}
By symmetry, one has
$$
\int_\xR \frac{g_x(x)}{\alpha}
\left(\mathcal{E}\left(\alpha,\cdot\right)-\frac{(\partial_xf)^2}{1+(\partial_xf)^2}\right)\dalpha=0.
$$
Hence, $R(f)g$ simplifies to
\begin{align*}
R(f)g&=\frac{1}{\pi}\int_\xR \frac{g_x(\cdot-\alpha)}{\alpha}\left(\mathcal{E}\left(\alpha,\cdot\right)
-\frac{(\partial_xf)^2}{1+(\partial_xf)^2}\right)\dalpha+\frac{1}{\pi}\int_\xR\frac{\partial_x g(\cdot-\alpha)}{\alpha}
\mathcal{O}\left(\alpha,\cdot\right) \dalpha\\
&= \frac{1}{\pi}\int_\xR \frac{g_x(\cdot-\alpha)}{\alpha}\left(\frac{\left(\Delta_\alpha f\right)^2}{1+\left(\Delta_\alpha f\right)^2}
-\frac{(\partial_xf)^2}{1+(\partial_xf)^2}\right)\dalpha.
\end{align*}
Introduce the function $B\colon\xR\to\xR$ defined by
$$
B(u)=\frac{u^2}{1+u^2}
$$
and then decompose $R(f,g)=R_{1,\lam}+R_{2,\lam}$ with
\begin{align*}
R_{1,\lam}
&=\frac{1}{\pi}\int_\xR \frac{g_x(\cdot-\alpha)}{\alpha}
\left( B\(\Delta_\alpha f\right)-B(\partial_xf)\right) \chi\Big(\frac{\alpha}{\lambda}\Big)\dalpha,\\
R_{2,\lam}&=\frac{1}{\pi}\int_\xR \frac{g_x(\cdot-\alpha)}{\alpha}\left( B\(\Delta_\alpha f\right)-B(\partial_xf)\right)  \Big(1-\chi\Big(\frac{\alpha}{\lambda}\Big)\Big)\dalpha.
\end{align*}
We have 
\begin{align*}
|\big\langle R(f)g, h\big\rangle|\leq
\Vert \L{f_x}R_{1,\lam}\Vert_{L^2}\lA\frac{h}{\L{f_x}}\rA_{L^2}+
\Vert	R_{2,\lam}\Vert_{L^2}\Vert h\Vert_{L^2}.
\end{align*}
We will estimate these two terms separately. 

\textit{Step 1: estimate of $R_{2,\lam}$.} Note that
$$
g_x(x-\alpha)=\partial_\alpha\big(\delta_{\alpha}g(x)\big).
$$
Then, by integrating by parts in $\alpha$, we get that
\begin{align*}
&|R_{2,\lam}(x)|
\lesssim \underset{|\alpha|\geq \lam}{\int} |\delta_{\alpha}g(x)| \la\Theta_\lam(x,\alpha)\ra\dalpha\quad\text{where}\\
&\Theta_\lam(x,\alpha)\defn
\partial_\alpha
\left[\frac{1}{\alpha}\Big( (B (f_x(x))-B\(\Delta_\alpha f(x)\right))
\Big(1-\chi\Big(\frac{\alpha}{\lambda}\Big)\Big)\Big)\right].
\end{align*}
It follows from the Cauchy-Schwarz inequality that
\begin{align*}
|R_{2,\lam}(x)|^2&\le \bigg(\underset{|\alpha|\geq \lam}{\int}
\sqrt{\frac{|\alpha|}{\lam}}|\delta_{\alpha}g(x)| \la\Theta_\lam(x,\alpha)\ra\dalpha\bigg)^2\\
&
\le \frac{1}{\lam}\bigg(\int_\xR \bigg\vert \frac{\delta_{\alpha}g(x)}{|\alpha|^{\frac14}}
\bigg\vert^2\frac{\dalpha}{|\alpha|}\bigg)
\(\int_\xR\la\alpha\ra^\frac{5}{2} \la\Theta_\lam(x,\alpha)\ra^2\dalpha\right).
\end{align*}
By using again the Cauchy-Schwarz inequality, 
\be\label{r1227}
\lA R_{2,\lam}\rA_{L^2}\le \frac{1}{\sqrt{\lam}}\lA g\rA_{\dot{F}^{\frac14}_{4,2}}
\(\int_\xR\(\int_\xR\la\alpha\ra^\frac{5}{2}
\la\Theta_\lam(x,\alpha)\ra^2\dalpha\right)^2\dx\)^{\frac14}.
\ee
We now have to estimate $\la\Theta_\lam(x,\alpha)\ra$. 

Since $\la B'(u)\ra\le 1$ for all $u\in \xR$, we have
\begin{align*}
&\la B\(\Delta_\alpha f\right)-B(f_x)\ra\le \la \Delta_\alpha f-f_x\ra,\\
&\la \partial_\alpha (B\(\Delta_\alpha f\right)-B(f_x))\ra =\la B'\(\Delta_\alpha f\right)
\partial_\alpha\Delta_\alpha f \ra
\le \la \partial_\alpha \Delta_\alpha f\ra.
\end{align*}
Consequently, 
\begin{align*}
\left|\Theta_\lam(x,\alpha)\right|
&\les \left[\frac{1}{\alpha^2} +\frac{1}{\alpha \lambda}\la\chi'\Big(\frac{\alpha}{\lambda}\Big)\ra \right]
\la \Delta_\alpha f-f_x\ra+\frac{1}{\alpha} \la \partial_\alpha \Delta_\alpha f\ra\\
&\les\frac{1}{\alpha^2}
\la \Delta_\alpha f-f_x\ra+\frac{1}{\alpha} \la \partial_\alpha \Delta_\alpha f\ra,
\end{align*}
where we have used that $\alpha\sim \lam$ on the support of $\chi'(\alpha/\lam)$. 
Now, verify that
\begin{align*}
\partial_\alpha \Delta_\alpha f&=-\frac{f(x)-f(x-\alpha)-\alpha f_x(x-\alpha)}{\alpha^2}\\
&=-\frac{1}{\alpha^2}\int_0^\alpha \big(f_x(x-y)-f_x(x)\big)\dy+\frac{1}{\alpha}\big(f_x(x)-f_x(x-\alpha)\big).
\end{align*}
On the other hand, we have
$$
\frac{1}{\alpha^2}\la \Delta_\alpha f-f_x\ra\le \frac{1}{\alpha^2}\la \frac{1}{\alpha}
\int_0^\alpha \la f_x(x-y)-f_x(x)\ra\dy\ra.
$$
By combining the previous estimates and using $(a+b)^2\le 2a^2+2b^2$, we conclude
\begin{align*}
|\alpha|^\frac{5}{2}\left|\Theta_\lam(x,\alpha)\right|^2&\les 
\frac{1}{|\alpha|^{3/2}}\la\frac{1}{\alpha} \int_0^\alpha \la f_x(x-y)-f_x(x)\ra\dy\ra^2\\
&\quad+\frac{1}{|\alpha|^{3/2}}\la f_x(x)-f_x(x-\alpha)\ra^2.
\end{align*}
Therefore, it follows from Hardy's inequality~\e{Hardy} that
$$
\int_\xR|\alpha|^\frac{5}{2} \la\Theta_\lam(x,\alpha)\ra^2\dalpha \les
\int_\xR \frac{\la \delta_\alpha f_x(x)\ra^2}{|\alpha|^\mez}\frac{\dalpha}{|\alpha|}\cdot
$$
Consequently
$$
\(\int_\xR\(\int_\xR\la\alpha\ra^\frac{5}{2}
\la\Theta_\lam(x,\alpha)\ra^2\dalpha\right)^2\dx\)^{\frac14}\les \lA f_x\rA_{\dot{F}^{\frac14}_{4,2}}.
$$
By plugging this in~\e{r1227}, we get
$$
\lA R_{2,\lam}\rA_{L^2}\le \frac{1}{\sqrt{\lam}}\lA g\rA_{\dot{F}^{\frac14}_{4,2}}\lA f_x\rA_{\dot{F}^{\frac14}_{4,2}}.
$$
Then the Sobolev embedding $\dot{H}^\mez(\xR)\hookrightarrow \dot{F}^{\frac14}_{4,2}(\xR)$ implies that
$$
\lA R_{2,\lam}\rA_{L^2}\le
\frac{1}{\sqrt{\lam}} \Vert g\Vert_{\dot H^{\mez}} \Vert f\Vert_{\dot{H}^{\frac{3}{2}}}^2.
$$
	
\textit{Step 2: estimate of $R_{1,\lam}$.} Notice that 
\be\label{r961}
B(u)-B(v)=\frac{u^2-v^2}{(1+u^2)(1+v^2)}=\frac{(u-v)(u-v+2v)}{(1+u^2)(1+v^2)}\le
\frac{\la u-v\ra+\la u-v\ra^2}{\L{u}^2}\cdot
\ee
Then, since $E_\alpha(f)=\Delta_\alpha f-f_x$, 
$$
\left|R_{1,\lam}(x)\right|\lesssim \underset{|\alpha|\leq 2\lam}{\int}\la g(x-\alpha)\ra 
\frac{\la E_\alpha(f)(x)\ra^2+\la E_\alpha(f)(x)\ra}{\L{f_x(x)}^2}
\frac{\dalpha}{|\alpha|}\cdot
$$
Now introduce the function $\tilde{g}={|g_x|}/{\L{f_x}}$. 
Directly from the previous inequality, we have
\begin{align*}
\left|\L{f_x(x)}R_{1,\lam}(x)\right|\lesssim 
\underset{|\alpha|\leq 2\lam}{\int}\L{f_x(x-\alpha)}\la\tilde{g}(x-\alpha)\ra
\frac{\la E_\alpha(f)(x)\ra^2+\la E_\alpha(f)(x)\ra}{\L{f_x(x)}}
\frac{\dalpha}{|\alpha|}\cdot
\end{align*}

Since 
\begin{equation*}
 \frac{\L{f_x(x-\alpha)}}{\L{f_x(x)}}\leq 1+ \la\frac{\L{f_x(x-\alpha)}}{\L{f_x(x)}}-1\ra\lesssim 
 1+\la f_x(x)-f_x(x-\alpha)\ra,
\end{equation*}
we have
\begin{align*}
\left|\L{f_x(x)}R_{1,\lam}(x)\right|
\lesssim 
\underset{|\alpha|\leq 2\lam}{\int}|\tilde{g}(x-\alpha)| 	W_\alpha(x) \frac{\dalpha}{|\alpha|},
\end{align*}
with 
\begin{align*}
W_\alpha(x)=\big(1+|\delta_{\alpha}f_x(x)|\big)\big(|E_\alpha(f)(x)|^2+|E_\alpha(f)(x)|\big).
\end{align*}
Using H\"older's inequality, we infer that
$$
\left|\L{f_x(x)}R_{1,\lam}(x)\right|
\lesssim 
\left(|D|^{-\frac{1}{4}} (\tilde{g}^{\frac{3}{2}})(x)\right)^{\frac{2}{3}}
\Bigg(\underset{|\alpha|\leq 2\lam}{\int}
W_\alpha(x)^{3} \frac{\dalpha}{|\alpha|^{\frac{3}{2}}}\Bigg)^{\frac{1}{3}}\cdot
$$
Using successively the H\"older inequality, the boundedness of Riesz potentials on Lebesgue spaces (see~\e{Riesz})  and the Minkowski's inequality, we obtain that
\begin{align*}
\Vert \L{f_x}R_{1,\lam}\Vert_{L^2}
&\lesssim 
\left(\int_\xR   \left(|D|^{-\frac{1}{4}} (\tilde{g}^{\frac{3}{2}})(x)\right)^{2}\dx\right)^{\frac{1}{3}}
\Bigg(\int_\xR  \Bigg(	\underset{|\alpha|\leq 2\lam}{\int}
W_\alpha(x)^{3} \frac{\dalpha}{|\alpha|^{\frac{3}{2}}}\Bigg)^{2}\dx\Bigg)^{\frac{1}{6}}
\\&\lesssim 
\Vert\tilde{g}\Vert_{L^2}  \Bigg(\underset{|\alpha|\leq 2\lam}{\int} 
\left(	\int	W_\alpha(x)^{6}\dx \right)^{\frac{1}{2}} \frac{\dalpha}{|\alpha|^{\frac{3}{2}}}\Bigg)^{\frac{1}{3}}.
\end{align*}
Since
\begin{align*}
\int_\xR	W_\alpha(x)^{6}\dx &\lesssim  
\int_\xR \la E_\alpha(f)(x)\ra^{6}\dx+ \int_\xR
|E_\alpha(f)(x)|^{18}\dx+\int_\xR|\delta_\alpha f_x(x)|^{18}\dx,
\end{align*}
we conclude that
\begin{align*}
\int_\xR	W_\alpha(x)^{6}\dx &\lesssim\Vert 
E_\alpha( f)\Vert_{\dot H^{\frac{1}{3}}}^6
+\Vert E_\alpha(f)\Vert_{\dot H^{\frac{4}{9}}}^{18}+ \Vert \delta_\alpha (f_x)\Vert_{\dot H^{\frac{4}{9}}}^{18}.
\end{align*}
Therefore, 
$$
\Vert \L{f_x}R_{1,\lam}\Vert_{L^2}
\lesssim 
\Vert\tilde{g}\Vert_{L^2}
\Bigg(\underset{|\alpha|\leq 2\lam}{\int}\Big[ \Vert E_\alpha(f)\Vert_{\dot H^{\frac{1}{3}}}^3
+\Vert E_\alpha(f)\Vert_{\dot H^{\frac{4}{9}}}^{9}+
\Vert\delta_\alpha( f_x)\Vert_{\dot H^{\frac{4}{9}}}^9\Big] \frac{\dalpha}{|\alpha|^{\frac{3}{2}}}\Bigg)^{\frac{1}{3}}.
$$
We are now in position to apply Lemma~\ref{L:2.7}, which implies that 
\begin{align*}
\Vert \L{f_x}R_{1,\lam}\Vert_{L^2}\lesssim 
\Vert\tilde{g}\Vert_{L^2}\left(1+\Vert f\Vert_{\dot H^{\frac{3}{2}}}^2\right)
\lA f\rA_{\dot{H}^\tdm_\lambda}.
\end{align*}
This completes the proof.
\end{proof}

\subsection{Commutator with the nonlinearity}
Eventually, we want to estimate the term $(III)$ which appears in Corollary~\ref{C:5.2} 
(this the most delicate step). The next proposition contains the key estimate 
which will allow us to commute arbitrary weighted fractional 
Laplacians with~$\mathcal{T}(f)$.

\begin{proposition}\label{P:5.7}
Assume that $\phi$ is as defined in~\e{n10} for some admissible weight~$\kappa$. 
Then, for any $\lambda>0$ and any smooth functions $f$ and~$g$, there holds
\begin{align}\label{n18}
&\la\left\langle \big[\D^{1,\phi},\mathcal{T}(f)\big](g),h\right\rangle\ra\\
&\quad
\lesssim
\left( 1+\blA\D^{\frac{3}{2},\phi} f\brA_{L^2}^{5}\right)  \left(\brA\D^{\frac{3}{2},\phi}g\brA_{L^2}
\lA\frac{f_{xx}}{\L{f_x}}\rA_{L^2}
+\brA\D^{\frac{3}{2},\phi}f\brA_{L^2}\lA\frac{g_{xx}}{\L{f_x}}\rA_{L^2}\right)
\lA\frac{h}{\L{f_x}}\rA_{L^2}.\nonumber
\end{align}
\end{proposition}
\begin{proof}
In view of the obvious inequality
\begin{align*}
\left|\left\langle \big[\D^{1,\phi},\mathcal{T}(f)\big]g,h\right\rangle \right|
&\leq \blA \L{f_x}\big[\D^{1,\phi},\mathcal{T}(f)\big]g\brA_{L^2}\lA\frac{h}{\L{f_x}}\rA_{L^2},
\end{align*}
it is enough to prove that 
\begin{align}\label{goal}
&
\blA\L{f_x}\big[\D^{1,\phi},\mathcal{T}(f)\big](g)\brA_{L^2}\\
&\quad\quad\lesssim
\left( 1+\blA\D^{\frac{3}{2},\phi} f\brA_{L^2}^{5}\right)  \left(\brA\D^{\frac{3}{2},\phi}g\brA_{L^2}
\lA\frac{f_{xx}}{\L{f_x}}\rA_{L^2}+\brA\D^{\frac{3}{2},\phi}f\brA_{L^2}
\lA\frac{g_{xx}}{\L{f_x}}\rA_{L^2}\right).\nonumber
\end{align}

Recall that the operator $\mathcal{T}(f)$ is defined by
$$
\mathcal{T}(f)g
 = -\frac{1}{\pi}\int_\xR\left(\Delta_\alpha g_x\right)F_\alpha\dalpha\quad\text{where}\quad
 F_\alpha=\frac{\left(\Delta_\alpha f\right)^2}{1+\left(\Delta_\alpha f\right)^2}\cdot
$$
Let us introduce 
$$
\Gamma_\alpha\defn\D^{1,\phi}\left[F_\alpha\Delta_\alpha g_x \right]-F_\alpha\D^{1,\phi}
\left[\Delta_\alpha g_x\right]-\Delta_\alpha g_x\D^{1,\phi}\left[F_\alpha\right].
$$
With this notation, we have 
$$
\left[\D^{1,\phi},\mathcal{T}(f)\right](g)
=  -\frac{1}{\pi} \int_\xR \Delta_\alpha g_x\D^{1,\phi}F_\alpha  \dalpha
-\frac{1}{\pi} \int_\xR  \Gamma_\alpha  \dalpha.
$$

By definition of $\mathcal{T}(f)$ we have
\be\label{r980}
\begin{aligned}
&\Vert \L{f_x}\big[\D^{1,\phi},\mathcal{T}(f)\big](g)\Vert_{L^2}\les (I)+(II) \qquad \text{where} \\
&(I)\defn \Big\Vert \L{f_x}\int
\la \Delta_\alpha g_x\ra \bla\D^{1,\phi}F_\alpha\bra \dalpha\Big\Vert_{L^2},\\
&(II)\defn
\Big\Vert \L{f_x}\int |\Gamma_\alpha|\dalpha\Big\Vert_{L^2}.
\end{aligned}
\ee

{\em Step 1: estimate of $(I)$}. 
By Holder's inequality and Minkowski's inequality, one has 
$(I)\le (I_A)\times (I_B)$
with
\begin{align*}
(I_A)&\defn 
\left(\int_\xR\left(\int_\xR\big\vert\Delta_\alpha g_x(x)\big\vert^{\frac{3}{2}}\L{f_x(x)}^{-\frac{3}{2}} |\alpha|^{-\frac{1}{2}}\dalpha\right)^{4}\dx\right)^{\frac{1}{6}},\\
(I_B)&\defn\left(\iint \L{f_x(x)}^6\big\vert\D^{1,\phi}F_\alpha(x)\big\vert^3
|\alpha|\dalpha\dx\right)^{\frac{1}{3}},
\end{align*}
Set
\begin{equation*}
u(x)=\frac{|g_{xx}(x)|}{\L{f_x(x)}}\cdot
\end{equation*}
Since 
\begin{align*}
|\Delta_\alpha g_x(x)|&\leq  \left|\fint_0^\alpha u(x-\tau) \L{f_x(x-\tau)} \dtau\right|\\&\leq 
 \left|\fint_0^\alpha u(x-\tau) |\L{f_x(x-\tau)} -\L{f_x(x)} |\dtau\right|+  \left|\fint_0^\alpha u(x-\tau) \dtau\right|\L{f_x(x)} 
\\&\lesssim 
\left|\fint_0^\alpha u(x-\tau)^{\frac{10}{9}}\dtau\right|^{\frac{9}{10}}
\left|\fint_0^\alpha |\delta_{\tau}f_x(x)  |^{10}\dtau\right|^{\frac{1}{10}}
+ \left|\fint_0^\alpha u(x-\tau) \dtau\right|\L{f_x(x)},
\end{align*}
we have
\begin{align*}
 &\big\vert\Delta_\alpha g_x(x)\big\vert^{\frac{3}{2}}\L{f_x(x)}^{-\frac{3}{2}}
\\&\lesssim 
\left|\fint_0^\alpha u(x-\tau)^{\frac{10}{9}}\dtau\right|^{\frac{27}{20}} \left|\fint_0^\alpha |\delta_{\tau}f_x(x)  |^{10}\dtau\right|^{\frac{3}{20}}+ \left|\fint_0^\alpha u(x-\tau) \dtau\right|^{\frac{3}{2}}
\\&\lesssim 
\left|\fint_0^\alpha u(x-\tau)^{\frac{3}{2}}\dtau\right|
\left|\fint_0^\alpha |\delta_{\tau}f_x(x)|^{10}\dtau\right|^{\frac{3}{20}}
+\left|\fint_0^\alpha u(x-\tau)^{\frac{3}{2}} \dtau\right|.
\end{align*}
Thus, $(I_A)$ is bounded from above by
\begin{align*}
&\left(\int_\xR\left(	\int \left|\fint_0^\alpha u(x-\tau)^{\frac{3}{2}} \dtau\right| |\alpha|^{-\frac{1}{2}}\dalpha\right)^{4}\dx\right)^{\frac{1}{6}}\\
&+
\left(\int_\xR\left(\int \left|\fint_0^\alpha u(x-\tau)^{\frac{3}{2}}\dtau\right|
\left|\fint_0^\alpha |\delta_{\tau}f_x(x) |^{10}\dtau\right|^{\frac{3}{20}}
|\alpha|^{-\frac{1}{2}}\dalpha\right)^{4}\dx\right)^{\frac{1}{6}}.
\end{align*}
This implies that
\begin{align*}
(I_A)&\lesssim
\left(\int_\xR\left(	\int \left|\fint_0^\alpha u(x-\tau)^{\frac{3}{2}} \dtau\right| |\alpha|^{-\frac{1}{2}}\dalpha\right)^{4}\dx\right)^{\frac{1}{6}}\\
&\quad
+
\left(\int\left(	\int \left|\fint_0^\alpha u(x-\tau)^{\frac{3}{2}}\dtau\right|^{\frac{10}{9}}
|\alpha|^{-\frac{1}{3}}\dalpha\right)^{6}\dx\right)^{\frac{1}{10}}\\
&\quad\quad\times
\left(\int	\int  \left|\fint_0^\alpha |\delta_{\tau}f_x(x)  |^{10}\dtau\right|^{\frac{3}{2}}
|\alpha|^{-2}\dalpha\dx\right)^{\frac{1}{15}}.
\end{align*}
Now we use Hardy's inequality (see~\e{Hardy}) to write, for any function $v_1\geq 0$,
\begin{equation}
\int \left|\fint_0^\alpha v_1(x-\delta)\ddelta\right|^{p} |\alpha|^{-\beta}
\dalpha\lesssim \int v_1(x-\alpha)^p|\alpha|^{-\beta}\dalpha\quad 
(\forall  ~p\geq 1,~\forall \beta>0),
\end{equation}
and 
\begin{equation}
\int  \left|\fint_0^\alpha |\delta_{\tau}f_x(x)  |^{10}\dtau\right|^{\frac{3}{2}}
|\alpha|^{-2}\dalpha\lesssim  	\int   |\delta_{\alpha}f_x(x) |^{15} |\alpha|^{-2}\dalpha.
\end{equation}
Using also the representation 
formula~\e{b1}, it follows that
$$
(I_A)\lesssim \blA \D^{-\frac{1}{2}} ( u^{\frac{3}{2}})\brA_{L^4}^{\frac{2}{3}}
+\blA \D^{-\frac{2}{3}}(u^{\frac{5}{3}})\brA_{L^6}^{\frac{3}{5}} \Vert f_x\Vert_{\dot W^{\frac{1}{15},15}}.
$$
Now, by using the boundedness of Riesz potentials on Lebesgue spaces (see~\e{Riesz}) 
and Sobolev inequality, 
\be\label{X13a}
(I_A)\lesssim  \Vert	 u^{\frac{3}{2}}\Vert_{L^{\frac{4}{3}}}^{\frac{2}{3}}
+\Vert u^{\frac{5}{3}}\Vert_{L^{\frac{6}{5}}}^{\frac{3}{5}} \Vert f\Vert_{\dot H^{\frac{3}{2}}}.
\ee
By definition of $u$, this means that
\begin{equation}
(I_A)\lesssim (1+\Vert f\Vert_{\dot H^{\frac{3}{2}}})\lA\frac{g_{xx}}{\L{f_x}}\rA_{L^2}.
\end{equation}
This completes the analysis of $(I_A)$. 

As regards the second integral $(I_B)$, 
we will prove that 
\begin{equation}\label{X13}
(I_B)\lesssim  
\blA \D^{\tdm,\phi}f\brA_{L^2}^2\left(1+\blA \D^{\tdm,\phi}f\brA_{L^2}^3\right).
\end{equation}
Recall that
$$
(I_B)\defn\left(\iint \L{f_x(x)}^6\big\vert\D^{1,\phi}F_\alpha(x)\big\vert^3
|\alpha|\dalpha\dx\right)^{\frac{1}{3}}.
$$

 To prove \eqref{X13},  we will bound the integrand, starting from
\begin{align*}
\L{f_x(x)}^2\bla\D^{1,\phi}F_\alpha(x)\bra&\lesssim \L{f_x(x)}^2
\la\D^{1,\phi}F_\alpha(x)-2\frac{\Delta_{\alpha}f(x)}{\langle \Delta_{\alpha}f(x)\rangle^4 } \D^{1,\phi}\Delta_{\alpha}f(x)\ra\\
&\quad+\frac{\L{f_x(x)}^2}{\langle \Delta_{\alpha}f(x)\rangle^2 } |\D^{1,\phi}\Delta_{\alpha}f(x)|.
\end{align*}
Then, we want to use the formula~\e{d1phi}, which states that
$$
\D^{1,\phi}U(x)=\frac{1}{4}\int_{\xR}\frac{2U(x)-U(x+z)-U(x-z)}{z^2}
\kappa\left(\frac{1}{|z|}\right) \dz.
$$
For the sake of shortness, set 
\begin{equation*}
\diff\!\mu(z)=\kappa\left(\frac{1}{|z|}\right)\frac{\dz}{z^2},
\end{equation*}
and observe that $2U(x)-U(x+z)-U(x-z)=\delta_{-z}\delta_z U(x)$. 
By combining this 
with 
the fact that $E_\alpha(f)=\Delta_\alpha f-f_x$ and the elementary inequality
\begin{equation*}
\L{f_x(x)}^2\lesssim \L{\Delta_{\alpha}f(x)}^2(1+|E_\alpha(f)(x)|^2),
\end{equation*}
we deduce that
\begin{align*}
&\L{f_x(x)}^2\bla\D^{1,\phi}F_\alpha(x)\bra\\
&\lesssim  (1+|E_\alpha(f)(x)|^2) \int \L{\Delta_{\alpha}f(x)}^2\la \delta_{-z}\delta_{z}
F_\alpha(x)-2\frac{\Delta_{\alpha}f(x)}{\langle \Delta_{\alpha}f(x)\rangle^4 } \delta_{-z}\delta_{z}
\Delta_{\alpha}f(x)\ra \diff\!\mu(z)\\
&\quad+\left(1+|E_\alpha(f)(x)|^2\right)
 \bla\D^{1,\phi}\Delta_{\alpha}f(x)\bra.
\end{align*}

To estimate the latter integral, we pause to establish an elementary inequality which will allow 
us to bound the term
$$
\delta_{-z}\delta_{z}
F_\alpha(x)-2\frac{\Delta_{\alpha}f(x)}{\langle \Delta_{\alpha}f(x)\rangle^4 } \delta_{-z}\delta_{z}
\Delta_{\alpha}f(x).
$$

\begin{lemma}\label{Z10}
There exists a constant $C>0$ such that, for any triple of real numbers $(x_1,x_2,x_3)$, the quantity
$$
\Sigma(x_1,x_2,x_3)=\frac{2x_1^2}{1+x_1^2}-\frac{x_2^2}{1+x_2^2}-\frac{x_3^2}{1+x_3^2}
$$
satisfies
\begin{multline}\label{Z4}
\left|\Sigma(x_1,x_2,x_3)-\frac{2x_1}{\langle x_1\rangle^4}\left(2x_1-x_2-x_3\right)\right|\\
\le \frac{C}{\langle x_1\rangle^2}\left(|x_1-x_2|^2+|x_1-x_3|^2+|x_1-x_2|^3+|x_1-x_3|^3\right).
\end{multline}	
\end{lemma}
\begin{proof}
By an elementary calculation, one verifies that
$$
\Sigma(x_1,x_2,x_3)=\frac{1}{\langle x_1\rangle^2\langle x_2\rangle^2 \langle x_3\rangle^2}\left( (1+x_2^2)(x_1^2-x_3^2)+(1+x_3^2)(x_1^2-x_2^2)\right).
$$
This yields the identity
\begin{align*}
&\Sigma(x_1,x_2,x_3)-\frac{2x_1}{\langle x_1\rangle^4}\left(2x_1-x_2-x_3\right)\\&=\frac{\sigma_1(x_1,x_2,x_3)}{\langle x_1\rangle^2\langle x_2\rangle^2 \langle x_3\rangle^2}+\left(\frac{x_2+x_1}{\langle x_1\rangle^2\langle x_2\rangle^2 }-\frac{2x_1}{\langle x_1\rangle^4}\right)(2x_1-x_2-x_3)\\&
=\frac{\sigma_1(x_1,x_2,x_3)}{\langle x_1\rangle^2\langle x_2\rangle^2 \langle x_3\rangle^2}+\frac{(x_2-x_1)(2x_1-x_2-x_3)}{\langle x_1\rangle^2\langle x_2\rangle^2 }\\
&\quad+\frac{2x_1(x_1+x_2)(x_1-x_2)(2x_1-x_2-x_3)}{\langle x_1\rangle^4\langle x_2\rangle^2 },
\end{align*}
with 
$$\sigma_1(x_1,x_2,x_3)=(x_1x_2+x_2x_3+x_3x_1-1)(x_2-x_3)(x_1-x_3).$$
Now we apply the triangle inequality together with  $2ab\le a^2+b^2$  to get
\begin{align*}
\la x_1x_2+x_2x_3+x_3x_1-1\ra\lesssim (1+ \la x_1-x_2\ra) \langle x_2\rangle^2 \langle x_3\rangle^2,
\end{align*}
together with $|x_1(x_1+x_2)|\lesssim \langle x_1\rangle^2\langle x_2\rangle^2$. 
As a result, we have
\begin{align*}
&\left|\Sigma(x_1,x_2,x_3)-\frac{2x_1}{\langle x_1\rangle^4}\left(2x_1-x_2-x_3\right)\right|\\&\lesssim\frac{1}{\langle x_1\rangle^2}(1+ \la x_1-x_2\ra)|x_2-x_3||x_1-x_3|+\frac{1}{\langle x_1\rangle^2}|x_2-x_1||2x_1-x_2-x_3|,
\end{align*} which implies \eqref{Z4}. This completes the proof.
\end{proof}

Consequently we deduce, 
\begin{align*}
\L{f_x(x)}^2\big\vert\D^{1,\phi}F_\alpha(x)|&\lesssim  \left(1+|E_\alpha(f)(x)|^2\right) \int \left(\big\vert\delta_{z}\Delta_{\alpha}f(x)\big\vert^2+\big\vert\delta_{z}\Delta_{\alpha}f(x)\big\vert^3\right)\diff\!\mu(z)\\
&+\left(1+|E_\alpha(f)(x)|^2\right)
\bla\D^{1,\phi}\Delta_{\alpha}f(x)\bra.
\end{align*}
Thus, 
\begin{align*}
(I_B)^3&=\iint \L{f_x(x)}^6\big\vert\D^{1,\phi}F_\alpha(x)\big\vert^3 |\alpha|\dalpha\dx\\
&\lesssim  \iint\left(\int \left(1+|E_\alpha(f)(x)|^2\right)  |\delta_{z}\Delta_\alpha f(x)|^2\diff\!\mu(z)\right)^3  |\alpha|\dalpha\dx\\
&\quad+ \iint \left(\int \left(1+|E_\alpha(f)(x)|^2\right)  |\delta_{z}\Delta_\alpha f(x)|^3
\diff\!\mu(z)\right)^3 |\alpha|\dalpha\dx\\
&\quad+\iint \left(1+|E_\alpha(f)(x)|^6\right)
|\D^{1,\phi}\Delta_{\alpha}f(x)|^3 |\alpha|\dalpha\dx.
\end{align*}
Then, by using  the Minkowski's inequality and Holder's inequality, 
\begin{align*}
&(I_B)
\lesssim  \int \left(\int\lA\Delta_\alpha \delta_zf\rA_{L^6}^6|\alpha|\dalpha \right)^{\frac{1}{3}} \diff\!\mu(z) +\int \left(\int\lA\Delta_\alpha \delta_zf\rA_{L^9}^9|\alpha|\dalpha \right)^{\frac{1}{3}} \diff\!\mu(z)  \\&+\int\left(\int\lA\Delta_\alpha \delta_zf\rA_{L^{12}}^{12}
|\alpha|^4\dalpha \right)^{\frac{1}{6}}   \diff\!\mu(z)\left(\int\lA E_\alpha(f)\rA_{L^{12}}^{12}\frac{\dalpha}{\alpha^2} \right)^{\frac{1}{6}} \\&+\int\left(\int\lA\Delta_\alpha \delta_zf\rA_{L^{18}}^{18}
|\alpha|^4\dalpha \right)^{\frac{1}{6}}  \diff\!\mu(z) \left(\int\lA E_\alpha(f)\rA_{L^{12}}^{12}\frac{\dalpha}{\alpha^2} \right)^{\frac{1}{6}}\\
&+\left(\int \blA\D^{1,\phi}\Delta_{\alpha}f\brA_{L^3}^3|\alpha|\dalpha\right)^{\frac{1}{3}}\\
&+\left(\int\blA\D^{1,\phi}\Delta_{\alpha}f\brA_{L^6}^6|\alpha|^4\dalpha\right)^{\frac{1}{6}}
\left(\int\lA E_\alpha(f)\rA_{L^{12}}^{12}\frac{\dalpha}{\alpha^2}\right)^{\frac{1}{6}}.
\end{align*}
By Sobolev's inequality  and  Lemma~\ref{L:2.7}, we have \label{page41}
\begin{align*}
&\int \left(\int\lA\Delta_\alpha \delta_zf\rA_{L^6}^6|\alpha|\dalpha \right)^{\frac{1}{3}} \diff\!\mu(z) 
\lesssim \int \left(\int\blA\Delta_\alpha \delta_z|D|^{\frac{1}{3}}f\brA_{L^2}^6|\alpha|\dalpha \right)^{\frac{1}{3}} \diff\!\mu(z)\\&\quad\quad\quad\quad\quad\quad\lesssim\int \int |\xi|^2|\mathcal{F} (\delta_{z} f)(\xi)|^2\kappa\left(\frac{1}{|z|}\right)\frac{\dz}{z^2}\\
& \quad\quad\quad\quad\quad\quad\lesssim  \int \int |\xi|^2\min\{\la z\ra \la \xi\ra,1\}^2|\hat f(\xi)|^2\kappa\left(\frac{1}{|z|}\right)\frac{\dz}{z^2}
\\& \quad\quad\quad\quad\quad\quad \lesssim \blA \D^{\tdm,\phi}f\brA_{L^2}^2,
\end{align*}
where we have used again the equivalence $\phi\sim\kappa$ to obtain the last inequality. 
By the same token, we have 
\begin{align*}
\int \left(\int\lA\Delta_\alpha \delta_zf\rA_{L^9}^9|\alpha|\dalpha \right)^{\frac{1}{3}} \diff\!\mu(z)
&\lesssim\int \left(\int |\xi|^{\frac{7}{3}} \min\{|\xi||z|,1\}^2|\hat f(\xi)|^2 \dxi\right)^{\frac{3}{2}}\diff\!\mu(z)\\
&\overset{\e{Min}, \e{est:kappa3}}{\lesssim}    \left(\int |\xi|^{3}\kappa(|\xi|)^{\frac{2}{3}} |\hat f(\xi)|^2 \dxi\right)^{\frac{3}{2}}\\
&\overset{\e{n60b}}{\lesssim} \blA \D^{\tdm,\phi}f\brA_{L^2}^3.
\end{align*}
Similarly, we also obtain that one can estimate by $\blA \D^{\tdm,\phi}f\brA_{L^2}$ the five following quantities:
\begin{align*}
&\left(\int\left(\int\lA\Delta_\alpha \delta_zf\rA_{L^{12}}^{12}
|\alpha|^4\dalpha \right)^{\frac{1}{6}}   \diff\!\mu(z) \right)^{\frac{1}{2}},
\quad \left(\int \blA\D^{1,\phi}\Delta_{\alpha}f\brA_{L^3}^3|\alpha|\dalpha\right)^{\frac{1}{3}},
\\
&
\left(\int\left(\int\lA\Delta_\alpha \delta_zf\rA_{L^{18}}^{18}
|\alpha|^4\dalpha \right)^{\frac{1}{6}}  \diff\!\mu(z)\right)^{\frac{1}{3}},\quad 
\left(\int\blA\D^{1,\phi}\Delta_{\alpha}f\brA_{L^6}^6|\alpha|^4\dalpha\right)^{\frac{1}{12}},\\
&\left(\int\lA E_\alpha(f)\rA_{L^{12}}^{12}\frac{\dalpha}{\alpha^2}\right)^{\frac{1}{6}}.
\end{align*}
Thus, 
\begin{align*}
(I_B)
\lesssim \blA \D^{\tdm,\phi}f\brA_{L^2}^2\left(1+\blA \D^{\tdm,\phi}f\brA_{L^2}^3\right),
\end{align*}
which proves~\eqref{X13}.

By combining \e{X13a} and \e{X13}, we conclude that the term $(I)$ which appears in~\e{r980} satisfies
\begin{align*}
(I)&	\lesssim\blA \D^{\tdm,\phi}f\brA_{L^2}^2\left(1+\blA \D^{\tdm,\phi}f\brA_{L^2}^4\right) \lA\frac{g_{xx}}{\L{f_x}}\rA_{L^2},
\end{align*}
which completes the analysis of $(I)$. 

\textit{Step 2.2: estimate of $(II)$.} It remains to estimate the second term in \e{r980}. We want to prove that
\begin{equation}
\label{w}
\bigg\Vert \L{f_x}\int |\Gamma_\alpha|\dalpha\bigg\Vert_{L^2}
\lesssim\brA \D^{\frac{3}{2},\phi}g\brA_{L^2}
\left( 1+\blA\D^{\frac{3}{2},\phi} f\brA_{L^2}^{4}\right) \lA\frac{f_{xx}}{\L{f_x}}\rA_{L^2}.
\end{equation}

We have 
\begin{equation*}
\L{f_x(x)}\la \Gamma_\alpha(x)\ra
\lesssim\int \L{f_x(x)}\la\delta_zF_\alpha(x)\ra
\la \delta_z\Delta_\alpha g_x(x)\ra \kappa\left(\frac{1}{|z|}\right)\frac{\dz}{z^2}\cdot
\end{equation*}
Using again the estimate (see~\e{r961})
\begin{equation*}
\left|\frac{x_1^2}{1+x_1^2}-\frac{x_2^2}{1+x_2^2}\right|\lesssim \frac{1}{1+x_1^2}
\left(|x_1-x_2|^2+|x_1-x_2|\right),
\end{equation*}
we have
\begin{align*}
\la \delta_zF_\alpha(x)\ra 
&\les \L{\Delta_{\alpha} f(x)}^{-2}|\delta_z\Delta_{\alpha} f(x)|^2+ \L{\Delta_{\alpha} f(x)}^{-2}|\delta_z\Delta_{\alpha} f(x)|.
\end{align*}	
Combining this with 
\begin{equation*}
\L{f_x(x)}^2 \L{\Delta_{\alpha} f(x)}^{-2}\lesssim 1+|E_\alpha(f)(x)|^2,
\end{equation*}
one gets
\begin{multline*}
\L{f_x(x)}\la \Gamma_\alpha(x)\ra\lesssim \left(1+|f_x(x)-\Delta_{\alpha} f(x)|^2\right)
\int \L{f_x(x)}^{-1}|\delta_z\Delta_{\alpha} f(x)|\times \\
\times \left(1+|\delta_z\Delta_{\alpha} f(x)|\right)
|\delta_z\Delta_\alpha g_x(x)| \kappa\left(\frac{1}{|z|}\right)\frac{\dz}{z^2}\cdot
\end{multline*}
Set 
\begin{equation*}
u(x)=\frac{|f_{xx}(x)|}{\L{f_x(x)}}\cdot
\end{equation*}
One has  
\begin{align*}
|\delta_z\Delta_\alpha f(x)|&\leq\frac{1}{|\alpha|} \left|\int_{0}^{\alpha}
\int_{0}^{z} u(x-\tau-z_1)\L{f_x(x-\tau-z_1)}\dz_1\dtau\right|\\&\leq 
\frac{1}{|\alpha|} \left|\int_{0}^{\alpha}\int_{0}^{z} u(x-\tau-z_1)|\L{f_x(x-\tau-z_1)}-\L{f_x(x-\tau)}|\dz_1\dtau\right|\\
&\quad+\frac{1}{|\alpha|} \left|\int_{0}^{\alpha}\int_{0}^{z} u(x-\tau-z_1)\la \L{f_x(x-\tau)}-\L{f_x(x)}\ra\dz_1\dtau\right|\\
&\quad+\frac{\L{f_x(x)}}{|\alpha|} \left|\int_{0}^{\alpha}\int_{0}^{z} u(x-\tau-z_1)\dz_1\dtau\right|.
\end{align*}
Then, thanks to $|\langle x_1\rangle-\langle x_2\rangle |\lesssim |x_1-x_2|$, 
we infer that
\begin{align*}
&\L{f_x(x)}|^{-1}|\delta_{z}\Delta_\alpha f(x)|\lesssim |z|\left|\fint_{0}^{\alpha}\fint_{0}^{z} u(x-\tau-z_1)|\delta_{z_1}f_x(x-\tau)|\dz_1\dtau\right|\\
&\quad\quad+ |z|\left|\fint_{0}^{\alpha}\fint_{0}^{z} u(x-\tau-z_1)|\delta_\tau f_x(x)|\dz_1\dtau\right|
+ |z|\left|\fint_{0}^{\alpha}\fint_{0}^{z} u(x-\tau-z_1)\dz_1\dtau\right|\\
&\quad\quad\lesssim |z|
\left|\fint_{0}^{\alpha}\fint_{0}^{z} u(x-\tau-z_1)^{\frac{10}{9}}\dz_1\dtau\right|^{\frac{9}{10}} J_\alpha(x,z),
\end{align*}
where 
$$
J_\alpha(x,z)=1+\left|\fint_{0}^{\alpha} |\delta_{\tau}f_x(x)|^{10}\dtau\right|^{\frac{1}{10}}
+\left|\fint_{0}^{\alpha}\fint_{0}^{z} |\delta_{z_1}f_x(x-\tau)|^{10}\dz_1\dtau\right|^{\frac{1}{10}}.
$$
Set 
$$
\Phi_\alpha(x,z)
=J_\alpha(x,z)\left(1+|f_x(x)-\Delta_{\alpha} f(x)|^2\right)
\left(1+|\delta_{z}\Delta_{\alpha} f(x)|\right)\la \delta_z\Delta_\alpha g_x(x)\ra.
$$
Then, repeating arguments already used before, we find that
\begin{align*}
&\L{f_x(x)} \int\la \Gamma_\alpha(x)\ra \dalpha\\
&\qquad\lesssim \iint_{|\alpha|\leq \lam}
\left|\fint_{0}^{\alpha}\fint_{0}^{z} u(x-\tau-z_1)^{\frac{10}{9}}\dz_1\dtau\right|^{\frac{9}{10}}
\Phi_\alpha(x,z) \kappa\left(\frac{1}{|z|}\right)|\alpha|\frac{\dz \dalpha}{\la z\ra \la \alpha\ra}\\
&\qquad\lesssim\left(\iint
\left\vert\alpha|^{\frac{1}{18}}|z|^{\frac{1}{18}}\fint_{0}^{\alpha}\fint_{0}^{z} u(x-\tau-z_1)^{\frac{10}{9}}\dz_1\dtau\right|\frac{\dz \dalpha}{\la z\ra \la \alpha\ra}\right)^{\frac{9}{10}}\\
&\qquad\qquad\qquad\times\left(\iint\Phi_\alpha(x,z)^{10} \kappa\left(\frac{1}{|z|}\right)^{10}|\alpha|^{\frac{17}{2}}\frac{\dz \dalpha}{|z|^{\frac{3}{2}}}\right)^{\frac{1}{10}}.
\end{align*}
By Hardy's inequality ( see \eqref{Hardy}),
\begin{align*}
& \iint
\left\vert\alpha|^{\frac{1}{18}}|z|^{\frac{1}{18}}\fint_{0}^{\alpha}\fint_{0}^{z} u(x-\tau-z_1)^{\frac{10}{9}}\dz_1\dtau\right|\frac{\dz \dalpha}{\la z\ra \la \alpha\ra}\\
&\quad\lesssim \iint
|\alpha|^{\frac{1}{18}}|z|^{\frac{1}{18}} u(x-\alpha-z)^{\frac{10}{9}}\frac{\dz \dalpha}{\la z\ra \la \alpha\ra}\\&\quad=c_1 \iint
|\alpha|^{\frac{1}{18}} |D|^{-\frac{1}{18}}(u^{\frac{10}{9}})(x-\alpha)\frac{ \dalpha}{\la \alpha\ra}=c_2 |D|^{-\frac{1}{9}}(u^{\frac{10}{9}})(x).\end{align*}
Consequently we deduce, 
\begin{multline*}
\L{f_x(x)} \int \la \Gamma_\alpha(x)\ra \dalpha\\
\lesssim\left( |D|^{-\frac{1}{9}}(u^{\frac{10}{9}})(x)\right)^{\frac{9}{10}}
\left(\iint\Phi_\alpha(x,z)^{10} \kappa\left(\frac{1}{|z|}\right)^{10}|\alpha|^{\frac{17}{2}}
\frac{\dz \dalpha}{|z|^{\frac{3}{2}}}\right)^{\frac{1}{10}}.
\end{multline*}
By Cauchy-Schwarz, Minkowski and Hardy-Littlewood-Sobolev,    
\begin{align*}
&\lA \L{f_x}\int |\Gamma_\alpha|\dalpha\rA_{L^2}\\
&\qquad\lesssim
\lA\left( |D|^{-\frac{1}{9}}(u^{\frac{10}{9}})(x)\right)^{\frac{9}{10}}\rA_{L^{\frac{5}{2}}}
\lA\left(\iint\Phi_\alpha(\cdot,z)^{10} \kappa\left(\frac{1}{|z|}\right)^{10}|\alpha|^{\frac{17}{2}}
\frac{\dz \dalpha}{|z|^{\frac{3}{2}}}\right)^{\frac{1}{10}}\rA_{L^{10}}\\
&\qquad= \brA \D^{-\frac{1}{9}}(u^{\frac{10}{9}})\brA_{L^{\frac{9}{4}}}^{\frac{9}{10}}
\left(\iint\Vert\Phi_\alpha(\cdot,z)\Vert_{L^{10}}^{10}
\kappa\left(\frac{1}{|z|}\right)^{10}|\alpha|^{\frac{17}{2}}\frac{  \dalpha\dz}{|z|^{\frac{3}{2}}}\right)^{\frac{1}{10}}\\
&\qquad\lesssim\Vert u\Vert_{L^2}
\left(\iint\Vert\Phi_\alpha(\cdot,z)\Vert_{L^{10}}^{10}
\kappa\left(\frac{1}{|z|}\right)^{10}|\alpha|^{\frac{17}{2}}\frac{  \dalpha\dz}{|z|^{\frac{3}{2}}}\right)^{\frac{1}{10}}.
\end{align*}
We now estimate the $L^{10}$-norm of $\Phi_\alpha(\cdot,z)$. 
Note that 
$$
|\delta_{z}\Delta_{\alpha} f(x)|+|f_x(x)-\Delta_{\alpha} f(x)|\lesssim J_\alpha(x,z).
$$ 
So, 
\begin{align*}
\Phi_\alpha(x,z)&\lesssim J_\alpha(x,z)^4|\delta_{z}\Delta_\alpha g_x(x)|\\
&\lesssim |\alpha|^{-1} |\delta_z\delta_\alpha g_x(x)|\bigg(1+ |\fint_{0}^{\alpha} |\delta_{\tau}f_x(x)|^{10}\dtau|^{\frac{4}{10}}\\
&\phantom{\lesssim |\alpha|^{-1} |\delta_z\delta_\alpha g_x(x)|\bigg(1\,}+\left|\fint_{0}^{\alpha}\fint_{0}^{z} |\delta_{z_1}f_x(x-\tau)|^{10}\dz_1\dtau\right|^{\frac{4}{10}}\bigg).
\end{align*}
Then, 
\begin{align*}
&|\alpha|^{10}\Vert\Phi_\alpha(\cdot,z)\Vert_{L^{10}}^{10}\\
&\lesssim
\Vert \delta_z\delta_\alpha g_x\Vert_{L^{10}}^{10}+ \Vert \delta_z\delta_\alpha g_x\Vert_{L^{20}}^{10}
\left(\bigg\vert\fint_{0}^{z} \Vert \delta_{z_1} f_x\Vert^{40}_{L^{80}}\dz_1\bigg\vert
+\bigg\vert\fint_{0}^{\alpha} \Vert \delta_{\tau}f_x\Vert_{L^{80}}^{40}\dtau\bigg\vert\right) 
\\
&~\lesssim 
\Vert \delta_z\delta_\alpha g_x\Vert_{\dot H^{\frac{2}{5}}}^{10}+
\Vert \delta_z\delta_\alpha g_x\Vert_{\dot H^{\frac{9}{20}}}^{10}
\left(\bigg\vert\fint_{0}^{z} \Vert \delta_{z_1}f_x\Vert^{40}_{\dot H^{\frac{39}{80}}}\dz_1 \bigg\vert
+\bigg\vert\fint_{0}^{\alpha} \Vert \delta_{\tau}f_x\Vert_{\dot H^{\frac{39}{80}}}^{40}\dtau\bigg\vert\right).
\end{align*}
Thus, 
\begin{align*}
&\iint\Vert\Phi_\alpha(\cdot,z)\Vert_{L^{10}}^{10} \kappa\left(\frac{1}{|z|}\right)^{10}|\alpha|^{\frac{17}{2}}
\frac{\dalpha\dz}{|z|^{\frac{3}{2}}}\\
&\qquad\lesssim
\int\left[ \int  \Vert \delta_\alpha\delta_z g_x(\cdot)\Vert_{\dot H^{\frac{2}{5}}}^{10}
\kappa\left(\frac{1}{|z|}\right)^{10}\frac{ \dz}{|z|^{\frac{3}{2}}} \right]\frac{  \dalpha }{|\alpha|^{\frac{3}{2}}}\\
&\qquad\quad+ \int \left[\int   \Vert \delta_\alpha\delta_z g_x(\cdot)\Vert_{\dot H^{\frac{9}{20}}}^{10}
\frac{\dalpha }{|\alpha|^{\frac{3}{2}}}\right]
\bigg\vert\fint_{0}^{z} \Vert\delta_{z_1} f_x(\cdot)\Vert^{40}_{\dot H^{\frac{39}{80}}}\dz_1  \bigg\vert
\kappa\left(\frac{1}{|z|}\right)^{10}\frac{  \dz}{|z|^{\frac{3}{2}}}
\\
&\qquad\quad+
\int \left[ \int \Vert\delta_\alpha\delta_z g_x(\cdot)\Vert_{\dot H^{\frac{9}{20}}}^{10}
\kappa\left(\frac{1}{|z|}\right)^{10}\frac{ \dz}{|z|^{\frac{3}{2}}}\right]
 \bigg\vert\fint_{0}^{\alpha} \Vert \delta_{\tau}f_x(\cdot)\Vert_{\dot H^{\frac{39}{80}}}^{40}\dtau  \bigg\vert
 \frac{\dalpha }{|\alpha|^{\frac{3}{2}}}\cdot
\end{align*}
Proceeding as in the proof of  Lemma~\ref{L:2.7}, we get
 \begin{align*}
\int  \Vert \delta_\alpha \delta_z g_x\Vert_{\dot H^{\frac{2}{5}}}^{10}
\kappa\left(\frac{1}{|z|}\right)^{10}\frac{ \dz}{|z|^{\frac{3}{2}}}
&\lesssim  \left(\int |\xi|^{\frac{29}{10}}  \kappa\left(\frac{1}{|z|}\right)^{2}|\mathcal{F}(\delta_\alpha g)(\xi)|^2\dxi\right)^{5}\\
&\les
\blA\delta_\alpha 
\big(\D^{\frac{29}{20},\phi}g\big)\brA_{L^2}^{10},\\
\int \Vert \delta_\alpha \delta_z g_x(\cdot)\Vert_{\dot H^{\frac{9}{20}}}^{10}
\kappa\left(\frac{1}{|z|}\right)^{10}\frac{ \dz}{|z|^{\frac{3}{2}}}
&\lesssim \left(\int |\xi|^{3}  \kappa\left(\frac{1}{|z|}\right)^{2}
|\mathcal{F}(\delta_\alpha g)(\xi)|^2\dxi\right)^{5}\\
&\lesssim \brA \D^{\frac{3}{2},\phi}g\brA_{L^2}^{10},
\end{align*}
and
\begin{align*}
\int  \Vert \delta_{z}\delta_\alpha g_x(\cdot)\Vert_{\dot H^{\frac{9}{20}}}^{10}
\frac{  \dalpha }{|\alpha|^{\frac{3}{2}}}
&\lesssim \left(\int |\xi|^{3}|\min\{\la z\ra \la \xi\ra,1\}^2|\mathcal{F}g(\xi)|^2\dxi\right)^{5}\\&\lesssim \blA\D^{\frac{3}{2},\phi} g\brA_{L^2}^{10} \kappa\(\frac{1}{|z|}\)^{-10}.
 \end{align*}
Here we have used  the fact that $\min\{\la z\ra \la \xi\ra,1\}^2\lesssim \phi(|\xi|)^2 \kappa\Big(\frac{1}{|z|}\Big)^{-2}$ in the last inequality. 

Therefore, by applying~Lemma~\ref{L:2.7} together with Hardy and Minkowski inequalities, we obtain
 \begin{align*}
 &\iint\Vert\Phi_\alpha(\cdot,z)\Vert_{L^{10}}^{10}
 \kappa\left(\frac{1}{|z|}\right)^{10}|\alpha|^{\frac{17}{2}}\frac{  \dalpha\dz}{|z|^{\frac{3}{2}}}\\
 &\qquad\lesssim   \int\Vert\delta_\alpha \left(|D|^{\frac{29}{20},\phi}g\right)\Vert_{L^2}^{10}
 \frac{  \dalpha }{|\alpha|^{\frac{3}{2}}}
 +\blA\D^{\frac{3}{2},\phi} g\brA_{L^2}^{10}\int  \Vert \delta_{z}f_x(\cdot)\Vert^{40}_{\dot H^{\frac{39}{80}}}
 \frac{  \dz}{|z|^{\frac{3}{2}}}
 \\
 &\qquad\quad+
\brA \D^{\frac{3}{2},\phi}g\brA_{L^2}^{10}
\int \Vert\delta_\alpha f_x(\cdot)\Vert_{\dot H^{\frac{39}{80}}}^{40}
\frac{  \dalpha }{|\alpha|^{\frac{3}{2}}}\\
&\qquad\lesssim \brA \D^{\frac{3}{2},\phi}g\brA_{L^2}^{10}
+\blA |D|^{\frac{3}{2},\phi}g\brA_{L^2}^{10}   \lA f\rA_{\dot H^{\frac{3}{2}}}^{40}.
\end{align*}
The proof of~\e{w} is now in reach. Indeed, altogether, the previous estimates imply that
$$
\iint\Vert\Phi_\alpha(\cdot,z)\Vert_{L^{10}}^{10} \kappa\left(\frac{1}{|z|}\right)^{10}|\alpha|^{\frac{17}{2}}\frac{  \dalpha\dz}{|z|^{\frac{3}{2}}}
\lesssim \brA \D^{\frac{3}{2},\phi}g\brA_{L^2}^{10}\left(1 +\blA\D^{\frac{3}{2},\phi} f\brA_{L^2}^{40}\right).$$
Therefore,
\begin{align*}
\lA \L{f_x}\int |\Gamma_\alpha|\dalpha\rA_{L^2}\lesssim 
\Vert u\Vert_{L^2}\brA \D^{\frac{3}{2},\phi}g\brA_{L^2}
\left( 1+\blA\D^{\frac{3}{2},\phi} f\brA_{L^2}^{4}\right)
\end{align*}
equivalent to the wanted result~\e{w}. 

This completes the proof.
\end{proof}

\section{Proof of the main theorem}\label{S:6}

\subsection{The main estimate}
By combining Corollary~\ref{C:5.2} with Propositions~\ref{X1},~\ref{P:5.6} and~\ref{P:5.7}, we obtain our main 
weighted Sobolev energy estimate. 

\begin{proposition}\label{P:6.1}
Consider a function $\phi$ defined by~\e{n10}, for some admissible weight~$\kappa$. Then, the following property holds: 
For all $\eps\in (0,1]$, for all smooth solution $f$ of the approximate 
Muskat equation~\e{main:acp}, and for all positive function $\lambda=\lambda(t)$,
\begin{align*}
&\fract \blA \D^{\tdm,\phi}f\brA_{L^2}^2
+ \int_\xR \frac{\bla \D^{2,\phi}f\bra^2}{1+(\partial_x f)^2} \dx+\eps\lA f\rA_{\dot{H}^4}^2\\
&\qquad\lesssim
 \kappa\Big(\frac{1}{\lam}\Big)^{-1}\left(1+\blA\D^{\tdm,\phi}f\brA_{L^2}^{7}\right)
\blA\D^{\tdm,\phi}f\brA_{L^2} \int_\xR \frac{\bla \D^{2,\phi}f\bra^2}{1+(\partial_x f)^2} \dx
\\
&\qquad\quad+
\left( 1+\blA\D^{\frac{3}{2},\phi} f\brA_{L^2}^{5}\right)\brA\D^{\frac{3}{2},\phi}f\brA_{L^2} \lA\frac{f_{xx}}{\L{f_x}}\rA_{L^2}\lA\frac{|D|^{2,\phi}f}{\L{f_x}}\rA_{L^2} \\
&\qquad\quad+\frac{1}{\sqrt{\lam}}\left(1+\blA\D^{\frac{3}{2},\phi}f\brA_{L^2}\right) \blA\D^{\frac{3}{2},\phi}f\brA_{L^2}^2\blA\D^{2,\phi}f\brA_{L^2}.
\end{align*}
\end{proposition}
\begin{proof} It follows from Corollary~\ref{C:5.2} with Propositions~\ref{X1},~\ref{P:5.6} and~\ref{P:5.7} that 
\begin{align*}
&\fract \blA \D^{\tdm,\phi}f\brA_{L^2}^2
+ \int_\xR \frac{\bla \D^{2,\phi}f\bra^2}{1+(\partial_x f)^2} \dx+\eps\lA f\rA_{\dot{H}^4}^2\\
&\qquad\lesssim
 \kappa\Big(\frac{1}{\lam}\Big)^{-1}\left(1+\blA\D^{\tdm,\phi}f\brA_{L^2}^{7}\right)
\blA\D^{\tdm,\phi}f\brA_{L^2} \int_\xR \frac{\bla \D^{2,\phi}f\bra^2}{1+(\partial_x f)^2} \dx\\
&\qquad\quad+\kappa\Big(\frac{1}{\lam}\Big)^{-1}\left(1+\blA\D^{\tdm,\phi}f\brA_{L^2}^2\right)
\blA\D^{\frac{3}{2},\phi}f\brA_{L^2} \lA\frac{|D|^{2,\phi}f}{\L{f_x}}\rA_{L^2} 
\bigg\Vert\frac{\mathcal{H}(\D^{2,\phi}f)}{\L{f_x}}\bigg\Vert_{L^2}
\\
&\qquad\quad+
\left( 1+\blA\D^{\frac{3}{2},\phi} f\brA_{L^2}^{5}\right)\brA\D^{\frac{3}{2},\phi}f\brA_{L^2} \lA\frac{f_{xx}}{\L{f_x}}\rA_{L^2}\lA\frac{|D|^{2,\phi}f}{\L{f_x}}\rA_{L^2} \\
&\qquad\quad+\frac{1}{\sqrt{\lam}}\left(1+\blA\D^{\frac{3}{2},\phi}f\brA_{L^2}\right) \blA\D^{\frac{3}{2},\phi}f\brA_{L^2}^2\blA\D^{2,\phi}f\brA_{L^2}.
\end{align*}
Combining this with  Lemma~\ref{L:Hilbert2}, which implies that
$$
\int \frac{(\mathcal{H}g)^2}{1+h^2}\dx \lesssim\(1+\Vert h\Vert_{\dot H^{\frac{1}{2}}}^2\)
\int  \frac{g^2}{1+h^2}\dx,
$$
one obtains the wanted result. 
\end{proof}

\subsection{Uniform global in time estimates under a smallness assumption}
Set $\phi=1$ and let $\lambda$ goes to $+\infty$ in the main estimate given by Proposition~\ref{P:6.1}, to obtain
\begin{align*}
&\fract \blA f\brA_{\dot H^{\frac{3}{2}}}^2
+ \int_\xR \frac{(\partial_{xx}f)^2}{1+(\partial_xf)^2} \dx+\eps\lA f\rA_{\dot{H}^4}^2\lesssim  \left(1+\Vert f\Vert_{\dot H^{\frac{3}{2}}}^{7}\right)
\Vert f\Vert_{\dot H^{\frac{3}{2}}} \int_\xR \frac{(\partial_{xx}f)^2}{1+(\partial_xf)^2} \dx.
\end{align*}
Then it is very easy to prove a uniform estimate under a smallness 
assumption on the initial data (see for instance \cite[\S$3.3$]{AN1} or \cite[\S$2.5$]{AN2}). 
We obtain the following
\begin{proposition}\label{C:6.2}
There exists a positive constant $c$ such that the following 
property holds. 
For all initial data $f_0$ in $H^{\tdm}(\xR)$ satisfying 
\be\label{nUG0}
\lA f_0\rA_{\dot{H}^\tdm}   \leq  c,
\ee
and for all $\eps\in (0,1]$,  the solution 
$f$ to the approximate Cauchy problem~\e{main:acp} satisfies
\begin{equation}\label{nUG1}
\sup_{t\in [0,+\infty)}\lA f(t)\rA_{\dot{H}^\tdm}^2 +\mez
\int_0^{+\infty}\int_\xR \frac{(\partial_{xx}f)^2}{1+(\partial_xf)^2} \dx\dt 
 \leq 2 \lA f_0\rA_{\dot{H}^\tdm}^2.
\end{equation}
\end{proposition}

\subsection{Uniform local in time estimates for arbitrary initial data} 
This is the most delicate step which requires to use the full strength of Proposition~\ref{P:6.1}. 
We will prove the following

\begin{proposition}\label{P:6.3}
Consider a function $\phi$ defined by~\e{n10}, for some admissible weight~$\kappa$ which in addition satisfies
$$
\lim_{r\to+\infty}\kappa(r)=+\infty. 
$$
Then for any $M_0>0$ there exists $T_0>0$ depending on $M_0,\phi$ such that the following properties holds. 
For all $\eps\in (0,1]$ and for all smooth solution $f$ of the approximate 
Muskat equation~\e{main:acp}, if
$$
\lA f_0\rA_{L^2}^2+\blA \D^{\tdm,\phi}f_0\brA_{L^2}^2\le M_0,
$$
then,
$$
\sup_{t\in [0,T_0]}\bigg\{\lA f(t)\rA_{L^2}^2+\blA \D^{\tdm,\phi}f(t)\brA_{L^2}^2\bigg\}
+\int_0^{T_0}\int_\xR \frac{\bla \D^{2,\phi}f(t,x)\bra^2}{1+f_x(t,x)^2} \dx\dt
\le 4M_0.
$$
\end{proposition}
\begin{proof}
We want to estimate the following quantities:
\begin{equation*}
A(t)=\lA f(t)\rA_{L^2}^2+\blA \D^{\tdm,\phi}f(t)\brA_{L^2}^2,\qquad B(t)=\int_\xR \frac{\bla \D^{2,\phi}f(t,x)\bra^2}{1+(\partial_xf(t,x))^2} \dx.
\end{equation*}
As already seen in the proof of Proposition~\ref{P:3.6}, the $L^2$-norm is decreasing:
\be\label{v12}
\lA f(t)\rA_{L^2}^2\le \lA f_0\rA_{L^2}^2.
\ee
On the other hand, notice that we have the obvious bound 
$\lA f\rA_{\dot{H}^\tdm}\le \sqrt{A}$. Hence, the main estimate, 
as given by Proposition~\ref{P:6.1}, implies that
\begin{multline}\label{v14}
\fract A
+ B\\
\lesssim
\kappa\Big(\frac{1}{\lam}\Big)^{-1}\left(1+A\right)^4 B+
\left(1+A\right)^3 B^{\frac{1}{2}} \lA\frac{f_{xx}}{\L{f_x}}\rA_{L^2} +\frac{1}{\sqrt{\lam}}\left(1+A\right)^{\frac{3}{2}} \blA\D^{2,\phi}f\brA_{L^2}.
\end{multline}
Our next task presents itself: we must estimate 
\be\label{v13}
\blA\D^{2,\phi}f\brA_{L^2}, \quad 
\lA\frac{f_{xx}}{\L{f_x}}\rA_{L^2}
\ee
in terms of $A$ and $B$. 

To estimate the three other terms in~\e{v13}, we shall use 
some interpolation inequalities adapted to our problem. 
We begin with the following

\begin{lemma}The following inequalities hold:
\begin{align}
&\lA f_x\rA_{L^\infty} \lesssim \big( 1+A\big)
\log\big(2+B\big)^{\frac{1}{2}},\label{v10}\\
&\blA\D^{2,\phi}f\brA_{L^2}\lesssim \big( 1+A\big)
\log\big(2+B\big)^{\frac{1}{2}}\label{v98}
B^{\frac{1}{2}}.
\end{align}
\end{lemma}
\begin{proof}
Let us prove \e{v10}. Recall the classical interpolation inequality 
\be\label{v1}
\lA f_x\rA_{L^\infty}\lesssim 1+ \lA f\rA_{L^2}+\log\big(2+ \lA f\rA_{\dot H^{2}}\big)^{\frac{1}{2}} 
\lA f\rA_{\dot H^{\frac{3}{2}}}.
\ee
For the sake of completeness, recall that \e{v1} is proved by writing 
\begin{align*}
\lA f_x\rA_{L^\infty}&\leq \int \la \xi\ra \bla\hat f(\xi)\bra\dxi \\
&\lesssim\lA f\rA_{L^2}+\int_{1\leq |\xi|\leq \lambda_0 }\la \xi\ra \bla\hat f(\xi)\bra\dxi
+\int_{|\xi|\geq \lambda_0}\la \xi\ra \bla\hat f(\xi)\bra\dxi\\
&\lesssim\lA f\rA_{L^2}
+(\log \lambda_0)^{\frac{1}{2}} \lA f\rA_{\dot H^{\frac{3}{2}}}+|\lambda_0|^{-\frac{1}{2}} \lA f\rA_{\dot H^{2}},
\end{align*}
and then choosing $\lambda_0=(2+ \lA f\rA_{\dot H^{2}})^2$. 

On the other hand, directly from the definition of 
$B=\int_\xR \frac{\bla \D^{2,\phi}f(x)\bra^2}{1+f_x(x)^2} \dx$, we have
\begin{equation}\label{v2}
\lA f\rA_{\dot H^2}\les \blA\D^{2,\phi}f\brA_{L^2}\leq (1+\lA f_x\rA_{L^\infty})B^{\mez}
\le 1+\lA f_x\rA_{L^\infty}^2+B.
\end{equation}
By combining~\e{v1} and \e{v2}, we obtain
 \begin{align*}
\lA f_x\rA_{L^\infty}\lesssim
\big( 1+\lA f\rA_{ H^{\frac{3}{2}}}\big)\log\big(3+ \lA f_x\rA_{L^\infty}^2+B\big)^{\frac{1}{2}},
\end{align*}
which in turn implies that, 
$$
\lA f_x\rA_{L^\infty} \lesssim \Big( 1+\lA f\rA_{ H^{\frac{3}{2}}}^2\Big)
\log\bigg(2+\int_\xR \frac{\bla \D^{2,\phi}f(x)\bra^2}{1+f_x(x)^2} \dx\bigg)^{\frac{1}{2}}.
$$
This proves~\e{v10}. Also, by plugging the bound \e{v10} in \e{v2} we get~\e{v98}. 
\end{proof}
It follows from \e{v14} and the previous estimate that
\begin{align*}
\fract A
+ B &\lesssim
	\kappa\Big(\frac{1}{\lam}\Big)^{-1}\left(1+A\right)^4 B+
\left(1+A\right)^3 B^{\frac{1}{2}} \lA\frac{f_{xx}}{\L{f_x}}\rA_{L^2}\\
&\quad+\frac{1}{\sqrt{\lam}}\left(1+A\right)^{2}
\log\big(4+B\big)^{\frac{1}{2}} B^{\frac{1}{2}}.
\end{align*}
Choosing $\lam =B^{-\mez}$, 
\begin{align}
\fract A
+ B&\lesssim \left(1+A^{4}\right)\left(\log(4+B)^{\frac{1}{2}}B^{\frac{3}{4}}+B/\kappa(B^{\mez})\right) +(1+A^3)B^{\mez}
\lA\frac{f_{xx}}{\L{f_x}}\rA_{L^2}.\label{Z5}
\end{align}

\begin{lemma}
There exists a non-decreasing function $\mathcal{G}\colon\xR_+\to\xR_+$ such that
\be\label{v200}
\lA\frac{f_{xx}}{\L{f_x}}\rA_{L^2}
\lesssim \left(1+A\right)B^{\mez}/\kappa\left(B^{\frac{1}{50}}\right)+\mathcal{G}\left(A\right).
\ee
\end{lemma}
\begin{proof}
For this proof, the key inequality is given by Proposition~\ref{P:commD}, which implies 
that, for any $\sigma\in (0,1/2)$, 
\be\label{Z1-b}
\lVert \D^{\sigma,\phi}\left(
\frac{g}{\sqrt{1+h^2}}\right)- \frac{1}{\sqrt{1+h^2}}
\D^{\sigma,\phi}g\rVert_{L^2}
\les\blA \D^{\frac{1}{2},\phi}h\brA_{L^2}   
\lVert \frac{g}{\sqrt{1+h^2}}\rVert_{\dot H^{\sigma}}.
\ee
We will work out two applications of the above estimate. 
Firstly we notice that the previous inequality holds for any function $\phi$ 
given by~\e{n10}, for some admissible weight~$\kappa$. In particular, 
it holds for $\phi\equiv 1$. Since $\partial_{xx}=-\D^\frac14 \D^\frac74$, 
the latter result implies that
\begin{equation}
\lA\frac{f_{xx}}{\L{f_x}}\rA_{L^2}\lesssim 
\left(1+\Vert f\Vert_{\dot H^{\frac{3}{2}}}\right)\lA\D^{\frac{1}{4}}
\left(\frac{|D|^{\frac{7}{4}}f}{\L{f_x}}\right)\rA_{L^2}.
\end{equation}
Secondly, we will also use the fact that \e{Z1-b} implies that
\be\label{v201}
\lA\D^{\frac{1}{4},\phi}\left(\frac{|D|^{\frac{7}{4}}f}{\L{f_x}}\right)\rA_{L^2}
\lesssim B^{\frac{1}{2}}+
\blA\D^{\frac{3}{2},\phi}f\brA_{L^2}\lA\D^{\frac{1}{4}}\left(\frac{|D|^{\frac{7}{4}}f}{\L{f_x}}\right)\rA_{L^2}.
\ee
Then, to derive the wanted estimate~\e{v200}, we must compare the two following quantities:
$$
X=\lA\D^{\frac{1}{4}}\left(\frac{|D|^{\frac{7}{4}}f}{\L{f_x}}\right)\rA_{L^2}
\quad\text{and}\quad
X_\phi=\lA\D^{\frac{1}{4},\phi}\left(\frac{|D|^{\frac{7}{4}}f}{\L{f_x}}\right)\rA_{L^2}.
$$
Firstly, we use the elementary inequality: for any $\lambda>0$, 
$$
\blA \D^\uq u\brA_{L^2}^2\lesssim \sqrt{\lambda}\lA u\rA_{L^2}^2+\frac{1}{\kappa(\lambda)^2}\blA \D^{\uq,\phi}u\brA_{L^2}^2,
$$
which is proved as above, by using Plancherel's identity, dividing the integral 
into two parts and using the equivalence $\phi\sim \kappa$. Now the previous estimate implies at once that
\begin{equation}\label{Z2}
X\lesssim  X_\phi^{\frac{1}{100}} \blA\D^{\frac{7}{4}}f\brA_{L^2}+X_\phi/\kappa\left( X_\phi^{\frac{1}{25}}\right)\cdot
\end{equation}
So, it follows from \e{v201} that
$$
X_\phi\lesssim  B^{\frac{1}{2}}+    X_\phi^{\frac{1}{100}}
\blA\D^{\frac{7}{4}}f\brA_{L^2}\blA\D^{\frac{3}{2},\phi}f\brA_{L^2}
+\blA\D^{\frac{3}{2},\phi}f\brA_{L^2} X_\phi/\kappa\left( X_\phi^{\frac{1}{25}}\right).
$$
Since $\blA\D^{\frac{7}{4}}f\brA_{L^2}^2\les \lA f\rA_{\dot{H}^{\tdm}}\lA f\rA_{\dot{H}^2}$, 
by using Young's inequality one infers that
$$
X_\phi\les  B^{\frac{1}{2}}
+\Vert f\Vert_{\dot H^{\frac32}}^{\frac{50}{99}}\Vert f\Vert_{\dot H^2}^{\frac{50}{99}}
\blA\D^{\frac{3}{2},\phi}f\brA_{L^2}^{\frac{100}{99}} +\blA\D^{\frac{3}{2},\phi}f\brA_{L^2}X_\phi/\kappa\big(X_\phi^{\frac{1}{25}}\big).$$
Since $\kappa(r)\to \infty$ as $r\to \infty$, we infer that there exists a non-decreasing function 
$\mathcal{G}\colon\xR_+\to\xR_+$ so that
$$
X_\phi\lesssim B^{\frac{1}{2}}+
\Vert f\Vert_{\dot H^2}^{\frac{20}{33}} +\mathcal{G}\left(\blA\D^{\frac{3}{2},\phi}f\brA_{L^2}\right).
$$
Combining this with \eqref{Z2}, we obtain
\begin{equation}
X\lesssim  B^{\frac{1}{2}}/\kappa\left(B^{\frac{1}{50}}\right)+\Vert f\Vert_{\dot H^2}^{\frac{20}{33}}
+\mathcal{G}\left(\blA\D^{\frac{3}{2},\phi}f\brA_{L^2}\right),
\end{equation}
and hence
\begin{equation}
\lA\frac{f_{xx}}{\L{f_x}}\rA_{L^2}\lesssim   \left(1+\Vert f\Vert_{\dot H^{\frac{3}{2}}}\right)B^{\frac{1}{2}}/
\kappa\left(B^{\frac{1}{50}}\right)
+\Vert f\Vert_{\dot H^2}^{\frac{7}{11}} +\mathcal{G}\left(\blA\D^{\frac{3}{2},\phi}f\brA_{L^2}\right).
\end{equation}
This means that
\begin{align*}
\lA\frac{f_{xx}}{\L{f_x}}\rA_{L^2}\lesssim \left(1+A\right)B^{\mez}/\kappa\left(B^{\frac{1}{50}}\right)+\Vert f\Vert_{\dot H^2}^{\frac{7}{11}} +\mathcal{G}\left(A\right).
\end{align*}
Then the wanted result follows from the bound~\e{v98}. 
This proves the lemma.
\end{proof}

It follows from~\e{Z5} and \e{v200} that
$$
\fract A
+ B\leq C \left(1+A^{4}\right)\left(\log(4+B)^{\frac{1}{2}}B^{\frac{3}{4}}+B/\kappa(B^{\mez})+ B/\kappa(B^{\frac{1}{50}})\right)+\mathcal{G}\left(A\right).
$$
Since $\lim_{r\to+\infty}\kappa(r)=+\infty$, it follows that, up to modifying the function $\mathcal{G}$, we have
$$
\fract A
+ \frac{1}{2}B\leq \mathcal{G}\left(A\right).
$$
Now the proposition follows at once from a continuity argument. 
\end{proof}

\subsection{Contraction estimate}\label{S:3.5}
To conclude the proof we need an estimate for the difference of two solutions. 
This will serve us to prove uniqueness results as well as to prove that the approximate solutions of the Muskat equations defined above form a Cauchy sequence and hence they converge strongly.
\begin{proposition}\label{P:6.final}
Consider a function $\phi$ defined by~\e{n10}, for some admissible weight~$\kappa$ which in addition satisfies
$$
\lim_{r\to+\infty}\kappa(r)=+\infty. 
$$
Then there exists a non-decreasing function $\mathcal{G}\colon\xR_+\to\xR_+$ 
such that the following property holds. 
For any $\eps\in [0,1]$, any $T\in (0,1]$ and any couple of solutions $f_1,f_2$ of the approximate 
Muskat equation~\e{main:acp}, defined in the time interval $[0,T]$  
with initial data $f_{1,0},f_{2,0}$ respectively, 
satisfying 
\begin{equation}\label{Z105}
\sup_{t\in [0,T]}
\blA f_k(t)\brA_{\mathcal{H}^{\tdm,\phi}}^2\leq M<\infty,\quad
\int_0^{T}
\int \frac{(\partial_{xx}f_k)^2}{1+(\partial_xf_k)^2}\dx\dt <\infty \quad (k=1,2),
\end{equation}
the difference $g=f_1-f_2$ is estimated by
\begin{equation}\label{wue}
\fract\Vert g(t) \Vert_{\dot{H}^\mez}\leq \mathcal{G}(M)
\log\left(2+\sum_{k=1}^2
\lA\frac{\partial_{xx}f_k}{\langle\partial_xf_k\rangle}(t)\rA_{L^2}\right)  \Vert g(t)\Vert_{\dot H^{\mez}} .
\end{equation}
\end{proposition}

\begin{proof}[Proof of Proposition~\ref{P:6.final}]
Using the decomposition 
of $\mathcal{T}(f_k)f_k$, we find that the difference $g=f_1-f_2$ satisfies
\begin{align*}
\partial_tg+W(f_1)\partial_x g+\frac{\D g}{1+(\partial_xf_1)^2}
&=  R(f_1)(g)+\left(\mathcal{T}(f_2+g)-\mathcal{T}(f_2)\right)f_2.
\end{align*}
We take 
the $L^2$-scalar product of this equation with $\D g$ to get
\begin{align*}
\frac{1}{2}\fract\Vert g \Vert^{2}_{\dot{H}^\mez}+\int\frac{( \D g)^2}{1+(\partial_x f_{1})^2} \dx
&\leq \left|\langle W(f_1)\partial_x g,|D| g\rangle\right|+\left|\left\langle R(f_1)(g),\D g\right\rangle\right|\\
&\quad+\left|\langle\left(\mathcal{T}(f_2+g)-\mathcal{T}(f_2)\right)f_2,\D g\rangle\right|.
\end{align*}
By combining the previous estimate with the one that we obtain by 
interchanging the role of $f_1$ and $f_2$, we find that
\be\label{claim:unique0}
\fract\Vert g \Vert^{2}_{\dot{H}^\mez}+
\int\left(\frac{( \D g)^2}{1+(\partial_x f_{1})^2}+\frac{( \D g)^2}{1+(\partial_x f_{2})^2}\right) \dx
\le (I)+(II)+(III)
\ee
with
\begin{align*}
(I)&=\left|\langle W(f_1)\partial_x g,|D| g\rangle\right|+\left|\langle W(f_2)\partial_x g,|D| g\rangle\right|,\\
(II)&=\left|\left\langle R(f_1)(g),\D g\right\rangle\right|+\left|\left\langle R(f_2)(g),\D g\right\rangle\right|,\\
(III)&=\left|\langle\left(\mathcal{T}(f_2+g)-\mathcal{T}(f_2)\right)f_2,\D g\rangle\right|+\left|\langle\left(\mathcal{T}(f_1-g)-\mathcal{T}(f_1)\right)f_1,\D g\rangle\right|.
\end{align*}

Set
\begin{align*}
A(t)&=\blA f_1(t)\brA_{\dot{H}^{\frac{3}{2}}}^2+\blA f_2(t)\brA_{\dot {H}^{\frac{3}{2}}}^2,\\
B(t)&=\int \frac{(\partial_{xx}f_1(t,x))^2}{1+(\partial_xf_1(t,x))^2}\dx+
\int \frac{(\partial_{xx}f_2(t,x))^2}{1+(\partial_xf_2(t,x))^2}\dx.
\end{align*}
We claim that for any positive constant $\lambda\in (0,1)$, there holds
\begin{equation}\label{claim:unique}
\begin{aligned}
\la (I)\ra+\la (II)\ra+\la(III)\ra&\lesssim \frac{1}{\sqrt{\lam}}  (1+A)^2\log(2+B)^\mez \Vert g\Vert_{\dot H^{\mez}}\sum_{k=1}^2\lA\frac{\D g}{\L{\partial_xf_k}}\rA_{L^2}\\
&\quad + \frac{1}{\kappa(1/\lambda)}
\left(1+A^4\right)A^{\frac{1}{2}}\sum_{k=1}^2\lA\frac{\D g}{\L{\partial_xf_k}}\rA_{L^2}^2.
\end{aligned}
\end{equation}
Assume that this is true. Since $\lim_{r\to+\infty}\kappa(r)=+\infty$, 
we can choose  $\lambda=\lambda(A,\kappa)$ small enough such that 
$$    \la (I)\ra+\la (II)\ra+\la(III)\ra\leq \mathcal{G}(A)\log(2+B)  \Vert g\Vert_{\dot H^{\mez}}^2+\frac{1}{2}\sum_{k=1}^2\lA\frac{\D g}{\L{\partial_xf_k}}\rA_{L^2}^2.$$
So, by \eqref{claim:unique0}, 
\begin{equation*}
\fract\Vert g \Vert^{2}_{\dot{H}^\mez}\leq \mathcal{G}(A)\log(2+B)  \Vert g\Vert_{\dot H^{\mez}}^2. 
\end{equation*}
Then the wanted estimate~\e{wue} follows at once. 

{\em Step 1: estimates of $(I)$ and $(II)$.} 
It follows from~\e{w1} that
$$
\left|\langle W(f_k)\partial_x g, \D g\rangle\right|
= \frac{1}{2} \left|\Big\langle \left[\mathcal{H}, W(f_k)\right ] \D g,\D g\Big\rangle\right|.
$$
Therefore, we deduce from Proposition~\ref{X1} that for any $\lambda\in (0,1)$
$$
\la (I)\ra 
\les\frac{1}{\sqrt{\lam}}  A^{\frac{1}{2}}
\Vert g\Vert_{\dot H^{\frac{1}{2}}}
\Vert \D g\Vert_{L^2}+ \frac{1}{\kappa(1/\lambda)} \sum_{k=1}^2\left(1+A^{\frac{7}{2}}\right)
A^{\frac{1}{2}}
\lA\frac{g_x}{\L{\partial_xf_k}}\rA_{L^2}
\lA\frac{\D g}{\L{\partial_xf_k}}\rA_{L^2}.
$$
Similarly, it follows from Proposition~\ref{P:5.6} that for any $\lambda\in (0,1)$ 
$$
\la (II)\ra \les\frac{1}{\sqrt{\lam}}  A\Vert g\Vert_{\dot H^{\mez}}\Vert \D g\Vert_{L^2}+\frac{1}{\kappa(1/\lambda)}\sum_{k=1}^2\left(1+A\right)
A^{\frac{1}{2}}\lA\frac{g_x}{\L{\partial_xf_k}}\rA_{L^2}\lA
\frac{\D g}{\L{\partial_xf_k}}\rA_{L^2}.
$$

Notice that 
$$\lA \D g\rA_{L^2}\le \big(1+\lA \partial_xf_k\rA_{L^\infty}^2\big)^\mez
\lA\frac{\D g}{\L{\partial_xf_k}}\rA_{L^2}
\les (1+A)\log(2+B)^\mez\lA\frac{\D g}{\L{\partial_xf_k}}\rA_{L^2},$$
where we have used the interpolation estimate~\e{v98}.\\
On the other hand, it follows from Lemma~\ref{L:Hilbert2} that
$$
\lA\frac{g_x}{\L{\partial_xf_k}}\rA_{L^2}
\les \(1+\lA f_k\rA_{\dot{H}^{\frac{3}{2}}}\)
\lA\frac{\D g}{\L{\partial_xf_k}}\rA_{L^2}.
$$
Consequently,
\begin{equation}
\begin{aligned}
\la (I)\ra+\la (II)\ra &\les \frac{1}{\sqrt{\lam}}
(1+A)^2\log(2+B)^\mez \Vert g\Vert_{\dot H^{\mez}}\sum_{k=1}^2\lA\frac{\D g}{\L{\partial_xf_k}}\rA_{L^2}\\
&\quad + \frac{1}{\kappa(1/\lambda)}
\left(1+A^4\right)A^{\frac{1}{2}}\sum_{k=1}^2\lA\frac{\D g}{\L{\partial_xf_k}}\rA_{L^2}^2.
\end{aligned}
\end{equation}
By using the Young's inequality, we see that $\la (I)\ra+\la (II)\ra$ is estimated by 
the right-hand side of~\e{claim:unique}.

{\em Step 2: estimate of $(III)$.} Dividing and multiplying by $\L{\partial_xf_k}$ 
and then using the Cauchy-Schwarz inequality, we have
\begin{align*}
\la (III)\ra&\le \lA \L{\partial_xf_2}
\(\mathcal{T}(f_2+g)-\mathcal{T}(f_2)\)f_2\rA_{L^2}
\lA\frac{\D g}{\L{\partial_xf_2}}\rA_{L^2}\\
&+\lA \L{\partial_xf_1}\(\mathcal{T}(f_1-g)-\mathcal{T}(f_1)\)f_1\rA_{L^2}
\lA\frac{\D g}{\L{\partial_xf_1}}\rA_{L^2}.
\end{align*}
Consider three functions $h_1,h_2,h$. We want to estimate 
$$
\Lambda\defn \L{\partial_xh_1}\left(\mathcal{T}(h_1)-\mathcal{T}(h_2)\right)h.
$$ 
Recall that
$$
\mathcal{T}(h_k)h = -\frac{1}{\pi}\int_\xR\left(\partial_x\Delta_\alpha h\right)
\frac{\left(\Delta_\alpha h_k\right)^2}{1+\left(\Delta_\alpha h_k\right)^2}\dalpha.
$$
Using the obvious estimate
\begin{equation*}
\left|\frac{x_1^2}{1+x_1^2}-\frac{x_2^2}{1+x_2^2}\right|\lesssim \frac{1}{\sqrt{1+x_1^2}}
|x_1-x_2|,
\end{equation*}
we have
\begin{align*}
\la \Lambda\ra 
&\les \int \frac{\L{\partial_x h_1}}{\L{\Delta_\a h_1}}
\la \Delta_\a(h_1-h_2)\ra\la \Delta_\a h_x\ra
\dalpha\\
&\les \int \(1+|\partial_xh_1-\Delta_{\alpha}h_1|\)
\la \Delta_\a(h_1-h_2)\ra\la \Delta_\a h_x\ra
\dalpha.
\end{align*}
Consequently, using the Cauchy-Schwarz inequality, 
we find that the $L^2$-norm of $\Lambda$ 
is bounded by 
\begin{align*}
&\(\iint \la \Delta_\a(h_1-h_2)\ra^2\dalpha\dx\)^\mez
\(\iint \la \Delta_\a h_x\ra^2\dalpha\dx\)^\mez+\\
&\(\iint \la \Delta_\a h_1-\partial_xh_1\ra^2\frac{\dalpha\dx}{|\alpha|^2}\)^\mez
\bigg(\iint\la\frac{\delta_\alpha (h_1-h_2)}{|\alpha|^{1/4}}\ra^4\frac{\dalpha\dx}{|\alpha|}\bigg)^\frac14
\bigg(\iint\la\frac{\delta_\alpha h_x}{|\alpha|^{1/4}}\ra^4\frac{\dalpha\dx}{|\alpha|}\bigg)^\frac14
\end{align*}
and hence, in view of 
the embedding $\dot{H}^\mez(\xR)\hookrightarrow \dot{F}^{1/4}_{4,4}(\xR)$ (see~\e{FB}) and 
the estimate~\e{an1:x1}, we conclude that
$$
\lA \Lambda\rA_{L^2}\
\les \(1+\lA h_1\rA_{\dot{H}^\tdm}\)\lA h_1-h_2\rA_{\dot{H}^\mez}\lA h\rA_{\dot{H}^\tdm}.
$$
This implies that
\begin{align*}
\la (III)\ra&\le \sum_{k=1}^2 \(1+A^{\frac{1}{2}}\)A^{\frac{1}{2}}
\lA g\rA_{\dot{H}^\mez}\lA \frac{\D g}{\L{\partial_xf_k}}\rA_{L^2}.
\end{align*}
This proves that $(III)$ is estimated by the right-hand side of \e{claim:unique}, which completes the proof.
\end{proof}

\subsection{Enhanced regularity}
Eventually, we need to prove a result which asserts that one can always 
improve the regularity of a solution with values in the 
endpoint Sobolev space in order to gain the existence of a weight.  
\begin{proposition}\label{P:6.1b}
Let $f\in C^0([0,T];H^{3/2}(\xR))$ be a solution of the Muskat equation~\e{a1} such that 
$f_{xx}/\langle f_x\rangle$ belongs to $L^2((0,T)\times\xR)$. 
Assume that initially $f_0$ belongs to $\mathcal{H}^{\tdm,\phi_0}(\xR)$ where $\phi_0$ is given by~$\e{n60b}$ 
for some admissible weight $\kappa_0$. Then there exists $T_0<T$ such that
$$
\sup_{t\in [0,T_0]}
\blA D^{\tdm,\phi_0}f(t)\brA_{L^2}^2
+\int_0^{T_0}\int_\xR \frac{\bla \D^{2,\phi_0}f\bra^2}{1+(\partial_x f)^2} \dx \dt<\infty.
$$
\end{proposition}
\begin{proof}
Consider a function $\phi$ defined by~\e{n10}, for some admissible weight~$\kappa$ 
with $\sup_{r\geq 0}\kappa(r)<\infty$ and such that $\kappa\lesssim \kappa_0$. 
Introduce $f_\varepsilon=f\star \rho_\varepsilon$ where 
$\rho_\varepsilon$ is the standard mollifier in $\mathbb{R}$. 
Proceeding as in the proof of Corollary \ref{C:5.2} (that is by commuting $\D^{1,\phi}$ 
with the Muskat equation and then by taking the $L^2$-scalar 
product with $(\D^{2,\phi}f_\eps)\star \rho_\eps$), 
we deduce from Fatou's lemma that, for any $T_0\in (0,T]$,
\begin{align*}
\sup_{t\in [0,T_0]} \blA D^{\tdm,\phi}f(t)\brA_{L^2}^2
+ 2\int_{0}^{T_0}\int_\xR 
\frac{\bla \D^{2,\phi}f\bra^2}{1+(\partial_x f)^2} \dx \dtau 
\leq  \blA D^{\tdm,\phi}f(0)\brA_{L^2}^2+2\sum_{k=1}^3I_k(T_0),
\end{align*}
where
\begin{align*}
&I_1(T_0)=\limsup_{\varepsilon\to 0} \int_{0}^{T_0} \left|\Big\langle  W(f)\mathcal{H} \D^{2,\phi}f,\D^{2,\phi}f_\varepsilon\star \rho_\varepsilon\Big\rangle \right|\dtau, \\
&
I_2(T_0)=\limsup_{\varepsilon\to 0} \int_{0}^{T_0}
\left|\big\langle R(f)(\D^{1,\phi}f), \D^{2,\phi}f_\varepsilon\star \rho_\varepsilon\big\rangle\right|\dtau ,
\\&
I_3(T_0)=\limsup_{\varepsilon\to 0} \int_{0}^{T_0}
\left|\Big\langle\left[\D^{1,\phi},\mathcal{T}(f)\right]f, \D^{2,\phi}f_\varepsilon\star \rho_\varepsilon\Big\rangle\right|\dtau.
\end{align*}

By combining the estimates given by 
Propositions~\ref{X1},~\ref{P:5.6} and~\ref{P:5.7}, we obtain 
\begin{align*}
& I_2(T_0)+I_3(T_0)+\limsup_{\varepsilon\to 0} \int_{0}^{T_0}
\left|\Big\langle  W(f)\mathcal{H} \D^{2,\phi}f_\varepsilon,\D^{2,\phi}f_\varepsilon\Big\rangle \right|\dtau\\
&\qquad\lesssim
\kappa\Big(\frac{1}{\lam}\Big)^{-1}\int_0^{T_0}\left(1+\blA\D^{\tdm,\phi}f\brA_{L^2}^{7}\right)
\blA\D^{\tdm,\phi}f\brA_{L^2}\lA\frac{|D|^{2,\phi}f}{\L{f_x}}\rA_{L^2}^2 \dtau
\\
&\qquad\quad+
\int_0^{T_0}\left( 1+\blA\D^{\frac{3}{2},\phi} f\brA_{L^2}^{5}\right)
\brA\D^{\frac{3}{2},\phi}f\brA_{L^2} \lA\frac{f_{xx}}{\L{f_x}}\rA_{L^2}
\lA\frac{|D|^{2,\phi}f}{\L{f_x}}\rA_{L^2}\dtau \\
&\qquad\quad+\frac{1}{\sqrt{\lam}}\int_0^{T_0}
\left(1+\blA\D^{\frac{3}{2},\phi}f\brA_{L^2}\right) \blA\D^{\frac{3}{2},\phi}f\brA_{L^2}^2
\blA\D^{2,\phi}f\brA_{L^2}\dtau,
\end{align*}
for any $\lambda\in (0,1).$\\
Now, we will prove that 
\begin{align}\label{Z1b}
\limsup_{\varepsilon\to 0} \int_{0}^{T_0}
\left|\Big\langle(W(f)\mathcal{H} g)\star\rho_\varepsilon-W(f)\mathcal{H} g\star\rho_\varepsilon \, , \, g\star\rho_\varepsilon\Big\rangle \right|\dtau=0,
\end{align}
with $g=\D^{2,\phi}f$. 
It is enough to show that 
\begin{align}
\limsup_{\varepsilon\to 0}\int_{0}^{T_0} \blA\langle f_x\rangle 
\big( (W(f)\mathcal{H} g)\star\rho_\varepsilon-W(f)\mathcal{H} g\star\rho_\varepsilon\big)\brA_{L^2}^2\dtau=0.
\end{align}
To see this, we set $\tilde{g}=|\mathcal{H}(g)|/\langle f_x\rangle$ 
and repeat the computations used to 
derive~\eqref{Z1-W}. This gives
\begin{align*}
&\int_{0}^{T_0}\blA 
\langle f_x\rangle
\big((W(f)\mathcal{H} g)\star\rho_\varepsilon-W(f)\mathcal{H} g\star\rho_\varepsilon\big)\brA_{L^2}^2\dtau\\
&\lesssim  \int_{0}^{T_0}\int\left(\fint_{|y|\leq \varepsilon}
\langle f_x(x)\rangle \langle f_x(x-y)\rangle|W(f)(x)- W(f)(x-y)|\tilde{g}(x-y)\dy\right)^2 \dx\dtau\\
&\lesssim
\int_0^{T_0}\left(\iint_{|x-y|\leq \varepsilon}
\langle f_x(x)\rangle^5 \langle f_x(y)\rangle^5 |W(f)(x)-W(f)(y)|^5 \frac{\dy\dx}{|x-y|^{2}}\right)^{\frac{2}{5}} \dtau
\Vert\tilde{g}\Vert_{L^2((0,T)\times\mathbb{R})}^{2}.
\end{align*}
The next step is a reprise of one argument in the proof of Proposition~\ref{X1}. 
More precisely, it follows from the proof of \eqref{X12} that 
\begin{equation*}
\lim_{\varepsilon\to 0}\int_0^{T_0} \left(\iint_{|x-y|\leq \varepsilon}
\langle f_x(x)\rangle^5 \langle f_x(y)\rangle^5
\la W(f)(x)-W(f)(y)\ra^5 \frac{\dy\dx}{|x-y|^{2}}\right)^{\frac{2}{5}}\dtau=0.
\end{equation*}
This implies \eqref{Z1b}.

Now, we set 
\begin{equation*}
A(T_0)=\sup_{t\in [0,T_0]} \Big\{ \blA D^{\tdm,\phi}f(t)\brA_{L^2}^2+\lA f(t)\rA_{L^2}^2\Big\},
\qquad
B(T_0)=\int_{0}^{T_0}\lA\frac{|D|^{2,\phi}f}{\L{f_x}}\rA_{L^2}^2 \dtau.
\end{equation*}
We previous estimates imply that, for any $\lambda\in (0,1)$, 
\begin{align*}
 A(T_0 )+ B(T_0)
 &\lesssim  \blA D^{\tdm,\phi}f(0)\brA_{L^2}^2+\blA f(0)\brA_{L^2}^2+
 \left(1+A(T_0 )^{4}\right) \kappa\Big(\frac{1}{\lam}\Big)^{-1} B(T_0)
\\\nonumber
&\quad+ \left( 1+A(T_0 )^{3}\right)
\int_0^{T_0} \lA\frac{f_{xx}}{\L{f_x}}\rA_{L^2}\lA\frac{|D|^{2,\phi}f}{\L{f_x}}\rA_{L^2}\dtau \\\nonumber
&\quad+\frac{1}{\sqrt{\lam}}\left(1+A(T_0)^{2}\right) \int_0^{T_0}\blA\D^{2,\phi}f\brA_{L^2}\dtau.
\end{align*}
On the other hand, by \eqref{v98} and \eqref{v200}, for any $\delta\in (0,1)$, we have 
a non-decreasing function $\mathcal{G}_{\kappa_0}\colon \xR_+\to\xR_+$ (depending only on $\kappa_0$) such that
\begin{align*}
&\lA\frac{f_{xx}}{\L{f_x}}\rA_{L^2}
\leq  \delta \lA\frac{|D|^{2,\phi}f}{\L{f_x}}\rA_{L^2}+\mathcal{G}_{\kappa_0}
\left(A(T_0 )+\frac{1}{\delta}\right),\\
&\blA\D^{2,\phi}f\brA_{L^2}\leq \delta \lA\frac{|D|^{2,\phi}f}{\L{f_x}}\rA_{L^2}^2
+\mathcal{G}_{\kappa_0}\left(A(T_0 )+\frac{1}{\delta}\right).
\end{align*}
It follows that, up to modifying the function $\mathcal{G}_{\kappa_0}$, 
\begin{align*}
& A(T_0 )
+ B(T_0) \leq   C\left(\blA D^{\tdm,\kappa_0}f(0)\brA_{L^2}^2
+\blA f(0)\brA_{L^2}^2\right)+T_0\mathcal{G}_{\kappa_0}\left(A(T_0 )+\frac{1}{\delta}+\frac{1}{\lambda}\right)\\
&+ \left( C\left(1+A(T_0 )^{4}\right)
 \kappa\Big(\frac{1}{\lam}\Big)^{-1}+C\delta \left( 1+A(T_0 )^{3}\right)
 +C\frac{\delta}{\sqrt{\lam}}\left(1+A(T_0 )^{2}\right)\right) B(T_0,\phi ). 
\end{align*}
Set 
\begin{equation*}
M_0=2C\left(\blA D^{\tdm,\kappa_0}f(0)\brA_{L^2}^2+\blA f(0)\brA_{L^2}^2\right).
\end{equation*}
By successively choosing $\lambda=\lambda(A(T_0,\phi))\in (0,1)$ small enough, 
and then $\delta=\delta(A(T_0,\phi))\in (0,1)$ small enough, one can 
guarantee that 
\begin{align*}
C\left(1+A(T_0 )^{4}\right)
 \kappa\Big(\frac{1}{\lam}\Big)^{-1}+C\delta \left( 1+A(T_0 )^{3}\right)
 +C\frac{\delta}{\sqrt{\lam}}\left(1+A(T_0 )^{2}\right)\leq \frac{1}{2}\cdot
\end{align*}
Now, we deduce that 
\begin{align*}
& A(T_0)
+\frac{1}{2} B(T_0 )\leq   \frac{M_0}{2}+T_0\mathcal{G}_{\kappa_0}\left(A(T_0)\right).
\end{align*}
Then we fix $T_0>0$ such that 
\begin{equation*}
T_0\mathcal{G}_{\kappa_0}\left(M_0\right)= \frac{M_0}{2}.
\end{equation*}
With this special choice, it follows that 
\begin{align*}
& A(T_0 )
+\frac{1}{2} B(T_0 )\leq M_0.
\end{align*}
Now, letting $\kappa\uparrow \kappa_0$, we find that 
$$
\sup_{t\in [0,T_0]} \blA D^{\tdm,\kappa_0}f(t)\brA_{L^2}^2
+ \int_{0}^{T_0}\int_\xR \frac{\bla \D^{2,\kappa_0}f\bra^2}{1+(\partial_x f)^2} \dx \dtau<\infty.
$$
Since $\phi_0\sim\kappa_0$, this completes the proof.
\end{proof}

The following is as a consequence of Proposition \ref{P:6.1b}.
\begin{corollary}
\label{P:6.1c}
Let $(f_{0,n})_n$ be  equi-integrable in $H^{\frac{3}{2}}(\mathbb{R}).$  Then, there exists a unique solution $f_n$ of the Muskat equation ~\e{a1} with initial data $f_{0,n}$ in $[0,T_0]$ for some $T_0>0$ and all $n$,  such that
$$\sup_{n\in \mathbb{N}}\left(\sup_{t\in [0,T_0]}
||f_n(t)||_{H^{\frac{3}{2}}}^2
+\int_0^{T_0}\int_\xR \frac{\bla \partial_{xx}f_n\bra^2}{1+(\partial_x f)^2} \dx \dt\right)<\infty,
$$
and 
$$
\sup_{n\in \mathbb{N}}\left(\sup_{t\in [0,T_0]}
||(f_n-f_n\star \rho_\varepsilon)(t)||_{H^{\frac{3}{2}}}^2
+\int_0^{T_0}\int_\xR \frac{\bla \partial_{xx}(f_n-f_n\star \rho_\varepsilon)\bra^2}{1+(\partial_x f_n)^2} \dx \dt\right)\underset{\varepsilon\to 0}\to 0.
$$
\end{corollary}

This proof is simple and left to the reader. 
\subsection{End of the proof}
We are now in position to complete the proof 
of Theorems~\ref{Theorem} and~\ref{Theorem2}. These two theorems contain two different results: 
an existence result and a uniqueness result. We discuss these two parts separately.

{\em Uniqueness.} Firstly, 
we notice that the estimate~\e{wue} in the 
Proposition~\ref{P:6.final} applies also for the Muskat equation~\e{a1} 
(that is for the approximate equation 
with $\eps=0$). This proves immediately that the solutions of the Muskat 
equation are unique in 
the space of functions satisfying~\e{Z105} (the fact that the 
size of the time interval is assumed to be less than $1$ is not a 
limitation since one can always 
divide an arbitrary interval into intervals of length less than one). 
This proves the uniqueness part of 
Theorem~\ref{Theorem2}. In view of Proposition~\ref{P:6.1b}, this 
also proves the uniqueness part 
of Theorem~\ref{Theorem}.

{\em Existence.} On the other hand, one can apply Proposition~\ref{P:6.final} to infer the existence of a solution to the Muskat equation. 
To do so, consider the approximate solutions of the approximate Muskat equation~\e{main:acp}, whose existence 
is given by Proposition~\ref{P:3.6}. In the previous paragraphs, we have proved uniform bounds for these 
solutions (see Proposition~\ref{C:6.2} and Proposition~\ref{P:6.3}). Then Proposition~\ref{P:6.final}
implies that these approximate solutions 
forms a Cauchy sequence 
in the space of bounded functions with values in $\dot{H}^{\mez}(\xR)$. 
Now, we claim that one has the following inequality
$$
\lA u\rA_{\dot{H}^\tdm}^2\les \lA u\rA_{\dot{H}^\mez}\blA \D^{\tdm,\phi}u\brA_{L^2}
+\phi\bigg(\bigg(\frac{\lA u\rA_{\dot{H}^\tdm}}{\lA u\rA_{\dot{H}^\mez}}\bigg)^\mez\bigg)^{-2}
\blA \D^{\tdm,\phi}u\brA_{L^2}^2.
$$
This is proved by writing
$$
\lA u\rA_{\dot{H}^\tdm}^2\les 
\int_{\la\xi\ra\le R}R^{2}\la \xi\ra\bla\hat{u}(\xi)\bra^2\dxi
+\int_{\la\xi\ra\ge R}\la \xi\ra^3\frac{\phi(|\xi|)^2}{\phi(R)^2}\bla\hat{u}(\xi)\bra^2\dxi,
$$
and then by choosing
$$
R=\(\frac{\blA \D^{\tdm,\phi}u\brA_{L^2}}{\lA u\rA_{\dot{H}^\mez}}\)^\mez
\gtrsim \(\frac{\lA u\rA_{\dot{H}^\tdm}}{\lA u\rA_{\dot{H}^\mez}}\)^\mez.
$$
We deduce from the previous inequality that the sequence of approximate solutions is also 
a Cauchy sequence in $C^0([0,T_0];H^{\tdm}(\xR))$. 
Now, in view of statement $\ref{Prop:low2})$ 
in Proposition~\ref{P:continuity} 
one can pass to the limit in the nonlinearity 
to prove that the limit of the Cauchy sequence satisfies 
the Muskat equation. 
This proves the existence part of Theorem~\ref{Theorem}. 
In fact, in view of Proposition~\ref{P:6.1b}, this also implies the existence part of 
Theorem~\ref{Theorem2}. 

Theorem~\ref{Theorem} and Theorem~\ref{Theorem2} are therefore proved.

\section*{Acknowledgments} 

\noindent The authors want to thank Beno{\^\i}t Pausader for his stimulating 
comments as well as for discussions about the critical problem for the Muskat equation.

\noindent  T.A.\ acknowledges the SingFlows project (grant ANR-18-CE40-0027) 
of the French National Research Agency (ANR).  Q-H.N.\ 
is  supported  by the ShanghaiTech University startup fund.


\vfill
\begin{flushleft}
\textbf{Thomas Alazard}\\
Universit{\'e} Paris-Saclay, ENS Paris-Saclay, CNRS,\\
Centre Borelli UMR9010, avenue des Sciences, 
F-91190 Gif-sur-Yvette\\
France.

\vspace{1cm}

\textbf{Quoc-Hung Nguyen}\\
ShanghaiTech University, \\
393 Middle Huaxia Road, Pudong,\\
Shanghai, 201210,\\
China

\end{flushleft}

\end{document}